\documentclass[12pt,a4paper,reqno]{amsart}
\usepackage{amsmath}
\usepackage{amsfonts, color,amssymb,fancyhdr}
\usepackage[colorlinks,linkcolor=blue, citecolor=red,anchorcolor=blue, pagebackref=true
]{hyperref}
 \usepackage[nobysame]{amsrefs}
 \def\MR#1{}                
\def\@bibmrnumber#1{}       
\def\@bib@mrreview#1{}      
\def\@bib@mathreviews#1{}   
 
\usepackage{amsthm}
\usepackage{tikz}
\usepackage{graphicx}
\usepackage{multirow}
\usepackage{booktabs}

\usepackage{pdflscape}
\usepackage{geometry}
\usepackage{epic}
\usepackage{threeparttable}
\usepackage{url}

\usepackage{array} 
\usepackage{tabularx} 
\usepackage{booktabs} 

\def\Box{\vcenter{\vbox{\hrule\hbox{\vrule
     \vbox to 8.8pt{\hbox to 10pt{}\vfill}\vrule}\hrule}}}

\newcommand{\PGaL}{\textup{P}\Gamma \textup{L}}
\newcommand{\Tr}{\textup{Tr}}

\newcommand{\ra}{\rangle}

\newcommand{\bfc}{\boldsymbol{c}}

\newcommand{\bfv}{\boldsymbol{v}}

\newcommand{\bfx}{\boldsymbol{x}}
\newcommand{\la}{\langle}
\newcommand{\F}{{\mathbb F}}
\newcommand{\bbF}{{\mathbb F}}
\newcommand{\bbZ}{{\mathbb Z}}

\newcommand{\Z}{{\mathbb Z}}

\newcommand{\cC}{{\mathcal C}}

\newcommand{\cL}{{\mathcal L}}
\newcommand{\cLst}{{\mathcal L_{\rm st}}}
\newcommand{\cU}{{\mathcal U}}

\newcommand{\cS}{{\mathcal S}}

\newcommand{\PGL}{\textup{PGL}}

\newcommand{\GL}{\textup{GL}}
\newcommand{\GO}{\textup{GO}}
\newcommand{\SU}{\textup{SU}}
\newcommand{\SL}{\textup{SL}}
\newcommand{\Sp}{\textup{Sp}}

\newcommand{\AG}{\textup{AG}}
\newcommand{\PSL}{\textup{PSL}}

\newcommand{\Aut}{\textup{Aut}}
\newcommand{\PAut}{\textup{PAut}}

\newcommand{\tr}{\textup{Tr}}

\newcommand{\ord}{\textup{ord}}

\newtheorem{thm}{Theorem}
\newtheorem{problem}[thm]{Problem}
\newtheorem{lemma}[thm]{Lemma}
\newtheorem{corollary}[thm]{Corollary}
\newtheorem{proposition}[thm]{Proposition}
\newtheorem{definition}[thm]{Definition}

\newtheorem{example}[thm]{Example}\theoremstyle{definition}
\numberwithin{equation}{section}
\numberwithin{thm}{section}

\theoremstyle{definition}
\newtheorem{remark}[thm]{Remark}\theoremstyle{definition}
\newtheorem{fact}[thm]{Fact}
\newtheorem{construction}[thm]{Construction}

\newcommand{\Tm}[1]{Theorem~\protect\ref{#1}}
\newcommand{\Le}[1]{Lemma~\protect\ref{#1}}

\newcommand{\Sec}[1]{Section~\protect\ref{#1}}

\newcommand{\ben}{\begin{enumerate}}
\newcommand{\een}{\end{enumerate}}
\newcommand{\bit}{\begin{itemize}}
\newcommand{\eit}{\end{itemize}}

\DeclareMathOperator{\lcm}{lcm}

\def\mymedskip{\vskip\medskipamount}
\def\mymedbreak{\par \ifdim\lastskip<\medskipamount
  \removelastskip \penalty-100 \mymedskip \fi}
\def\myaftermedspace{\par \ifdim\lastskip<\medskipamount
  \removelastskip \penalty55\mymedskip\fi}
\newcommand{\eop}{{\unskip\nobreak\hfil\penalty50
          \hskip2em\hbox{}\nobreak\hfil$\Box$
          \parfillskip=0pt \finalhyphendemerits=0 \par}}

\newcommand{\gs}{\sigma}

\newcommand{\ga}{\alpha}
\newcommand{\gb}{\beta}

\newcommand{\gD}{\Delta}

\newcommand{\btm}[1]{\begin{thm} \label{#1}}
\newcommand{\etm}{\end{thm}}
\newcommand{\btmn}[2]{\begin{teor}[#1] \label{#2}}
\newcommand{\etmn}{\end{teor}}
\newcommand{\ble}[1]{\begin{lemma} \label{#1}}
\newcommand{\ele}{\end{lemma}}
\newcommand{\bLe}[1]{\begin{Lemma} \label{#1}}
\newcommand{\eLe}{\end{Lemma}}
\newcommand{\blen}[2]{\begin{lem}[#1] \label{#2}}
\newcommand{\elen}{\end{lem}}
\newcommand{\bpn}[1]{\begin{prop} \label{#1}}
\newcommand{\epn}{\end{prop}}
\newcommand{\bex}[1]{\begin{examp} \label{#1}}
\newcommand{\eex}{\eop\end{examp}}
\newcommand{\bde}[1]{\begin{defi} \label{#1}}
\newcommand{\ede}{\end{defi}}
\newcommand{\bco}[1]{\begin{corollary} \label{#1}}
\newcommand{\eco}{\end{corollary}}
\newcommand{\bcon}[1]{\begin{con} \label{#1}}
\newcommand{\econ}{\end{con}}
\newcommand{\bfa}[1]{\begin{fact} \label{#1}}
\newcommand{\efa}{\end{fact}}
\newcommand{\bpr}[1]{\begin{problem} \label{#1}}
\newcommand{\epr}{\end{problem}}
\newcommand{\bprnn}[1]{\begin{problemnn} \label{#1}}
\newcommand{\eprnn}{\end{problemnn}}
\newcommand{\bprn}[2]{\begin{problem}[#1] \label{#2}}
\newcommand{\eprn}{\end{problem}}
\newcommand{\bexer}[1]{\begin{exercise} \label{#1}}
\newcommand{\eexer}{\end{exercise}}
\newcommand{\bre}[1]{\begin{remark} \label{#1}}
\newcommand{\ere}{\end{remark}}

\newcommand{\beq}{\begin{equation}}
\newcommand{\eeq}{\end{equation}}
\newcommand{\beql}[1]{\begin{equation} \label{#1}}
\newcommand{\eeql}{\end{equation}}
\newcommand{\beqa}{\begin{eqnarray*}}
\newcommand{\eeqa}{\end{eqnarray*}}
\newcommand{\beqal}[1]{\begin{eqnarray} \label{#1}}
\newcommand{\eeqal}{\end{eqnarray}}
\newcommand{\beqan}{\begin{eqnarray}}
\newcommand{\eeqan}{\end{eqnarray}}
\newcommand{\bpf}{\begin{proof}}
\newcommand{\epf}{\end{proof}}

\begin{document}

\title{The permutation automorphism groups of irreducible cyclic codes}
\author{Tao~Feng, Henk D.L.~Hollmann, Weicong~Li, Qing~Xiang$^\star$}

\date{\today}
\thanks{$^\star$ Corresponding author.}

\address{Tao~Feng, School of Mathematical Sciences,   Zhejiang University, Hangzhou, China.}
\email{tfeng@zju.edu.cn}
\address{Henk D.L.~Hollmann, Institute of Computer Science, University of Tartu,  51009 Tartu, Estonia.}
\email{henk.d.l.hollmann@ut.ee}
\address{Weicong~Li, Department of Mathematics, School of Sciences, Great Bay University, Dongguan, China.}
\email{liweicong@gbu.edu.cn}
\address{Qing~Xiang,  Department of Mathematics and Shenzhen International Center for Mathematics, Southern University of Science and Technology, Shenzhen, China.}
\email{xiangq@sustech.edu.cn}

\begin{abstract}
 The study of permutation automorphism groups of cyclic codes is a central topic in algebraic coding theory. A cyclic code over $\mathbb{F}_q$ is called irreducible if its check polynomial is irreducible over $\mathbb{F}_q$, and it is non-degenerate if it is not periodic. Such a non-degenerate code is standard if its permutation automorphism group is equal to the group generated by the cyclic shift and the Frobenius automorphism, and non-standard otherwise. In this paper, we give a complete classification of all non-standard non-degenerate irreducible cyclic codes, using the classification of finite simple groups. Our result shows that, apart from a small number of explicit exceptional families and their descendants under certain secondary constructions, every non-degenerate irreducible cyclic code is standard, and up to four explicit exceptions, every nonzero degenerate cyclic code is non-standard.

This classification has several consequences. First, it yields a general description of non-standard linear recurring sequence subgroups, extending the earlier work of Brison and Nogueira; secondly, it establishes the Schmidt-White conjecture for all non-standard non-degenerate irreducible cyclic codes. Moreover, our results provide strong evidence in support of the conjecture of Berger and Charpin that almost all cyclic codes are standard.

\medskip
\noindent{{\it Keywords\/}: irreducible cyclic codes, permutation group, non-standard, Aschbacher's theorem, linear recurring sequence subgroups.}

\smallskip

\noindent {{\it MSC (2020)\/}: 94B15, 20B25, 05E20, 05E18}
\end{abstract}
	\maketitle
\tableofcontents
\section{Introduction}

Let $p$ be a prime and $q=p^f$ for some positive integer $f$. Let $\F_q$ be the finite field of size $q$, and let $n$ be a positive integer. A $q$-ary linear code~$C$ of length $n$ is an $\F_q$-linear subspace of  $\F_{q}^n$; its {\em dimension\/}, denoted by~$\dim(C)$, is  simply the dimension of this subspace. The vectors in $C$ are called the {\it codewords} of $C$. In this paper, we are dealing with cyclic codes, so as usual, we enumerate the coordinate positions with $ 0, 1, \ldots, n-1$ and identify them with the elements of $\mathbb {Z} _ n$, the (additive) group of integers modulo~$n$. The symmetric group $S_n$, considered as the group of all permutations on~$\bbZ_n$, acts on the vectors of~$\F_{q}^n$ by permuting the coordinate positions, that is, for all $\pi\in S_n$ and for all $\bfv=(v_0,v_1,\ldots, v_{n-1})\in \F_q^n$, we have ${\pi(\bfv)}=(v_{\pi^{-1} (0)},\, \ldots,\, v_{\pi^{-1} (n-1)})$. The set-wise stabilizer of~$C$, $\mathrm{Stab}_{S_n}(C)$, under such an action is called the {\it permutation automorphism group} of~$C$, denoted by $\PAut(C)$. We say that the linear code $C$ is a \textit{cyclic} code if $\PAut(C)$ contains the cyclic shift of order~$n$.

Let $C$ be a $q$-ary linear code of length~$n$. A codeword $\bfc=(c_0,\ldots,c_{n-1})\in C$ is naturally identified with the polynomial $\bfc(x)=c_0+c_1x+\cdots+c_{n-1}x^{n-1}$ in the quotient ring $\F_q[x]/(x^n-1)$, and a $q$-ary code of length~$n$ is just a subset of polynomials in $\F_q[x]/(x^n-1)$.  As a result, a linear code is identified with a linear subspace of $\F_q[x]/(x^n-1)$, and a cyclic code is an ideal of $\F_q[x]/(x^n-1)$.

In this paper,  when dealing with cyclic codes, we will assume that $\gcd(n,q)=1$; note that, as a consequence, $x^n-1$ has no multiple zeros. Observe that any ideal of $\F_q[x]/(x^n-1)$ is principal, so a cyclic code~$C$ is generated by a monic polynomial $g(x)$ dividing $x^n-1$, which is called the {\it generator polynomial} of $C$; in that case, we call $h(x):=(x^n-1)/g(x)$ the {\it check} polynomial of $C$. We refer to the zeros of~$g(x)$ in the algebraic closure $\overline{\F_q}$ of $\bbF_q$ as the {\em zeros\/} of~$C$. We will write $C_{n,h,q}$ to denote the $q$-ary cyclic code of length~$n$ with check polynomial~$h(x)$; note that the dimension of this code is equal to the degree of~$h(x)$. Moreover, we will write $C_{h,q}$ to denote the code $C_{n',h,q}$ with $n'$ equal to the order of~$h(x)$, where the order of $h(x)$ is the smallest positive integer $k$ such that $h(x)\mid (x^k-1)$.
The cyclic code $C_{n,h,q}$ is called {\em degenerate\/} if the order~$n'$ of~$h(x)$ is smaller than $n$ and {\em non-degenerate\/} otherwise \cite{Mac1977ErrorCorrectcode}*{p. 223}. In the case when $C_{n,h,q}$ is degenerate, we have that $n'$ divides $n$, and $C_{n,h,q}$ is {\em periodic\/}, that is, the codewords are the words of the form $(\bfc', \bfc', \ldots, \bfc')$ with $\bfc'$ a codeword of the (non-degenerate) code $C_{h,q}$ (see, e.g., \cite{Hollmann2023}*{Lemma 5.22}); 
 it follows that $C_{n,h,q}$ and the {shorter code}~$C_{h,q}$ have the same number of nonzero Hamming weights (see \cite{Wolfmann}*{Corollary 11}).  {Conversely, if every codeword has a common period $d<n$, then $h(x) \mid (x^d-1)$, so $C_{n,h,q}$ is degenerate.}

A $q$-ary cyclic code $C$ is called {\it irreducible} if $C$ is a minimal ideal in the quotient ring $\F_q[x]/(x^n-1)$; equivalently,  $C$ is irreducible if the check polynomial $h(x)$ is irreducible over~$\F_q$. For further details on linear codes and cyclic codes, please refer to \cites{Huffman1998Handbook, Mac1977ErrorCorrectcode, GTMcoding1992vanLint, Rot06}.

There are close connections between codes and groups, cf. \cite{Huffman1998Handbook}. The automorphism group of a code provides information about its structure and can help design efficient decoding algorithms, cf. \cite{Huffman1998Handbook}*{Section 8} and \cite{Geiselhart2021}. The problem of determining the class of finite groups that can arise as the permutation automorphism group of cyclic codes has attracted much attention, cf. \cites{Bienert2010AutCcode, Guenda, Guenda2013, Guenda2017, MaYan2024}. In general, it is a difficult problem to determine the automorphism group of any specific cyclic code or extended cyclic code, and ``the best way seems to be the study of special classes," to quote \cite{OpenPcode}*{Section 3.5}. There has been extensive research in this direction, cf. \cites{Knapp1980,Berger1993AutRMcode,Huffman1995AutQRcode}. Most notably, Berger and Charpin developed a method to effectively determine the permutation groups of affine invariant extended cyclic codes, cf. \cites{Berger1996AutECcode,Berger1999AutBCH}.

The permutation automorphism group $\PAut(C)$ of a degenerate cyclic code $C=C_{n,h,q}$ can be expressed in terms of the non-degenerate code $C'=C_{h,q}$ as $\PAut(C)=S_k\wr \PAut(C')$ (wreath product, a semidirect product $S_k^{n'}\rtimes \PAut(C')$, see \Sec{SSL_group}), where $n'$ denotes the order of~$h(x)$ and $k:=n/n'$ (see, e.g., \cite{Hollmann2023}*{Corollary 5.24}). So when studying the permutation automorphisms of cyclic codes, we can restrict our attention to {\em non-degenerate\/} cyclic codes.
It is widely believed that in most cases, the permutation automorphism group $\PAut(C)$ of a $q$-ary non-degenerate cyclic code~$C$ is very small. To make this statement precise, we need some notation. Let $\AG(n)$ denote the group of affine transformations on~$\bbZ_n$. For a $q$-ary cyclic code $C$ of length~$n$, define $\PAut_{\rm st}(C):=\PAut(C)\cap \AG(n)$. It is well known that $\PAut_{\rm st}(C)$ contains the cyclic shift $\gs: x\mapsto x+1$ and the Frobenius map $\psi: x\mapsto qx$ \cite{GTMcoding1992vanLint}*{Section 6.1}, \cite{PHB98}*{Corollary 5.19}. In addition, if $C$ is non-degenerate and irreducible, then $\PAut_{\rm st}(C)=\la \gs, \psi\ra$, the subgroup of~$\AG(n)$ generated by~$\gs$ and~$\psi$ \cite{Hollmann2023}*{Corollary 5.14} (for more details, see Section~\ref{subsec_standard}; see also \cite{Hollmann2023}*{Theorem 5.28}).
We will refer to a permutation automorphism of a cyclic code~$C$ as {\em standard\/} if it is contained in~$\PAut_{\rm st}(C)$, and as {\em non-standard\/} otherwise; in addition, we will refer to a non-degenerate cyclic code $C$ as {\em standard\/} if every permutation automorphism of~$C$ is standard, and as {\em non-standard\/} otherwise.

So if $C$ is a non-degenerate irreducible cyclic code, then $\PAut(C)$ always contains the group of standard permutation automorphisms generated by the cyclic shift~$\gs$ and the Frobenius map~$\psi$; it is believed that in almost all cases, $\PAut(C)$ is not larger; that is, in general, $C$ is standard. In this paper, we give a complete classification of all non-standard non-degenerate irreducible cyclic codes by using the classification of finite simple groups (CFSG). To be precise, our main result is the following.
\begin{thm}\label{thm_mainCode}
Let $C$ be a non-degenerate irreducible cyclic code of length~$n$ over~a finite field of characteristic $p$, where $\gcd(n,p)=1$.   If $C$ is non-standard, then either $C$ is one of the codes in Examples~\ref{exa_repetition}--\ref{exa_imp}, or $C$ can be obtained from one of the codes in Examples~\ref{exa_repetition}--\ref{exa_imp} by repeated uses of Constructions \ref{exa_ext}--\ref{exa_product},  where the underlying finite fields of these cyclic codes are allowed to vary, but always have characteristic $p$.
\end{thm}
 We remark that, in general, a {nonzero} degenerate cyclic code $C$ of length~$n$ satisfies $\PAut(C)\nsubseteq\AG(n)$, with exactly 
 four (known) exceptions, cf. Lemma \ref{lem_deg}. We further remark that we also believe that if $C$ is an arbitrary non-degenerate cyclic code of length~$n$, even without the condition that $\gcd(n,q)=1$, then in general $C$ is standard, that is, $\PAut(C)\subseteq\AG(n)$; of course it will be much harder (if not impossible) to classify all the exceptions in this more general case.

In the remainder of this section, we will focus on $q$-ary non-degenerate irreducible cyclic codes of length~$n$, so with a check polynomial of order~$n$.  As was observed in \cite{Hollmann2023}*{Proposition 5.19}, all $q$-ary non-degenerate irreducible cyclic codes of length~$n$ are equivalent under a permutation of the coordinates. Hence, being non-standard depends only on the pair $(n,q)$. So in this paper, we will sometimes refer to a pair $(n,q)$ as being non-standard if some (hence every) $q$-ary non-degenerate irreducible cyclic code of length~$n$ is non-standard. We remark that this is the approach taken in~\cites{Hollmann2022ISIT,Hollmann2023}; note that, somewhat confusingly, these papers use the notion of ``irreducible cyclic code'' in the sense of ``non-degenerate irreducible cyclic code''; that is,  when speaking of irreducible cyclic codes, the author of  \cites{Hollmann2022ISIT,Hollmann2023} really meant non-degenerate irreducible cyclic codes.

Let $\cU_{n,q}$ be {the} multiplicative subgroup of order $n$ in some extension field $\F_{q^m}$ of $\F_q$, where $\gcd(n,q)=1$ and $n\mid (q^m-1)$. Note that $\cU_{n,q}$ is cyclic. Suppose the elements of $\cU_{n,q}$ can be arranged in a non-cyclic way such that the corresponding periodic sequence of period $n$ is a linear recurring sequence over $\F_q$ and has an irreducible characteristic polynomial. In that case, we call $\cU_{n,q}$ a non-standard linear recurring sequence subgroup of $\F_{q^m}$. Brison and Nogueira investigated such subgroups in a series of papers \cites{Brison2003,Brison2008,Brison2009,Brison2010,Brison2014,Brison2021}. By \cite{Hollmann2023}*{Corollary 4.13}, a non-degenerate irreducible cyclic code of length $n$ over $\F_q$ is non-standard if and only if $\cU_{n,q}$ is a non-standard linear recurring sequence subgroup, which occurs if and only if the pair $(n,q)$ is non-standard. Here, note that the check polynomials of non-degenerate $q$-ary irreducible cyclic codes of length~$n$ are precisely the minimal polynomials of the generators of~$\cU_{n,q}$. As a consequence of Theorem \ref{thm_mainCode}, we have generalized the results of Brison and Nogueira to a great extent, see Section \ref{sec_lrcc} for details.

There has been extensive work on the determination of the weight distributions of irreducible cyclic codes due to real-world applications, cf. \cites{DelGoe70,McEliece74,Langevin96,Ding}. This problem is extremely difficult in general, so codes with few weights are the focus of study in the literature. Two-weight irreducible cyclic codes are of particular interest, due to their connections with two-intersection sets in projective spaces, strongly regular Cayley graphs, and subdifference sets of the Singer difference set, cf. \cite{Schmidt2002}. Schmidt and White gave a necessary and sufficient numerical condition for an irreducible cyclic code to have at most two nonzero weights in \cite{Schmidt2002}. Their arguments involve Stickelberger's theorem on Gauss sums and the Parseval identity in discrete Fourier analysis.  In the same paper, they made the conjecture that after the reductions in \cite{Schmidt2002}{Remark 2.10}, all two-weight irreducible cyclic codes are either subfield codes, semiprimitive codes, or one of the eleven explicit codes with relatively small parameters. We refer to \cite{Vega15} for some remarks by Vega on the conjecture, and for definitions of the above terms. As an application of our main theorem, in Section \ref{sec:NSIC-2wt}, we confirm the Schmidt-White conjecture for non-standard {non-degenerate} irreducible cyclic codes by using Theorem \ref{thm_mainCode}.

The paper is organized as follows. In Section \ref{sec_prelim}, we first establish our notation, including some group-theoretic notation, and some important facts that we use throughout this paper. Then we justify the remark about degenerate cyclic codes following~\Tm{thm_mainCode}. Subsequently, we present {five} families of non-standard non-degenerate irreducible cyclic codes and four secondary constructions of new non-standard codes from old ones in Section \ref{sec_examples}.  In Section \ref{sec_newNSIC}, we further investigate one of the families of non-standard non-degenerate irreducible cyclic codes. In Section \ref{sec_mainproof}, we prove our main theorem. We embed $\PAut(C)$ into a classical group and examine all the possible Aschbacher classes. For dealing with the Aschbacher class $\cS$, the classification of finite permutation groups with a transitive cyclic subgroup in \cite{LiPraeger2012} plays a crucial role. In Section \ref{sec_application}, we discuss the applications of our results to linear recurring sequence subgroups and the Schmidt-White conjecture for non-degenerate irreducible cyclic codes.  We also consider the asymptotic behavior of the density of non-standard pairs in connection with the Berger-Charpin conjecture. We conclude this paper with some remarks on related open problems and future research directions.

\section{Preliminaries}\label{sec_prelim}
\subsection{\label{LSSnot}Notation}

Let $q$ be a power of a prime $p$, and let $\F_q$ be the finite field with $q$ elements. We write $\F_q^*:=\F_q\setminus\{0\}$. Given two subsets $A,B$ of $\F_q^*$, we define $A\cdot B=\{xy: x\in A,\, y\in B\}$. Let $n\ge 1$ be an integer. We denote the additive group of integers modulo~$n$ by~$\bbZ_n$, and we write~$\bbZ_n^*$ to denote the multiplicative group of units of~$\bbZ_n$. The {\em Euler totient function\/} $\varphi$ is defined by $\varphi(n)=|\bbZ_n^*|$. The {\em order\/} of $q$ modulo $n$, written as $\ord_n(q)$, is the smallest positive integer $m$ such that $n \mid (q^m-1)$.

Let $f(x)=f_0+f_1x+\cdots+f_kx^k\in\bbF_q[x]$ with $f_k\neq 0$. We write $\deg(f)$ to denote the {\em degree\/} $k$ of~$f(x)$. If also $f_0\neq 0$, then there exists a unique smallest positive integer~$n$ such that $f(x) \mid (x^n-1)$, which is denoted by $\ord(f)$ \cite{FiniteField}*{Lemma 3.1}. 

Let $C$ be a $q$-ary linear code of length $n$, that is, an $\bbF_q$-linear subspace of~$\bbF_q^n$. In this paper, we identify the coordinate positions $0,1,\ldots, n-1$ of vectors in~$\bbF_q^n$ with the corresponding elements in $\bbZ_n$, the additive group of integers modulo~$n$. The dual code of $C$ is $C^\perp=\{\bfx\in\F_q^n \mid \mbox{$(\bfx,\bfc)=0$ for all $\bfc\in C$}\}$, where $(\bfx,\bfc)=x_0c_0+x_1c_1+\cdots+x_{n-1}c_{n-1}$ is the Euclidean inner product on~$\F_q^n$.  We will use $C_{n,h,q}$ to denote the $q$-ary cyclic code of length~$n$ with check polynomial $h(x)$, where $h(x)\mid (x^n-1)$; if $\ord(h)=n$, then we also denote this code by $C_{h,q}$. We say that $C_{n,h,q}$ is \emph{non-degenerate} if $\ord(h)=n$ and \emph{degenerate} if~$\ord(h)<n$ \cite{Mac1977ErrorCorrectcode}*{p. 223}. Note that since we assume that $h(x)\mid (x^n-1)$, we have $\ord(h)\mid n$.

Recall that the symmetric group $S_n$ acts on the vectors of $\F_q^n$ as follows: for all $\pi\in S_n$ and for all $\bfv=(v_0, v_1, \ldots, v_{n-1})\in \F_q^n$, we have ${\pi(\bfv)}=(v_{\pi^{-1} (0)},\, \ldots,\, v_{\pi^{-1} (n-1)})$. We write $\AG(n)$ to denote the affine group  on $\mathbb{Z}_n$ consisting of all maps $x\mapsto tx+a$ on~$\bbZ_n$ with $a\in \bbZ_n$ and $t\in\bbZ_n^*$, whose order is $n\varphi(n)$. If a permutation $\pi$ (in its action on vectors) stabilizes~$C$ set-wise, we call $\pi$ a {\it permutation automorphism} of $C$. Let $\PAut(C)$ be the group of all permutation automorphisms of $C$. If $\PAut(C)$ contains the full cycle $\sigma=(0,1,\cdots, n-1)$, then $C$ is called a {\it cyclic} code. We also note that $C$ and $C^\perp$ share the same permutation automorphism group, that is, $\PAut(C^\perp)=\PAut(C)$.

\subsection{\label{SSL_group}Group theoretic notation and results}\label{subsec_grouptheory}

In this subsection, we introduce the group theoretic notation and facts that we shall use. We mainly follow the same notation as in \cite{BHR407}*{Section 1.2}. Let $C_n$ denote the \textit{cyclic group} of order $n$, and let $A_n$, $S_n$ be the \textit{alternating group} and the \textit{full symmetric group} on $n$ symbols, respectively. For a finite group~$G$, we write $Z(G)$ for its center, and we write $o(g)$ for the order of $g\in G$. The exponent  of~$G$ is the least common multiple of the orders of all its elements. {In what follows, we use the left group action, i.e.,  $(gh)(a)=g\big( h(a)\big)$.  For $g,h\in G$, let $g^h=h^{}gh^{-1}$ and $[g,h]=g^{}h^{}g^{-1}h^{-1}$}. 
We will denote the derived group by $G'=\langle [g,h]:\, g,h\in G\rangle$, and write $G^{(\infty)}$ for the terminating member of the derived series of $G$.  A group $G$ is called \textit{almost simple} if $S\leq G\leq \Aut(S)$ for some non-abelian simple group $S$.  A group $G$ is {\it quasisimple} if $G=G'$ and $G/Z(G)$ is a non-abelian simple group.

Let $G$ and $H$ be groups acting on finite sets $\Omega$ and $I$, respectively. The wreath product of $G$ by $H$, denoted by $G\wr H$, is a semidirect product $G^I\rtimes H$, where $G^I$ is the direct product of $|I|$ copies of $G$, and $H$ acts on $G^I$ by permuting the coordinates: $h: (g_j)_{j\in I}\mapsto (g_{h^{-1}(j)})_{j\in I}$.  For $(g_i)_{i\in I}\in G^I$ and $h\in H$,  the group element $((g_i)_{i\in I}, h)$ maps $(a_i)_{i\in I}$ to {$\big(g_{h^{-1}(i)} (a_{h^{-1}(i)})\big)_{i\in I}$}~in~$\Omega^I$.
In particular, $|G\wr H|=|G|^{|I|}|H|$. See \cites{MR2192256, PP1968} for more details on the wreath product.

For a vector space $V$ over a finite field $\F$,  we write $\GL(V)$ for the group of all invertible $\F$-linear transformations of $V$. We have $Z(\GL(V))=\{\lambda\cdot \textup{id}_V:\lambda\in\F^*\}$, where $\textup{id}_V$ is the identity map on $V$. We use $\PGL(V)$ for the projective linear group $\GL(V)/Z(\GL(V))$. If $G$ is a subgroup of $\GL(V)$, we use $\bar{G}$ for its quotient group in $\PGL(V)$. If $\dim V=m$ and $\F=\F_q$, we write $\GL_m(q)$ for $\GL(V)$.

Let $V_1$, $V_2$ be two vector spaces over $\F$ of dimensions $m_1,m_2$, respectively,  and let $H,K$ be subgroups of $\GL(V_1)$ and $\GL(V_2)$ respectively. The central product $H\circ K$ is the quotient group $(H\times K)/N$, where $N=\{(\lambda \cdot \textup{id}_{V_1},\,\lambda^{-1}\cdot \textup{id}_{V_2}):\ \lambda\in\F\setminus\{0\},\ \lambda\cdot \textup{id}_{V_1}\in H,\ \lambda\cdot \textup{id}_{V_2}\in K\}$. It acts on the tensor space $V_1\otimes V_2$ as follows: $g_1\circ g_2$ maps $v_1\otimes v_2$ to $g_1(v_1)\otimes g_2(v_2)$.

\begin{lemma}\label{lem_irre}
Let $g\in \GL_m(q)$ have order $n$, {where $\gcd(n,q)=1$},  and let $m_g(x)$ be its minimal polynomial over $\F_q$. If~$\la g\ra$ acts irreducibly on $V\cong \F_{q^m}$, then $\ord_n(q)=m$.
\end{lemma}
\begin{proof}
Suppose  that $\langle g \rangle$ acts irreducibly on $V \cong \F_{q^m}$. Then $g$ stabilizes no nontrivial proper subspaces of~$V$, so by \cite{HK71}*{Section 7.1, Theorem 1}, the minimal polynomial $m_g(x)$~of $g$ has degree~$m$. Moreover, since $g$ has order~$n$, we have $m_g(x)\mid (x^n-1)$. {As $\gcd(n,q)=1$},  $m_g(x)$ has no multiple roots. Hence, by the Primary Decomposition Theory \cite{HK71}*{Section 6.8, Theorem 12}, the minimal polynomial of~$g$ is irreducible. So if $\ga$ is a {root} of~$m_g(x)$, then $m_g(x)=(x-\ga)(x-\ga^q)\cdots(x-\ga^{q^{m'-1}})$ with $m'=\ord_n(q)$, so $m=m'=\ord_n(q)$.
\end{proof}

\begin{lemma}\label{lem_exponent}
Let $r$ be a prime, and let $H$ be a Sylow $r$-subgroup of $\GL_k(r)$.  Let $e$ be the smallest integer such that  $k\leq r^{e}$. Then the exponent of $H$ is at most $r^e$.
\end{lemma}
\begin{proof}Let $I_k$ denote the $k\times k$ identity matrix of $\GL_{k}(r)$.  Up to conjugation, $H$ is a subgroup of $\GL_k(r)$ consisting of all lower triangular matrices with diagonal entries 1.  For any $M\in H$, we have  $M=I_k+R\in H$, where $R$ is the nilpotent part of $M$ satisfying $R^k=0$.  We deduce from $r^{e-1}<k\leq r^{e}$ that $(M)^{r^e}=(I_k+R)^{r^e}=I_k+R^{r^e}=I_k$ for any $M\in H$, where we used the fact that $(A+B)^r=A^r+B^r$ if  $AB=BA$. This completes the proof.
\end{proof}

We shall need Aschbacher's classification theorem \cite{Aschbacher1984} on the maximal subgroups of finite almost simple classical groups. There are eight classes, $\cC_1,\cC_2,\ldots,\cC_8$, of geometric type which are stabilizers of certain geometric structures. Table \ref{Table_Asch} contains a rough description of these Aschbacher classes, cf.
\cite{BHR407}*{Table 2.1}, and we refer to the monograph \cite{KL1990} by  Kleidman and Liebeck for more detailed information.  The Aschbacher class $\cS$ consists of all the subgroups not covered by the eight classes of geometric type. We refer to \cite{BHR407}*{Definition 2.1.3} for the precise description of the Aschbacher class $\cS$.

\begin{table}[!hbtp]
\centering
\caption{Rough geometric interpretations of Aschbacher classes}
\label{Table_Asch}
\renewcommand{\arraystretch}{1.4} 
\setlength{\tabcolsep}{6pt} 
\begin{tabular}{@{}>{\centering}m{1.5cm}|l@{}}
\toprule
\textbf{Class} & \textbf{Geometric interpretation} \\
\midrule
$\cC_1$ & Stabilizers of totally singular or nonsingular subspaces \\
\hline
$\cC_2$ & Stabilizers of decompositions $V=\bigoplus_{i=1}^t V_i$, $\dim V_i= a$ \\  \hline
$\cC_3$ & Stabilizers of vector space structures over extension fields of prime degree \\  \hline
$\cC_4$ & Stabilizers of tensor product decompositions $V=V_1\otimes V_2$ \\
\hline
$\cC_5$ & Stabilizers of subgeometries over subfields of prime index \\ \hline
   \multirow{2}{*}{$\cC_6$} & Normalizers of $r$-groups of symplectic type  with $V$ as \\
         &  their absolutely irreducible module \\
\hline
$\cC_7$ & Stabilizers of tensor decompositions $V = \bigotimes_{i=1}^k V_i$, $\dim V_i =a$ \\
\hline
$\cC_8$ & Groups of similarities of non-degenerate classical forms \\
\bottomrule
\end{tabular}
\end{table}

\subsection{Standard and non-standard non-degenerate cyclic codes}\label{subsec_standard} %

We now introduce some notation that will be used frequently throughout this paper.
\begin{definition}\label{def_UL}
Assume that $\gcd(n,q)=1$, and let $m=\textup{ord}_n(q)$ be the smallest positive integer such that $n$ divides $q^m-1$. Let $\cU_{n,q}$ be the multiplicative subgroup of order $n$ of~$\F_{q^m}^*$, and let $\xi$ be a generator of $\cU_{n,q}$. Let $\GL_m(q)$ be the group of all invertible $\F_q$-linear transformations of $\F_{q^m}$, and let $\cL(n,q)$ be the set-wise stabilizer of $\cU_{n,q}$ in $\GL_m(q)$. Let $\cL_{\textup{st}}(n,q)$ be the subgroup of $\cL(n,q)$ generated by $\sigma$ and $\psi$, where {$\sigma(\alpha)=\xi \alpha$ and $\psi(\alpha)=\alpha^q$ for $\alpha\in\F_{q^m}$}. Note that $|\cL_{\textup{st}}(n,q)|=nm$. In the remainder of this paper, we always use $n_0$ to denote $\gcd(n,q-1)$, that is, $n_0:=\gcd(n, q-1)$.
\end{definition}

We adopt the notation of Definition \ref{def_UL}. Let $C$ be a $q$-ary cyclic code of length~$n$. We can identify the coordinate positions $0,1,\ldots,n-1$ in~$\bbZ_n$ with $1,\xi,\ldots,\xi^{n-1}$, respectively, so that the two definitions of $\sigma$ coincide. By a slight abuse of notation, we also use $\psi$ to denote the corresponding permutation of coordinate positions $i\mapsto qi$ on~$\bbZ_n$ induced by the action of~$\psi$. Then $\psi$ is also in $\PAut(C)$, cf. \cite[Corollary 5.19]{PHB98}. More generally, in~\cite[Theorem 5.28]{Hollmann2023} it is shown that if~$C$ has check polynomial~$h(x)$, then the map {$\mu_t: \alpha\mapsto t\alpha$} in~$\AG(n)$ is in~$\PAut(C)$ if and only if the map {$\mu_t: \alpha\mapsto \alpha^t$} with $\gcd(n,t)=1$ permutes the zeros of~$h(x)$; note that $\mu_q=\psi$. In what follows, we let $\PAut_{\rm st}(C):=\PAut(C)\cap \AG(n)$. We will say that a 
cyclic code~$C$ is {\em standard\/} if $\PAut(C)=\PAut_{\rm st}(C)$, that is, if $\PAut(C)$ is a subgroup of~$\AG(n)$.  It is shown in~\cite[Lemma 5.30]{Hollmann2023} that every cyclic code $C$ of length~$n\leq 3$ has $\PAut(C)\subseteq \AG(n)$,  so every cyclic code of length at most $3$ is standard.

An {\em irreducible\/} cyclic code is a cyclic code $C=C_{n,h,q}$ whose check polynomial~$h(x)$ is irreducible over $\F_q$, see, e.g., \cite[Definition 6.1.6]{GTMcoding1992vanLint} and \cite[Definition 2.2]{Schmidt2002}. By~\cite[Corollary 5.14]{Hollmann2023}, {if $C$ is non-degenerate, we have}
\begin{equation}\label{eqn_PAutst}
  \PAut_{\textup{st}}(C)=\la \sigma,\psi\ra.
\end{equation}

\subsection{Degenerate cyclic codes}

For the sake of completeness, we investigate which degenerate cyclic codes only possess permutation automorphisms contained in~$\la \gs,\psi\ra$ or only in~$\AG(n)$. The result is as follows.

\begin{lemma}\label{lem_deg}
Let $C=C_{n,h,q}$ be {a nonzero} $q$-ary cyclic code of length~$n$ with check polynomial $h(x)$. Let $m=\textup{ord}_n(q)$, let $m'=\deg(h)$, and let $n'=\ord(h)$ \big(the least positive integer such that $h(x)\mid (x^{n'}-1)$\big). Then  $n'\mid n$ and ${\rm dim}(C)=m'$. If $n'<n$, then $C$ is periodic and consists of all the codewords of the form $(\bfc', \bfc', \ldots, \bfc')$ with $\bfc' \in C'=C_{h,q}$, the $q$-ary non-degenerate cyclic code of length $n'$ with check polynomial $h(x)$. Moreover, {with $k:=n/n'$,}
\[\PAut(C) = S_k\wr\PAut(C'),\]
so that $|\PAut(C)|=(k!)^{n'}|\PAut(C')|$.
Finally, {if $n'<n$} then $C$ is non-standard, that is, $\PAut(C)\nsubseteq \AG(n)$, with the exception of the following four cases:
\ben
\item
$n=2$, {$n'=1$,} $h(x)=x-1$, $q$ is odd, and $C$ is the repetition code of length~2; here $\PAut(C)=S_2=\AG(2)=\la \gs,\psi\ra$.
\item
$n=3$, {$n'=1$,} $h(x)=x-1$, $q\equiv 1,2\bmod 3$, and $C$ is the repetition code of length~3; here $\PAut(C)=S_3\cong\AG(3)$. 
\item
$n=4$, {$n'=2$,} $h(x)=x+1$, $q$ is odd, and $C$ is the $\bbF_q$-span of $(1,-1,1,-1)$; here $\PAut(C)=S_2\wr S_2\cong \AG(4)$. 
\item
$n=4$, {$n'=2$,} $h(x)=x^2-1$, $q$ is odd, and {$C=\{(a,b,a,b):\ a,b\in \F_q\}\cong \F_q^2$}; here $\PAut(C)=S_2\wr S_2\cong \AG(4)$.
\een
\end{lemma}
\begin{proof}

Since $h(x)\mid (x^n-1)$, we have $n'\mid n$ and $\dim C=\deg h=m'$. If $n'<n${, so $k=n/n'\ge 2$,} then by \cite[Lemma~5.22 and Corollary~5.24]{Hollmann2023},
\[
C=\{(\mathbf{c}',\ldots,\mathbf{c}'):\ \mathbf{c}'\in C'\},\qquad
\PAut(C)=S_k^{n'}\rtimes \PAut(C'),\qquad C'=C_{h,q}.
\]
Thus $|\PAut(C)|=(k!)^{n'}|\PAut(C')|\ge (k!)^{n'}\cdot n'\cdot \ord_{n'}(q)\geq(k!)^{n'} n' $.

On the other hand, $|\AG(n)|=n\varphi(n)=kn'\varphi(kn')<(kn')^2$. Now since $k\geq 2$, we have $(k!)^{n'-2}\geq n'$ for $n'\geq 4$. Using this, a routine verification shows that the inequality {$(k!)^{n'} n' < (kn')^2$} cannot occur when $n'\geq 4$, and can only occur when $k\leq 3$  if $n'=1$ and $k\leq 2$ if $n'\in \{2,3\}$.   In particular, this leaves the possibilities $(n',k)\in\{(1,2),(1,3),(2,2),(3,2)\}$.
However, the case $(n',k)=(3,2)$ is impossible, since then $n=6$ and
$|\PAut(C)|\ge (2!)^3\cdot 3=24>|\AG(6)|=12$.
Hence $|\PAut(C)|\le |\AG(n)|$ can occur only for
$(n',k)\in\{(1,2),(1,3),(2,2)\}$.  We handle these three exceptional cases one-by-one. The condition on $q$ is obtained from the assumption $\gcd(n,q)=1$.
If $(n',k)=(1,2)$, then $n=2$, $h(x)=x-1$.
If $(n',k)=(1,3)$, then $n=3$, $h(x)=x-1$, yielding (2).
If $(n',k)=(2,2)$, then $n=4$, $h(x)\in\{x+1,x^2-1\}$, yielding (3) and (4).  This completes the proof.
\end{proof}

As a consequence of~\Le{lem_deg}, {nonzero} degenerate cyclic codes of length~$n$ are periodic and possess permutation automorphisms outside~$\AG(n)$, with only a few exceptions. By the above lemma, the automorphism group of such a code is completely determined by the smaller code that it extends periodically. 

{\bf So in the remainder of this paper, we assume that the $q$-ary irreducible cyclic codes of length~$n$ under discussion are {\em non-degenerate\/}, so with a check polynomial of order~$n$, and have dimension $m=\textup{ord}_n(q)$.} For convenience, we use the abbreviations {\bf SIC codes} (resp., {\bf NSIC codes}) for standard (resp., non-standard) non-degenerate irreducible cyclic codes, respectively.

\subsection{Non-degenerate irreducible cyclic codes}

In the rest of this paper, we will mostly consider non-degenerate $q$-ary irreducible cyclic codes~$C$ of length~$n$. If $C$ is such a code,  then by~\cite[Theorem 5.15]{Hollmann2023}, $\PAut(C)$ is isomorphic to $\cL(n,q)$ and $\PAut_{\textup{st}}(C)$ is isomorphic to $\cL_{\textup{st}}(n,q)$ under our identification of coordinate positions with the elements of $\cU_{n,q}$.  In particular, each permutation in $\PAut(C)$ corresponds to an $\F_q$-linear transformation of $\F_{q^m}$ that stabilizes $\cU_{n,q}$. If $\cL_{\textup{st}}(n,q)\varsubsetneqq \cL(n,q)$, then $C$ is a {\it non-standard} irreducible cyclic code, and $C$ is {\em standard\/} otherwise. Note that by the above, this definition coincides with the definition given earlier for general irreducible cyclic codes.

Irreducible cyclic codes can be represented as trace codes. We give the detailed explanations below. Let $C$ be an irreducible cyclic code of length $n$ over $\F_q$, and let $m=\ord_n(q)$. Let $h(x)$ be the check polynomial of $C$, and let $m':=\deg (h)$ and $n'=\ord(h)$. Then $n'\mid n$, and since $h(x)$ is irreducible over~$\bbF_q$, we have $m'=\ord_{n'}(q)={\rm dim}_{\F_q}(C)$, and hence $m'\mid m$. {Let $\beta\in \F_{q^{m'}}$ be such that $\beta^{-1}$ is a root of $h(x)$.}
Note that $\gb$ has order~$n'$. Then by~\cite[Theorem 5.25]{PHB98} or \cite[Section 7.2]{SRG}, the code $C$ can be represented as a trace code  consisting of codewords 
\begin{equation}\label{LEtrace}  \bfc_{\alpha}:=\big(\Tr(\alpha),\Tr(\alpha\beta),\ldots,\Tr(\alpha\beta^{n-1})\big),\, \quad \alpha  \in\F_{q^{m'}},
\end{equation}
where $\Tr=\Tr_{\F_{q^{m'}}/\F_q}$ is the trace function from $\F_{q^{m'}}$ to $\F_q$. 
The next lemma should be compared to~\cite[Corollary 5.24]{Hollmann2023}.
\ble{LLtrace}With the above notation, suppose that $n'<n$. Then the codewords of $C=C_{n,h,q}$ are periodic; more precisely, for every $\ga\in\bbF_{q^{m'}}$, we have $\bfc_{\alpha}=({\bfc_{\alpha}',\bfc_{\alpha}' ,\ldots,\bfc_{\alpha}'})$ (repeated $n/n'$ times),
where $\bfc_{\alpha}'$ is a codeword of the $q$-ary irreducible cyclic code $C'=C_{h,q}$ of length~$n'$ and dimension $m'$ with check polynomial $h(x)$.
\ele

\bpf
By the above arguments, the $q$-ary code $C'$ of length~$n'$ and check polynomial $h(x)$ has codewords $\bfc'_\ga=(\Tr(\alpha),\Tr(\alpha\beta),\ldots,\Tr(\alpha\beta^{n'-1}))$ for $\ga\in\bbF_{q^{m'}}$. Since $\gb$ has order~$n'$, the claim immediately follows from (\ref{LEtrace}).
\epf


\section{Examples of non-standard non-degenerate irreducible cyclic codes}\label{sec_examples}
We list some known examples of non-standard non-degenerate irreducible cyclic codes (NSIC codes) below. In Examples \ref{exa_repetition}--\ref{exa_Golay}, $\PAut(C)$ is non-solvable {except when $(m,q)=(2,3)$ in Example \ref{exa_whole}}; and it is worth noting that the extension ovoid codes in Example \ref{exa_C8} have not previously been identified as non-standard.
\begin{example}[Duals of repetition codes]\label{exa_repetition}\cite[Theorem 5.2]{Brison2010}, \cite[Example 3]{Hollmann2023}
Let $n$ be a prime such that $n\ge 5$, and let $p$ be a prime such that $p$ has order $m=n-1$ modulo $n$, i.e., $\ord_n(p)=m$. Let $C$ be the dual code of a repetition code of length $n$ over $\F_p$. Then $C$ is an NSIC code with $\PAut(C)=S_n$.
\end{example}

\begin{example}[Duals of primitive BCH codes of designed distance 2]\label{exa_whole}\cite[Theorem 4.3]{Brison2010}, \cite[Example 1]{Hollmann2023}
Let $n=q^m-1$, where $m\geq 2$ and $(m,q)\ne (2,2)$. Then $\cU_{n,q}=\F_{q^m}^*$, and $\cL(n,q)=\GL_m(q)$.  The corresponding irreducible cyclic code $C$ of length $n$ over $\F_q$ is non-standard, and it is permutation equivalent to the dual of a primitive narrow-sense BCH code of designed distance $2$, cf. \cite[Chapter 9.1]{Mac1977ErrorCorrectcode}.
\end{example}

\begin{example}[Extension ovoid codes]\label{exa_C8}
Let $n=(q-1)(q^2+1)$, so that $m=4$. Let $\gamma$ be a nonzero element of $\F_{q^2}$ such that $\gamma+\gamma^q=0$, and define $Q(x)=\tr_{\F_{q^2}/\F_q}(\gamma x^{q^2+1})$ for $x\in\F_{q^4}$. Then $\cU_{n,q}=\{x\in\F_{q^4}^*:\,Q(x)=0\}$, and $\cL(n,q)$ contains $\GO_{4}^-(q)$. The corresponding irreducible cyclic code $C$ of length $n$ over $\F_q$ is non-standard.
\end{example}

\begin{example}[Duals of Golay codes]\label{exa_Golay}\cite[Example 4, Example 5]{Hollmann2023}
Let $C$ be the dual of the binary Golay code of length $n=23$ or the dual of the ternary Golay code of length $n=11$. The automorphism group of $C$ is the Mathieu group $M_{23}$ or $2.M_{11}$, and $\PAut(C)$ is $M_{23}$ or $\PSL_2(11)$ respectively. In both cases, $\PAut(C)$ contains a full cycle of length $n$ and $C$ is an NSIC code.
\end{example}
We have the following new family of NSIC codes \cite[Section VI]{Hollmann2022ISIT}, \cite[Section 6]{HollmannII}, some special cases of which have already appeared in \cites{Hollmann2023} or \cite[Theorem 6.5]{Brison2010}.

\begin{example}[Equally-spaced check polynomials]\label{exa_imp}
Let $(n,n',k)$ be a triple with $n\geq 6$ that satisfies the following conditions:
\begin{equation}\label{eqn_nn0rCond}
  n=n'k,\quad k\ge 2,\quad \ord_n(q)=k\cdot\ord_{n'}(q).
\end{equation}
Then the irreducible cyclic code $C$ of length $n$ over $\F_q$ is non-standard, and $\PAut(C)$ contains a subgroup isomorphic to $(C_{n'} \rtimes C_{m'})^k\rtimes S_k$, where $m'=\ord_{n'}(q)$.
\end{example}
\begin{remark}
We shall give a proof that the code in Example \ref{exa_imp} is {non-degenerate and non-standard if and only if $n\ge 6$} in Section \ref{sec_newNSIC}. In Proposition \ref{prop_nn0rChar}, we shall also give an arithmetic characterization of the triples $(n,n',k)$ that satisfy the conditions in \eqref{eqn_nn0rCond}. {We list some explicit examples in the following table, where $m'=\ord_{n'}(q)$, $m=\ord_n(q)$, $k=n/n'$ and $h(x)$ is a check polynomial.
\begin{table}[htbp]
$$\begin{array}{cccccccc} \hline
q & n & n' & k & m & m' & h(x) & \PAut(C)\\ \hline
2 & 9 & 3 & 3 & 6 & 2 & x^6+x^3+1 & S_3\wr S_3\\
3 & 16 & 8 & 2 & 4 & 2 & x^4+x^2-1 & \GL_2(3)\wr C_2\\
5 & 8 & 4 & 2 & 2 & 1 & x^2-2 &C_4\wr C_2\\\hline
\end{array}$$
\end{table}
}
\end{remark}
We next present some secondary constructions of NSIC codes from known ones. The first two constructions are due to the second-named author, cf. \cites{Hollmann2022ISIT,Hollmann2023,HollmannII}, where the first construction is based on \cite[Theorem 3.1]{Brison2010} (called ``Lifting'' there). We employ the notation that we have introduced in Definition \ref{def_UL}, and identify a nonzero element of $\F_q$ with an element of $\GL_m(q)$ via left multiplication.

\begin{construction}[Extension method]\label{exa_ext}
Let $C$ be an NSIC code of length $n$ over $\F_q$,  let $m=\textup{ord}_n(q)$ and $u$ be a divisor of 
{$(q-1)/n_0$ (recall that $n_0=\gcd(n,q-1)$)}. The irreducible cyclic code $C'$ of length $n'=nu$ over $\F_q$ is non-standard.
\end{construction}
\begin{proof}
It is routine to check that $\ord_{n'}(q)=m$, 
and that $\cU_{n',q}$ is generated by $\cU_{n,q}$ and a subgroup $H$ of order $u\gcd(q-1,n)$ in $\F_q^*$. We also have
$\cL_{\textup{st}}(n',q)=\cL_{\textup{st}}(n,q)H$, and $\cL(n',q)$ contains $\la\cL(n,q),H\ra$. The subgroup $\cL(n,q)\cap H$ has order $\gcd(n,q-1)$ and is contained in $\cL_{\textup{st}}(n,q)$, so $\cL_{\textup{st}}(n',q)\cap \cL(n,q)=\cL_{\textup{st}}(n,q)$ by Dedekind's law. It follows that $\cL(n',q)\ne\cL_{\textup{st}}(n',q)$, so $C'$ is non-standard.
{For more details, see \cite[Section 5]{HollmannII}.}
\end{proof}

\begin{construction}[Lifting method]\label{exa_lift}
Let $C$ be an NSIC code of length $n$ over $\F_q$, and let $m=\textup{ord}_n(q)$. Take $t\in\mathbb{N}$ such that $\gcd(m,t)=1$. The code $C'= C\otimes\F_{q^t}$, which is the $\F_{q^t}$-span of $C$, is an NSIC code of length $n$ over $\F_{q^t}$.
\end{construction}
\begin{proof}
It is routine to check that $C'$ is a cyclic code with the same {check polynomial $h(x)$} as $C$, and {$h(x)$} is irreducible over $\F_{q^t}$ because $\gcd(m,t)=1$. We have  $\ord_n(q^t)=m$, $\cU_{n,q}=\cU_{n,q^t}$, $\cL_{\textup{st}}(n,q^t)\cap\GL_m(q)=\cL_{\textup{st}}(n,q)$, and $\cL(n,q^t)\cap\GL_m(q)$ contains $\cL(n,q)$ which is strictly larger than $\cL_{\textup{st}}(n,q)$. Hence $C'$ is non-standard.
\end{proof}

\begin{construction}[Descending method, special case]\label{exa_expansion}
Let $n\ge 4$, and let $r$ be a divisor of $m=\ord_n(q)$  with $r>1$. Let $C_0$ be an NSIC code of length $n$ over $\F_{q^r}$ with check polynomial $h_0(x)$, of degree $m_0=\ord_n(q^r)$. Let $\gb^{-1}$ be a zero of $h_0(x)$, and let $h(x)$ denote the minimal polynomial of~$\gb^{-1}$ over~$\bbF_q$. Let $C$ be the irreducible cyclic code over~$\bbF_q$ with check polynomial $h(x)$. Then $C=\{ \left(\tr_{\F_{q^r}/\F_q}(c_1),\ldots,\tr_{\F_{q^r}/\F_q}(c_n) \right)\mid(c_1,\ldots,c_n)\in C_0 \}$, and $C$ is non-standard of dimension $m$.
\end{construction}
\begin{proof}
{Note that $m_0=m/r$.}
The expression for the codewords of $C$ follows directly from the trace representations as in~(\ref{LEtrace}) for $C$ and~$C_0$.
Since $\GL_m(q)$ and $\GL_{m_0}(q^r)$ are the groups of invertible $\F_q$-linear and $\F_{q^r}$-linear transformations of $\F_{q^m}$ respectively, the former group contains the latter. It follows that $\cL(n,q)\cap \GL_{m_0}(q^r)=\cL(n,q^r)$. We also have $\cL_{\textup{st}}(n,q)\cap\GL_{m_0}(q^r)=\cL_{\textup{st}}(n,q^r)$. Since $\cL(n,q^r)>\cL_{\textup{st}}(n,q^r)$ by the fact that $C_0$ is non-standard, we deduce that $\cL(n,q)\varsupsetneqq\cL_{\textup{st}}(n,q)$, so $C$ is non-standard. As an irreducible cyclic code over $\F_q$ of length $n$, $C$ has dimension $m=\ord_n(q)$ over $\F_q$.
\end{proof}

\begin{construction}[Tensor-product method]\label{exa_product}
Let $C$ be an NSIC code of length $n$ over $\F_q$, and let $m=\textup{ord}_n(q)$. Take an integer $s>1$ such that $t=\ord_s(q)>1$ and $\gcd(m,t)=1$, and set $n'=\lcm(n,s)$. The irreducible cyclic code $C'$ of length $n'$ over $\F_{q}$ is
non-standard.
\end{construction}
\begin{proof}
We use $\cU_a$ to denote the multiplicative subgroup $\cU_{a,q}$; we will repeatedly use the fact that $\cU_a\cap \cU_b=\cU_{\gcd(a,b)}$ and $\cU_a\cU_b=\cU_{\lcm(a,b)}$. 

Let $\xi\in\bbF_{q^m}$ have order~$n$ and degree~$m=\ord_n(q)$ over~$\bbF_q$, where $m$ is the degree of the minimal polynomial of $\xi$ over $\bbF_q$. Since $\gcd(m,t)=1$, we also have $m=\ord_n(q^t)$, hence the minimal polynomials of $\xi$ over~$\bbF_q$ and over~$\bbF_{q^t}$ are identical. So the basis $\{e_0:=1,e_1:=\xi, \ldots, e_{m-1}:=\xi^{m-1}\}$ for $\bbF_{q^m}$ over~$\bbF_q$ is also a basis for $\bbF_{q^{mt}}$ over~$\bbF_{q^t}$. As a consequence, every $\bbF_q$-linear map $L$ on~$\bbF_{q^m}$ extends uniquely to an {$\bbF_{q^t}$}-linear map {$\widetilde{L}$} on~$\bbF_{q^{mt}}$, where if $x=\sum_{i=0}^{m-1} x_ie_i\in \bbF_{q^{mt}}$ with $x_i\in \bbF_{q^t}$ for all $i$, then ${\widetilde{L}}(x)=\sum_{i=0}^{m-1} x_iL(e_i)$. {Let $\widetilde{\cL}(n,q)=\{\widetilde{L}:L\in\cL(n,q)\}$, which is isomorphic to $\cL(n,q)$.}  

{It is easy to see that $\cU_{n'}=\cU_n\cU_s$  is stabilized by $\widetilde{\cL}(n,q)$. To establish the claim that $C'$ is non-standard, it suffices to show that $\widetilde{\cL}(n,q)$ is not contained in $\cLst(n',q)$. Take $L\in \cL(n,q)\setminus \cL_{st}(n,q)$, and let $\widetilde{L}$ be the corresponding $\F_{q^t}$-linear map in $\widetilde{\cL}(n,q)$. If $\widetilde{L}$ were in $\cL_{st}(n',q)$, then there would be $\mu\in\cU_{n'}$ and $i\in\mathbb{N}$ such that $\widetilde {L}(\alpha)=\mu \alpha^{q^i}$ for $\alpha\in\bbF_{q^{mt}}$. Since $\mu=\widetilde L(1)=L(1)\in\cU_{n}$,  this would imply that $L\in\cLst(n,q)$: a contradiction.}
\end{proof}

\begin{remark}
    It is worth noting that Example $\ref{exa_imp}$ and Constructions \ref{exa_ext}--\ref{exa_product} have a deep relation with the geometric classes of Aschbacher's classification theorem. If $\cL(n,q)$ is contained in one of the precise geometric classes $\cC_1$-$\cC_7$, the structure of $\cL(n,q)$ can be determined roughly. We establish these results in Propositions \ref{prop_C3}--\ref{prop_C6}.
\end{remark}

\begin{remark}Construction~\ref{exa_expansion} is a special case of the descending method of \cite[Proposition 3.4]{Brison2010}, see also \cite[Theorem 5.3]{HollmannII}. In fact, the main conclusion holds without any condition on~$r$. Moreover, it can be shown that Construction~\ref{exa_product}
 can be obtained from Constructions \ref{exa_ext} and \ref{exa_lift} in combination with  the descending method.
\end{remark}

\section{A new family of non-standard non-degenerate irreducible cyclic codes}\label{sec_newNSIC}
In this section, we present the proof for the fact that the irreducible cyclic codes in Example \ref{exa_imp} are non-standard. This is achieved by establishing the following slightly stronger result.
\begin{lemma}\label{lem_C2para}
Suppose that $n=n'k$ with $k\ge 2$ and $n'\geq2$, and set  $m=\ord_{n}(q)$, $m'=\ord_{n'}(q)$. Further, assume that $m'k=m$. Then $\cL_{\textup{st}}(n',q)\wr S_k\le \cL(n,q)$, and the irreducible cyclic code of length $n$ over $\F_q$ is non-degenerate and non-standard if and only if $n \geq 6$.
\end{lemma}
\begin{proof}
Take notation as in Definition \ref{def_UL}. Since $\ord_{n'k}(q)=m=m'k$, there exists $\xi\in \F_{q^{m'k}}$ of order~$n'k$ such that $\F_{q^m}=\bbF_q(\xi)$. Then $\xi^k$ has order $n'$ and $\F_{q^{m'}}=\F_q(\xi^k)$. Since $\langle \xi\rangle=\cup_{i=0}^{k-1}\langle \xi^k\rangle\xi^{i}$ and $m=km'$, we conclude that $\{1,\xi, \ldots, \xi^{k-1}\}$ is a basis for $\F_{q^{m}}$ over~$\F_{q^{m'}}$.

We claim that $\cL(n,q)$ contains $\cL_{\textup{st}}(n',q)\wr S_k$. For $\pi\in S_k$ and $L_0, \ldots, L_{k-1}\in \cL(n',q)$, define the map $L=L_{\pi, L_0, \ldots, L_{k-1}}$ by setting $L(\xi^i a):=\xi^{\pi(i)}L_i(a)$ for $i=0,\ldots, k-1$ and $a\in \F_{q^{m'}}$, and extend $L$ by~$\F_q$-linearity to a map on~$\F_{q^{m}}$. By the above, $L$ is well-defined and $\F_q$-linear, and moreover $L$ fixes $\langle\xi\rangle$, so that $L\in\cL(n,q)$. This proves the claim.

By the previous claim, $\cL(n,q)$ has size at least $k! (n'm')^k$. {The lower bound} is larger than $|\cL_{\textup{st}}(n'k,q)|$ if and only if $(n'm')^{k-1}(k-1)!>k$, that is, if $k\geq 3$ or if $k=2$ and $n'\geq 3$, so when $n\geq 6$. In the remaining case, we have $k=2$, $n'= 2$ and $m'=1$, so $q$ is odd, $n=4$, and $m=2$. From the factorization $x^4-1=(x^2+1)(x+1)(x-1)$, we conclude that $\xi$ has minimal polynomial $x^2+1$ over~$\F_q$, and then the generator polynomial is $x^2-1$ over $\F_q$. So the corresponding irreducible cyclic code $C$ is {the $2$-dimensional  code $\{(a,b,-a,-b):\ a,b\in \F_q\}$}. It is now easily checked that $|\PAut(C)|=nm=8$, so in this case, $C$ is standard.
\end{proof}

We next determine the triples $(n,n',k)$ that satisfy the conditions in \eqref{eqn_nn0rCond}, i.e.,
\begin{equation*}
  n=n'k,\; k\ge 2,\;\ord_n(q)=k \cdot\ord_{n'}(q).
\end{equation*}
%
%
We need some preparation. For an integer $n\ge 1$ and a prime $p$, we use $\pi(n)$ to denote the set of prime divisors of $n$, and we use $\nu_p(n)$ to denote the largest non-negative integer~$i$ such that $p^i$ divides $n$.  We need a well-known technical result, which is actually a direct consequence of a result often referred to as the ``Lifting the Exponent'' Lemma\footnote{ \url{https://en.wikipedia.org/wiki/Lifting-the-exponent_lemma}.}. For the sake of completeness, we include a proof here.
\begin{lemma}\label{lem_basic}
Let $r$ be a prime, and let {$N> 1$} be an integer with $f:=\nu_r(N-1)\geq 1$.
\bit
\item[(1)] If~$e$ is a positive integer with $r\not\,\mid e$, then $\nu_r(N^e-1)=f$.
\item[(2)] We have
$\nu_r(N^r-1)=f+1$ if $(r,f)\neq (2,1)$ and $\nu_r(N^r-1)\ge f+2$ if $(r,f)=(2,1)$.
\eit
\end{lemma}
\begin{proof}
By assumption, there is an integer $b$ relatively prime to $r$ such that $N=1+br^f$. If $r\nmid e$, then $N^e\equiv 1+eb r^f \not\equiv 1\bmod r^{f+1}$, and the first claim follows. Next, since $r$ divides $ \binom{r}{i}$ for $i=1,\ldots, r-1$, we have $N^r\equiv 1+br^{f+1}\bmod r^{f+2}$ unless $r=2$ and $f=1$. In the latter case, we have $N=1+2b$ with $b$ odd, hence $N^2-1=4b+4b^2\equiv 0\bmod 8$.
\end{proof}
As a consequence, we have the following.
\begin{lemma}\label{lem_ordqn}
Let $n'\ge 1$ and $q\ge 2$ be integers such that $\gcd(n',q)=1$. Define $m':=\ord_{n'}(q)$, and set $\gD:=(q^{m'}-1)/n'$.  Let $r$ be a prime such that $\gcd(r,q)=1$.
\begin{itemize}
  \item[(1)] If $r\not\in\pi(n')$, then $\ord_{rn'}(q)$ divides $(r-1)m'$.
  \item[(2)] If $r\in\pi(n')$, then $\ord_{rn'}(q)$ is equal to $rm'$ if
  $\nu_r(\gD)=0$ and equal to $m'$ if  $\nu_r(\gD)>0$.
\item[(3)] For every positive integer $k$ with $\gcd(k,q)=1$, we have $\ord_{kn'}(q)\leq k\cdot \ord_{n'}(q)$.
\end{itemize}
\end{lemma}
\bpf
(1) We have $r\mid (q^{r-1}-1)$, hence if $\gcd(r,n')=1$, then 
\[\ord_{n'r}(q)=\lcm\big(\ord_r(q),\ord_{n'}(q)\big)\mid (r-1)m'.\]

(2) Next, assume that $r$ divides $n'$. By Lemma \ref{lem_basic} we deduce that $rn'\mid (q^{rm'}-1)$, so $\ord_{rn'}(q)$ divides $rm'$. It follows that $\ord_{rn'}(q)$ is either $m'$ or $rm'$. Now if~$r\mid (q^{m'}-1)/n'$, then~$\ord_{rn'}(q)=m'$, and if $r\nmid (q^{m'}-1)/n'$, then $\ord_{rn'}(q)$ is not equal to~$m'$ and so it is equal to $rm'$.

(3) If $k=1$, the claim is trivially true; and if $r$ is a prime with $r\mid k$, then by \Le{lem_ordqn}, parts (1) and (2), we have $\ord_{kn'}(q)\leq r\cdot\ord_{(k/r)n'}(q)$. Now the claim follows by induction.
\epf
With the above preparation, we are now ready to derive necessary and sufficient conditions for condition~\eqref{eqn_nn0rCond} to hold.
\begin{proposition}\label{prop_nn0rChar}
Let $n,n',k$ be positive integers with $n=kn'$, and set $m'=\ord_{n'}(q)$. Write $\pi^*(n'):=\{r\in \pi(n')\, :\, \nu_r(n')=\nu_r(q^{m'}-1)\}$. Then $\ord_{kn'}(q)=km'$ if and only if $\pi(k)\subseteq \pi^*(n')$ and, moreover, if $2\in \pi^*(n')$ and $n'\equiv 2\bmod 4$, then $\nu_2(k)\leq 1$.
\end{proposition}
\begin{proof}
We proceed by induction on~$k$. The claim holds trivially if $k=1$. Assume that $k\geq 2$, and {let $r$ be a prime divisor of $k$}. By applying ~\Le{lem_ordqn}, part~(3) twice, we have
\[\ord_{kn'}(q)=\ord_{(k/r)(rn')}(q)\leq (k/r)\ord_{rn'}(q)\leq k\cdot \ord_{n'}(q)=km',\]
hence $\ord_{kn'}(q)=km'$ if and only if $\ord_{rn'}(q)=rm'$ and $\ord_{kn'}(q)=(k/r)\ord_{rn'}(q)$. By \Le{lem_ordqn} again, $\ord_{rn'}(q)=rm'$ if and only if $r\in\pi^*(n')$. Next, we claim that if $r\in \pi^*(n')$, then
\[\pi^*(rn')=\begin{cases}
    \pi^*(n') \setminus \{2\} & \text{if } r=2 \text{ and } n' \equiv 2 \pmod{4}; \\
    \pi^*(n') & \text{otherwise}.
\end{cases}
\]
Note that if $r$ is odd and $n'$ is even, then $rn'\equiv n'\pmod{4}$, so if the above claim holds, the proposition follows by induction.

To prove this claim, suppose that $r\in \pi^*(n')$, and consider a prime $s$ dividing $rn'$. If $s\neq r$, then $\nu_s(rn')=\nu_s(n')$ and by~\Le{lem_basic}, part~(1), we have $\nu_s(q^{rm'}-1)=\nu_s(q^{m'}-1)$. And when $s=r$, we have $\nu_r(rn')=\nu_r(n')+1$ and by Lemma \ref{lem_basic}, part~(2),  we have $\nu_r(q^{r m'}-1)=\nu_r(q^{m'}-1)+1$ unless $r=2$ and $\nu_2(q^{m'}-1)=1$. When $r=2$, since $r\in \pi^*(n')$, we have $\nu_2(q^{m'}-1)=\nu_2(n')$, hence $\nu_2(q^{m'}-1)=1$ holds if and only if $\nu_2(n')=1$, that is, $n'\equiv 2\bmod 4$. The proof of the claim is now complete.
\end{proof}

\section{Proof of Theorem \ref{thm_mainCode}}\label{sec_mainproof}

This section is devoted to the proof of Theorem \ref{thm_mainCode}. Throughout this section, let $C$ be a {\bf non-standard} irreducible cyclic code of length $n$ over $\F_q$, and let $m=\ord_n(q)$.  {Since $C$ is non-standard, we may assume that $m\geq 2$.} 
We will always assume that the $q$-ary irreducible cyclic code $C$ of length $n$ is non-degenerate. All cyclic codes of length at most $3$ are standard by \cite[Lemma 5.30]{Hollmann2023}, so we have $n\ge 4$. We take the same notation as introduced in Definition \ref{def_UL}; in particular, $\xi$ is an element of order $n$ in $\F_{q^m}^*$ and $\cU_{n,q}=\la \xi\ra$. We have $\F_q(\xi)=\F_{q^m}$ by the fact that $\ord_n(q)=m$. Let $\GL_m(q)$ be the group of invertible $\F_q$-linear transformations of $\F_{q^m}$, and we write $Z$ for its center. We regard $\F_q^*$ as a subgroup of $\GL_m(q)$ which acts on $\F_{q^m}$ via left multiplication, so $Z=\F_q^*$.  Let $\GL_1(q^m)=\{\rho_a:a\in\F_{q^m}^*\}$, where $\rho_a(x)=ax$ for $x\in\F_{q^m}$. Let $\cL(n,q)$ be the stabilizer of $\cU_{n,q}$ in $\GL_m(q)$. Let $\cL_{\textup{st}}(n,q)=\la\sigma,\psi\ra$, where $\sigma(x)=\xi x$, $\psi(x)=x^q$ for $x\in\F_{q^m}$. It is routine to check that  $\cL_{\textup{st}}(n,q)$ is the stabilizer of $\cU_{n,q}$ in $\GL_1(q^m)\rtimes \la\psi\ra$. Moreover, $\cL_{\textup{st}}(n,q)$ acts irreducibly on $\F_{q^m}$.

We identify the coordinate positions of $\F_q^n$ with $\cU_{n,q}$ via $i\mapsto \xi^{i-1}$ for $1\le i\le n$, so that $\PAut(C)$ is equal to $\cL(n,q)$ by \cite[Theorem 5.15]{Hollmann2023}.  The code $C$ is non-standard if and only if $\cL(n,q)\ne\cL_{\textup{st}}(n,q)$. Let $\overline{\cL(n,q)}$ and $\overline{\cL_{\textup{st}}(n,q)}$ be the corresponding quotient groups in $\PGL_m(q)=\GL_m(q)/Z$, and write $\bar{\sigma},\bar{\psi}$ for the respective images of $\sigma,\psi$. Let
\begin{equation}\label{eqn_Xdef}
  X=\{\la\xi^i\ra_{\F_q}:0\le i\le n-1\},
\end{equation}
where $\la\xi^i\ra_{\F_q}$ is the one-dimensional $\F_q$-space spanned by $\xi^i$. We note that $|X|=n/n_0,$ where $n_0:=\gcd(n,q-1)$.

Let $\sigma_1$, $\psi_1$ be the corresponding elements of $\textup{Sym}(X)$ induced by the actions of $\sigma$ and $\psi$ on $X$ respectively.
\begin{lemma}\label{lem_ordbarpsi}
Take notation as above, let $n_0=\gcd(n,q-1)$, {and write $n_1:=n/n_0$}. We have
\begin{itemize}
  \item[(1)]$o(\bar{\sigma})=n_1$, $o(\bar{\psi})=m$.
  \item[(2)]
{$o(\sigma_1)=n_1$, and either $o(\psi_1)=\ord_{n_1}(q)=m$, or the code $C$ is covered by Example \ref{exa_imp}}.
\end{itemize}
\end{lemma}
\begin{proof}
(1) The claim on $o(\bar{\sigma})$ is trivial. Let $m_0=o(\bar{\psi})$, which is a divisor of $o(\psi)=m$. For $x\in\F_{q^m}^*$, we have $\la x^{q^{m_0}}\ra_{\F_q}=\la x\ra_{\F_q}$, i.e., $x^{q^{m_0}-1}\in\F_q^*$. It follows that $q^m-1$ divides $(q^{m_0}-1)(q-1)$. If $m_0\ne m$, then $\frac{q^m-1}{q^{m_0}-1}\ge q^{m_0}+1>q-1$ which is a contradiction. Hence, we have $m_0=m$ as desired.

(2) The claim on $o(\sigma_1)$ is trivial. Let 
$m_1=o(\psi_1)$. Suppose that $m_1<m$. The number $m_1$ is the smallest positive integer $i$ such that $\xi^{q^i-1}\in\F_q^*$, or, equivalently, such that $n\mid (q^i-1)(q-1)$.  Since $\gcd(n_1,(q-1)/n_0)=1$, this is equivalent to $n_1\mid (q^i-1)$, so $m_1=\ord_{n_1}(q)$. We deduce that $m_1\mid m$. So $k:=m/m_1\in\mathbb{N}$, and we conclude that $n$ divides
\[
  \gcd\Big(q^m-1,(q^{m_1}-1)(q-1)\Big)=(q^{m_1}-1)\gcd(q-1,k).
\]
Let $n'=\gcd(n,q^{m_1}-1)$, and let $k':=\frac{n}{n'}$; then $k'\in\mathbb{N}$. Let $a\in\mathbb{N}$ be such that $na=(q^{m_1}-1)\gcd(q-1,k)$. Then $(q^{m_1}-1)\mid n'a$, and hence $k'=(na)/(n'a)$ divides $\gcd(q-1,k)$; in particular, $k'\mid k$. We have $\ord_{n'}(q)\mid m_1$, and, using Lemma \ref{lem_ordqn}, we have that $m=\ord_{k'n'}(q)\le k'\cdot\ord_{n'}(q)\le k'm_1\leq km_1=m$. It follows that $\ord_{n'}(q)=m_1$, and $n/n'=k'=m/m_1=k$. So the triple $(n,n',k)$ satisfies the conditions in \eqref{eqn_nn0rCond}, and hence the code $C$ arises from Example \ref{exa_imp}. This completes the proof.
\end{proof}

By Lemma \ref{lem_ordbarpsi}, we may assume without loss of generality that $\ord_{n_1}(q)=m$ and then $o(\bar{\psi})=m$. Let $V=\F_{q^m}$, which we regard as a vector space over $\F_q$. We write $\textup{End}(V)$ for the set of all linear transformations (endomorphisms) of $V$. We have the following possible embeddings of   $\cL(n,q)$ into classical groups:
\begin{itemize}
  \item[(A)] $\cL(n,q)$ is a subgroup of $\GL_m(q)$ but is not in the similarity group of a non-degenerate unitary, alternating or quadratic form on $V$ if $m\ge3$;
  \item[(B)] $m\ge 3$ and $\cL(n,q)$ is a subgroup of the similarity group of a non-degenerate unitary form on $V$;
  \item[(C)] $\cL(n,q)$ is a subgroup of the similarity group of a non-degenerate quadratic form on $V$;
  \item[(D)] $\cL(n,q)$ is a subgroup of the similarity group of a non-degenerate alternating form on $V$, and it is not in the similarity group of a quadratic form when $q$ is even.
\end{itemize}
Let $\kappa=0$ in the case (A), let $\kappa$ be the corresponding non-degenerate unitary or quadratic form in the cases (B) and (C), and let $\kappa$ be the corresponding non-degenerate alternating form in the case (D). We take the notation $\Omega(V,\kappa)$, $I(V,\kappa)$, $\Delta(V,\kappa)$ for the classical groups associated to $(V,\kappa)$ as in \cite{KL1990}. In particular, $\Delta(V,\kappa)$ is the group of similarities, and it contains $\cL(n,q)$ as a subgroup. We add a bar to this notation to denote their respective quotient images in $\PGL(V)$. We impose the following extra restrictions:
\begin{itemize}
  \item[(1)] For (C) we assume that $m\ge 3$.  (If $m=2$, $\Delta(V,\kappa)$ is a subgroup of $\GL_2(q)$ of Aschbacher class $\cC_2$ or $\cC_3$ according as $\kappa$ is hyperbolic or elliptic, cf. \cite[p.~165]{KL1990}.)
  \item[(2)] For (C) we assume that $\kappa$ is not hyperbolic when $m=4$. (If $\kappa$ is hyperbolic, then $\Delta(V,\kappa)$ is a subgroup of $\GL_4(q)$ of Aschbacher class $\cC_7$, cf. \cite[Proposition 2.9.1 (iv)]{KL1990}. The latter case will be handled in Proposition \ref{prop_C7}.)
  \item[(3)] For (D) we assume that $m\ge 4$, since $\SL_2(q)$ preserves a non-degenerate symplectic form  on $V$ and thus equals $\Sp_2(q)$, cf. \cite[Proposition 2.9.1 (i)]{KL1990}.
\end{itemize}

\begin{proposition}\label{prop_PreProof}
Take notation as above. If either $\Omega(V,\kappa)$ is not quasisimple or $\cL(n,q)$ contains $\Omega(V,\kappa)$, then the claims in Theorem \ref{thm_mainCode} hold.
\end{proposition}
\begin{proof}
For case (A), if $\cL(n,q)$ contains $\SL_m(q)$, then $\cU_{n,q}=\F_{q^m}^*$ by the fact that $\SL_m(q)$ is transitive on nonzero vectors of $\F_{q^m}$. It follows that $n=q^m-1$ and $C$ is permutation equivalent to the code in Example \ref{exa_whole}.  If $(m,q)=(2,2)$, then $n=3$ and the unique non-degenerate binary irreducible code of length 3 is standard. If $(m,q)=(2,3)$, then $n=4$ or $8$. {A $2$-dimensional ternary irreducible code of length $4$ is standard by Lemma~\ref{lem_C2para}, whereas a $2$-dimensional ternary irreducible code of length $8$ is covered by Example~\ref{exa_whole}.} If $(m,q)\not\in\{(2,2),(2,3)\}$, then $\SL_m(q)$ is quasisimple by \cite[Proposition 2.9.2]{KL1990}.
 
For case (B), if $(m,q)=(2,4)$, then $q^m-1=15$ and so $n\in\{5,15\}$. {None of them divides $|\Delta(V,\kappa)|=|\mathrm{GU}(2,2)|=18$, so there are no such codes of those lengths.}  If $(m,q)=(2,9)$, then $q^m-1=80$ and so $n\in 
\{5,10,16,20,40,80\}$. {Since $n$ divides $|\Delta(V,\kappa)|=192$, we deduce that $n=16$.} The triple $(16,8,2)$ satisfies the conditions in \eqref{eqn_nn0rCond} with $q=9$, so the code $C$ of length $16$ over $\F_9$ is covered by Example \ref{exa_imp}. If $(m,q)=(3,4)$, then $q^m-1=63$. Since $n$ divides both $63$ and $|\Delta(V,\kappa)|=648=2^3\cdot 3^4$,
we deduce that $n=9$. The triple $(9,3,3)$ satisfies the conditions in \eqref{eqn_nn0rCond}, so the corresponding code is covered by Example \ref{exa_imp}. If $(m,q)\not\in\{ (2,4), (2,9),(3,4)\}$, then $\SU_m(\sqrt{q})$ is quasisimple. The group $\SU_m(\sqrt{q})$ has $\sqrt{q}$ orbits on $\F_{q^m}^*$, and their sizes are {$(q^{m/2}-(-1)^m)(q^{(m-1)/2}+(-1)^m)$} with multiplicity $1$, $q^{(m-1)/2}(q^{m/2}-(-1)^m)$ with multiplicity $\sqrt{q}-1$, cf. \cite[Lemma 2.10.5]{KL1990} and \cite[Table 1.1 and Theorem 5.41]{Wan1993}. Those orbits are distinguished by the $\kappa$-value of their vectors, and $\Delta(V,\kappa)$ stabilizes the orbit corresponding to $\kappa$-value $0$.  If $\cL(n,q)$ contains $\SU_m(\sqrt{q})$, then by the fact $\gcd(n,q)=1$ we deduce that {$n=(q^{m/2}-(-1)^m)(q^{(m-1)/2}+(-1)^m)$}. We check that $n$ does not divide $q^m-1$ for any $(m,q)$ pairs under consideration, so neither case occurs.

For case (C), we have $m\ge 3$ and $\kappa$ is not hyperbolic if $m=4$. If $(m,q)=(3,3)$, then $q^m-1=26$ and $|\Delta(V,\kappa)|=2^4\cdot 3$, and there are no feasible lengths $n\ge 4$. {If $m$ is odd and $q$ is even, the polar form associated with $\kappa$ has a one-dimensional radical stabilized by $\Delta(V,\kappa)$, and hence by $\sigma$. This is a contradiction to the irreducibility of $\la\sigma\ra$ on $V$.} If $(m,q)\ne (3,3)$, then $\Omega(V,\kappa)$ is quasisimple. If $\cL(n,q)$ contains $\Omega(V,\kappa)$, we deduce that $n=\epsilon(q^{m-1}-1)$, $(q^{m/2}-1)(q^{m/2-1}+1)$ or $(q^{m/2}+1)(q^{m/2-1}-1)$ according as $\kappa$ is parabolic, hyperbolic or elliptic as in case (B). {Here, in the parabolic case, $\epsilon\in\{1,\frac12\}$. Hence $q^{m-1}-1\mid 2n$, which, together with $n\mid q^m-1$, gives $q^{m-1}-1\mid 2(q-1)$, a contradiction.} We check that $n$ divides $q^m-1$ only if $m=4$ and $\kappa$ is elliptic, and the corresponding code is covered by Example \ref{exa_C8}.

For case (D), we have $m\geq 4$. If $(m,q)=(4,2)$, then $q^m-1=15$ and so $n\in\{5,15\}$. The binary codes of those lengths are covered by Examples \ref{exa_repetition} and \ref{exa_whole}, respectively. If $(m,q)\ne (4,2)$, then $\Sp_{m}(q)$ is quasisimple. If $\cL(n,q)$ contains $\Sp_m(q)$, then it is transitive on $\F_{q^m}^*$ by \cite[Lemma 2.10.5]{KL1990}, and correspondingly $C$ is covered by Example \ref{exa_whole}.
\end{proof}

In view of Proposition \ref{prop_PreProof}, we can assume that $\Omega(V,\kappa)$ is quasisimple and $\cL(n,q)$ does not contain $\Omega(V,\kappa)$ for the remaining part of this section. By the main theorem in \cite{Aschbacher1984}, $\cL(n,q)$ is either contained in a maximal subgroup $M$ of $\Delta(V,\kappa)$, where $M$ contains the center $Z$ of $\GL_m(q)$, and belongs to one of the Aschbacher  classes $\cC_1$-$\cC_8$, or $\cL(n,q)$ is a subgroup of Aschbacher class $\cS$.  The subgroup $M$ in the former case is not of Aschbacher class $\cC_8$ by our choice of the form $\kappa$ for cases (A)-(D).

We will now deal with each of the classes $\cC_1$-$\cC_8$, and the class $\cS$. The following result, which relies on the classification of finite permutation groups with a transitive cyclic subgroup in \cite{LiPraeger2012}, plays a key role in handling the Aschbacher class $\cS$.
\begin{lemma}\label{lem_C9faithful}
Let $X$ be as in \eqref{eqn_Xdef}, and let $n_0=\gcd(n,q-1)$, $n_1=n/n_0$. If $\overline{\cL(n,q)}$ is an almost simple group, then it acts faithfully on $X$, and $\la\bar{\sigma}\ra$ induces a regular cyclic subgroup on it. Moreover, we have the following possibilities:
\begin{itemize}
  \item[(a)]$\overline{\cL(n,q)}=M_{23}$, and $n_1=23$;
  \item[(b)]$\overline{\cL(n,q)}= \PSL_{2}(11)$ or $M_{11}$, and $n_1=11$;
  \item[(c)]$\overline{\cL(n,q)}=A_{n_1}$ or $S_{n_1}$ with $n_1\ge 5$, and the former case occurs only when $n_1$ is {odd}.
  \item[(d)]$\PGL_d(q')\leq \overline{\cL(n,q)}\leq \PGaL_{d}(q')$ for some $d\ge2$ and prime power $q'$, and $n_1=\frac{q'^d-1}{q'-1}$;
\end{itemize}
\end{lemma}
\begin{proof}
Since the center $Z$ of $\GL_m(q)$ acts trivially on $X$ and $\cL(n,q)$ stabilizes $\cU_{n,q}$, we have an induced action of $\overline{\cL(n,q)}$ on $X$. Let $K$ be the kernel of this induced action.  Suppose that $\bar{g}$ acts on $X$ trivially for some $g\in\cL(n,q)$. Then there are $\lambda_i$'s in $\F_q^*$ such that $g(\xi^i)=\lambda_i\xi^i$ for $0\le i\le m-1$, i.e., $g$ is a diagonal matrix with respect to the basis $\{1,\xi,\ldots,\xi^{m-1}\}$. It follows that $K$ is a solvable group. Since $K$ is normal in $\overline{\cL(n,q)}$, and $\overline{\cL(n,q)}$ is almost simple, we deduce that $K=\{1\}$. The claim on $\la\bar{\sigma}\ra$ is obvious. Now $\overline{\cL(n,q)}$, which is an almost simple group, acts faithfully on $X$, and contains a regular cyclic subgroup. The claims (a)-(d) now follow from \cite[Theorem 1.2]{LiPraeger2012}. This completes the proof.
\end{proof}

We shall need the following technical lemmas.
\begin{lemma}\label{lem_LabsIrr}
Take notation as above, and let $Z$ be the center of $\GL_m(q)$.
\begin{enumerate}
\item[(1)]The group $\cL(n,q)$ is absolutely irreducible on $\F_{q^m}$, i.e., $C_{\GL(m,q)}(\cL(n,q))=Z$.
\item[(2)] {If there is a field $E\cong\F_{q^b}$ in $\textup{End}_{\F_q}(V)$ containing the subfield $\F_q$ such that $\cL(n,q)$ is $E$-semilinear, then either $C$ is covered by Example~\ref{exa_imp}, or $E\setminus\{0\}\subseteq\GL_1(q^m)$.}
\end{enumerate}
\end{lemma}
\begin{proof}
(1) We have $\F_q(\xi)=\F_{q^m}$ by the fact $\ord_n(q)=m$. By Lemma \ref{lem_irre}, $\cL(n,q)$ is irreducible on $\F_{q^m}$.  An element of $\GL_m(q)$ that commutes with $\la\sigma\ra$ is $\F_{q^m}$-linear, and we similarly deduce that $C_{\GL(m,q)}(\cL_{\textup{st}}(n,q))=Z$.  We thus have  $C_{\GL(m,q)}(\cL(n,q))=Z$, and so $\cL(n,q)$ is absolutely irreducible on $\F_{q^m}$ by \cite{Curtis-Reiner} or \cite[Lemma 2.10.1]{KL1990}.


{(2) Take a primitive element $f$ of $E$. There is a field automorphism $\tau$ of $E$ such that $\sigma(f(\alpha))=f^{\tau}(\sigma(\alpha))$ for each $\alpha\in\F_{q^m}$. Let $k$ be the order of $\tau$. We have $\sigma f\sigma^{-1}=f^\tau$ in $\GL_m(q)$, so $\sigma^i f\sigma^{-i}=f^{\tau^i}$ for $0\le i\le n-1$. Since $o(\sigma)=n$, we have $\tau^n=1$, i.e., $k$ divides $n$. On the other hand, we have $\lambda f^\tau=(\lambda f)^\tau$ for $\lambda\in\F_q$ by the fact $\sigma (\lambda f)\sigma^{-1}=\lambda \sigma f\sigma^{-1}$. Therefore, $\tau$ is an $\bbF_q$-linear automorphism of $E$. It follows that $k$ divides $b$.  Since $\F_{q^m}$ is a vector space over $E$, $b$ divides $m$. We thus have $k\mid m$.

If $k=1$, then $f$ commutes with $\sigma$, and we thus have $f\in C_{\GL_m(q)}(\sigma)=\GL_1(q^m)$. From now on, we assume that $k\ge 2$ and set $n'=n/k$. Let $\overline{\bbF}$ be the algebraic closure of $\bbF_q$, and let $\overline{V}=V\otimes_{\bbF_q}\overline{\bbF}$. The minimal polynomial of $f$ over $\bbF_q$ has $b$ distinct roots $\mu,\mu^q,\ldots,\mu^{q^{b-1}}$ with $\mu\in\overline{\bbF}$, so $\overline{V}$ decomposes into a direct sum of eigenspaces of $f$: $\overline{V}=\oplus_{i=0}^{b-1}\overline{V_i}$, where $\overline{V_i}$ is the eigenspace for $\mu^{q^i}$. Take a nonzero eigenvector $v\in \overline{V_i}$, so that $f(v)=\mu_1 v$ with $\mu_1=\mu^{q^i}$. Take $\tau'=\tau^c\in\la\tau\ra\setminus\{1\}$. Since $f^{\tau'}\in E=\la 1,f,\ldots,f^{b-1}\ra_{\bbF_q}$, we deduce that $f^{\tau'}(v)=\mu'v$ for some $\mu'\in\overline{\bbF}$. Since $f$ and $f^{\tau'}$ have the same minimal polynomial over $\bbF_q$, we have $\mu'=\mu_1^{q^j}$ for some $j\in\mathbb{N}$. It follows that $f(\xi^{-c}v)=\mu_1^{q^j}\xi^{-c}v$, i.e., $\xi^{-c}v$ is also an eigenvector of $f$ with corresponding eigenvalue $\mu_1^{q^j}$. Since $f-f^{\tau'}$ is invertible, we have $\mu_1\ne\mu_1^{q^j}$. We deduce that $\sigma$ permutes the eigenspaces of $f$ in $\overline{V}$ and each $\la\sigma\ra$-orbit has size $k$.

Let $\hat{h}(x)=\sum_{i=0}^{m}u_ix^i$ be the minimal polynomial of $\xi$ over $\bbF_q$, and write $h_1(x^k)=\sum_{j=0}^{m/k}u_{jk}x^{jk}$. For each $v\in\overline{V_i}$, we have $\hat{h}(\sigma)(v)=0$ by the Cayley-Hamilton Theorem. By restricting to $\overline{V_i}$, we obtain $h_1(\sigma^k)(v)=0$. It follows that $h_1(\sigma^k)=0$, and so $\hat{h}(x)=h_1(x^k)$. The minimal polynomial of $\xi^k$ over $\bbF_q$ is thus $h_1(x)$, whose degree is $m/k$. Since $o(\xi^k)=n'$, we deduce that $\ord_{n'}(q)=m/k$. Since $C$ is non-standard, Lemma~\ref{lem_C2para} gives $n\ge6$, and $C$ is covered by Example~\ref{exa_imp}. This proves (2).  }
\end{proof}

\begin{lemma}\label{lem_An1ClassS}
If $n_1\ge 5$ is a prime such that $\ord_{n_1}(q)=n_1-1$ and $n_0$ is a divisor of $q-1$, then the irreducible cyclic code $C$ of length $n=n_0n_1$ over $\F_q$ is obtained from Example \ref{exa_repetition} by applying Constructions \ref{exa_lift} and \ref{exa_ext}. 
\end{lemma}
\begin{proof}
Let $C'$ be the irreducible cyclic code of length $n_1$ over $\F_p$, where $q=p^f$ with $p$ prime. It is the dual of the repetition code of length $n_1$ over $\F_p$ as in Example \ref{exa_repetition}. Since $\ord_{n_1}(p^f)=n_1-1$, we have $\ord_{n_1}(p)=n_1-1$ and $\gcd(f,n_1-1)=1$. The code $C''=C'\otimes\F_q$ is an irreducible cyclic code of length $n_1$ over $\F_q$ by Construction \ref{exa_lift}. The code $C$ is obtained from $C''$ by applying Construction \ref{exa_ext}. This completes the proof.
\end{proof}

\begin{lemma}\label{lem_ordn1q}
If $d\ge 2$, $t\ge 1$ and $n_1=\frac{q^{dt}-1}{q^t-1}$, then $\ord_{n_1}(q)=dt$.
\end{lemma}
\begin{proof}
Let $d_1=\ord_{n_1}(q)$, which divides $dt$ since $q^{dt}\equiv 1\pmod{n_1}$. Suppose to the contrary that $d_1<dt$. Since $n_1$ divides $q^{d_1}-1$, we have $q^{(d-1)t}<n_1\le q^{d_1}-1$. It follows that $(d-1)t<d_1\le\frac{1}{2}dt$, i.e., $d<2$: a contradiction. This completes the proof.
\end{proof}

\begin{lemma}\label{lem_C6n3n0}
Suppose that $q$ is a prime, $m:=\ord_n(q)=2$, and  $n=n_1n_0$ with $n_0=\gcd(n,q-1)$ and $n_1\in\{3,4\}$.
\begin{itemize}
  \item[(1)]If $n_1=3$, then the irreducible cyclic code of length $n$ over $\F_q$ is standard.
  \item[(2)]If $n_1=4$ and $\cL(n,q)$ induces the full symmetric group $\textup{Sym}(X)$ on $X=\{\la\xi^i\ra_{\F_q}:0\le i\le 3\}$, then $q=3$ and the ternary irreducible cyclic code of length $8$ is non-standard and is covered by Example \ref{exa_whole}.
\end{itemize}
\end{lemma}
\begin{proof}
(1) Let $n_1=3$. By the assumptions,  $3$ divides $\frac{q-1}{n_0}(q+1)$ but not $\frac{q-1}{n_0}$, so $q\equiv 2\pmod{3}$. Let $\omega$ be an element of order $3$ in $\F_{q^2}\setminus\F_q$. We then  have $\cU_{n,q}=\cup_{i=0}^2\cU_{n_0,q}\omega^i$. The action of $\cL_{\textup{st}}(n,q)$ on $\{\la 1\ra_{\F_q},\la \omega\ra_{\F_q},\la \omega^2\ra_{\F_q}\}$ is $2$-transitive upon direct check. Take an element $g\in\cL(n,q)$. By replacing $g$ by $hg$ for some $h\in\cL_{\textup{st}}(n,q)$ if necessary, we assume without loss of generality that $g(1)=1$, $g(\omega)=\lambda \omega$ for some $\lambda\in\cU_{n_0,q}$.
Since $\omega^2=-\omega-1$, we have $g(\omega^2)=-1-\lambda \omega$. It follows that $(1+\lambda\omega)^3\in\F_q^*$. After expansion, we obtain $(1+\lambda\omega)^3=(\lambda^3-3\lambda^2+1)+3\lambda{(1-\lambda)}\omega$, so $\lambda=1$. Therefore, $g$ is the identity map. The claim then follows.

(2) Let $n_1=4$ (so $q$ is odd). Take $g\in\cL(n,q)$ such that $\bar{g}$ corresponds to the $3$-cycle $(\la\xi\ra_{\F_q},\la\xi^2\ra_{\F_q},\la\xi^3\ra_{\F_q})$. By replacing $g$ with $\sigma^i g$ for some $i$ if necessary, we assume without loss of generality that $g(1)=1$. We have $g(\xi)=\lambda\xi^2$ for some $\lambda\in\F_q^*$. Let $x^2+ax+b$ be the minimal polynomial of $\xi$ over $\F_q$. We have $ab\ne 0$ by the fact $\xi^2\not\in\F_q^*$. Then $\xi^3=ab+(a^2-b)\xi$,  and $\xi^4=(2b-a^2)a\xi-b(a^2-b)\in\F_q^*$. The latter yields $b=\frac{1}{2}a^2$. We have \[g(\xi^2)=-ag(\xi)-bg(1)=a^2\lambda\xi+(a\lambda-1)b\in\la\xi^3\ra_{\F_q},\] so $ab\cdot a^2\lambda=(a\lambda-1)b(a^2-b)$, i.e., $\lambda=-a^{-1}$. Similarly, $g(\xi^3)=\frac{1}{2}a^2\xi+\frac{3}{4}a^3\in\la\xi\ra_{\F_q}$, so we have $q=p=3$. It follows that $n=8$. The ternary irreducible cyclic code of length $8$ is covered by Example \ref{exa_whole}. This completes the proof.
\end{proof}

We shall complete the proof of Theorem \ref{thm_mainCode} after we establish Propositions \ref{prop_C3}--\ref{prop_C6}. In those propositions, we show that if $\cL(n,q)$ is a geometric subgroup of $\Delta(V,\kappa)$, then the code $C$ must arise as depicted in Theorem \ref{thm_mainCode}. This establishes our main theorem for the geometric subgroups. We will then show that Theorem \ref{thm_mainCode} holds for the Aschbacher class $\cS$ by examining the cases (a)-(d) in Lemma \ref{lem_C9faithful} in detail, and thus complete the proof.
\begin{proposition}\label{prop_C3}
Let $b$ be the largest integer such that there is a field $E\cong\F_{q^b}$ in ${\textup{End}(V)}$ {containing $\F_q$} such that $\cL(n,q)$ is $E$-semilinear, and assume that $b>1$. Then {$C$ is covered by Example \ref{exa_imp}, or $C$} can be constructed from an NSIC code of length $n$ over $\F_{q^b}$ via Construction \ref{exa_expansion}.
\end{proposition}
\begin{proof}
{Assume that $C$ is not covered by Example \ref{exa_imp}}. By Lemma \ref{lem_LabsIrr}, $E$ is the unique subfield of size $q^b$ contained in $\GL_1(q^m)\cup\{0\}$, or equivalently, $\cL(n,q)$ is $\F_{q^b}$-semilinear for the unique subfield of size $q^b$ in $\F_{q^m}$, where $b|m$. Let $m_0=\ord_n(q^b)=\frac{m}{b}$, and let $\GL_{m_0}(q^b)$ be the $\F_{q^b}$-linear invertible transformations of $\F_{q^m}$. We have $\cU_{n,q}=\cU_{n,q^b}$, and $\GL_{m_0}(q^b)\le \GL_m(q)$. It follows from the definitions that $\cL(n,q^b)=\cL(n,q)\cap \GL_{m_0}(q^b)$.  {Since $\cL(n,q)$ is $\F_{q^b}$-semilinear}, we deduce that $\cL(n,q)\le \la\GL_{m_0}(q^b),\psi\ra$. Since $\cL(n,q)$ contains $\psi$, we have
\[
  \cL(n,q)=(\cL(n,q)\cap \GL_{m_0}(q^b))\la\psi\ra= \cL(n,q^b)\la\psi\ra.
\]
On the other hand, we have $\cL_{\textup{st}}(n,q^b)=\cL_{\textup{st}}(n,q)\cap \GL_{m_0}(q^b)$, {which is exactly $\la \sigma, \psi^b\ra$}.  Therefore,   $\cL(n,q^b)\ne \cL_{\textup{st}}(n,q^b)$, and the irreducible cyclic code {$C_0$} of length $n$ over $\F_{q^b}$ is non-standard. This completes the proof.
\end{proof}

\begin{proposition}\label{prop_C5}
Suppose that there is an $\F_q$-basis $e_1,\ldots,e_m$ of $\F_{q^m}$ such that $\cL(n,q)$ stabilizes $\{\la v\ra_{\F_q}:v\in W\setminus\{0\}\}$, where $W=\la e_1,\ldots,e_m\ra_{\F_{q_0}}$ for a subfield $\F_{q_0}$ of $\F_q$. {Let $n'=\gcd(n,q_0^{m}-1)$.} If $q=q_0^t$ with $t>1$, then $\ord_{n'}(q_0)=m$ and $\gcd(m,t)=1$. Moreover, the irreducible cyclic code $C'$ of length $n'$ over $\F_{q_0}$ is non-standard and $C$ is permutation equivalent to the code obtained from $C'$ by applying first Construction \ref{exa_lift} and then Construction \ref{exa_ext}.
\end{proposition}
\begin{proof}
Let $M$ be the stabilizer of
$\{\la v\ra_{\F_q}:v\in W\setminus\{0\}\}$ in $\GL_m(q)$.
Then $M=\GL_m(q_0)Z$, where $Z$ is the center of $\GL_m(q)$.
The group $\cL(n,q)$ is the stabilizer of $\cU_{n,q}$ in $M$.

There is an element $a\in\F_q^*$ such that $\rho_a^{-1}\sigma$ lies in $\GL_m(q_0)$, i.e., it stabilizes $W$. Let $\xi_2=a^{-1}\xi$, and {let} $n_2$ be the multiplicative order of $\xi_2$. Let $f(x)$ be the minimal polynomial of $\xi_2$ over $\F_q$, and let $f_2(x)$ be the characteristic polynomial of the linear transformation $\rho_a^{-1}\sigma$. We have $\F_{q^m}=\F_{q}(\xi)=\F_q(\xi_2)$, so $\deg(f)=m$ and $\ord_{n_2}(q)=m$.  Since $\rho_a^{-1}\sigma$ stabilizes $W$, $f_2(x)$ is in $\F_{q_0}[x]$ and has degree $m$. We have $\rho_a^{-1}\sigma=\rho_{\xi_2}$, so $f(x)=f_2(x)\in\F_{q_0}[x]$ by the Cayley-Hamilton Theorem. It follows that $\ord_{n_2}(q_0)=m$ and $\F_{q_0}(\xi_2)=\F_{q_0^m}$. Since $\ord_{n_2}(q)=\ord_{n_2}(q_0^t)=m$, we deduce that $\gcd(m,t)=1$. 

{Let $K=\F_{q_0^m}$ and $Y=\{\la u\ra_{\F_q}:u\in W\setminus\{0\}\}$. Since $\xi_2W=W$, $W$ is a one-dimensional $K$-space, say, $W=Kv$. Since $\psi$ stabilizes $Y$, we deduce from $\la v^q\ra_{\F_q}\in Y$ that $v^{q-1}=\lambda \mu$ for some $\lambda\in\F_q^*$ and $\mu\in K^*$. We have
\begin{equation}\label{eqn_lambdam}
  1=v^{q^m-1}=(\lambda\mu)^{1+q+\cdots+q^{m-1}}=\lambda^m\mu^{(q_0^m-1)/(q_0-1)}.
\end{equation}
Take $g\in\cL(n,q)$. By assumption there is $\alpha(g)\in\F_q^*$ and an $\F_{q_0}$-linear transformation $\ell:K\rightarrow K$ such that $g(yv)=\alpha(g)\ell(y)v$ for $y\in K$. Since $\gcd(m,t)=1$, we have $\ell(y)=\sum_{i=0}^{m-1}a_iy^{q^i}$ for some $a_i$'s in $K$. Since $g$ is $\F_q$-linear, we deduce that $g(yv)=\alpha(g)v\sum_{i=0}^{m-1}a_iy^{q^i}$ for all $y\in V=\F_{q^m}$. By a change of variable $x=yv$, we have
\begin{align}\label{eqn_expgx}
  g(x)=\alpha(g)\sum_{i=0}^{m-1}a_iv^{1-q^i}x^{q^i}
      =\alpha(g)\sum_{i=0}^{m-1}a_i\mu^{-\frac{q^i-1}{q-1}}\lambda^{-i}x^{q^i}.
\end{align}
Let $\Lambda=\{0\le i\le m-1: a_i\ne 0\}$. It is routine to check  $|\Lambda|=1$ if and only if $g\in\cL_{st}(n,q)$. We assume that $g\in\cL(n,q)\setminus\cL_{st}(n,q)$ from now on. By replacing $g$ by $\psi^jg$ for a suitable $j$, we assume without loss of generality that $0\in\Lambda$. Let $\tilde{g}(x)=g(x^{q^m/q_0})^{q_0}$ which is $\F_q$-linear and stabilizes $\cU_{n,q}$, so that it is in $\cL(n,q)$. By the same arguments as above, we have $\tilde{g}(x)=\alpha(\tilde{g})\sum_{i=0}^{m-1}b_i\mu^{-\frac{q^i-1}{q-1}}\lambda^{-i}x^{q^i}$
for some $b_i$'s in $K$. On the other hand, we have
$\tilde{g}(x)=
\alpha(g)^{q_0}\sum_{i=0}^{m-1}a_i^{q_0}\mu^{-\frac{q^i-1}{q-1}q_0}\lambda^{-iq_0}x^{q^i}$.
By comparing their coefficients, we have
\begin{equation}\label{eqn_lambdaq0}
\alpha(g)^{q_0}\alpha(\tilde{g})^{-1}\lambda^{-i(q_0-1)}=b_ia_i^{-q_0}
\mu^{\frac{q^i-1}{q-1}(q_0-1)}
\end{equation}
for $i\in\Lambda$. Both sides lie in $\F_{q_0}^*$ by the fact $\F_q^*\cap K^*=\F_{q_0}^*$. In particular, $\alpha(g)^{q_0}\alpha(\tilde{g})^{-1}\in\F_{q_0}^*$ by the fact $0\in\Lambda$. It follows that the order of $\lambda^i$ divides $(q_0-1)^2$. It also divides $\gcd(m(q_0-1),q_0^t-1)$ by \eqref{eqn_lambdam}. We have
\[
  \gcd((q_0-1)^2,m(q_0-1),q_0^t-1)=(q_0-1)\gcd(q_0-1,m,t)=(q_0-1),
\]
so $\lambda^{i(q_0-1)}=1$, i.e., $\lambda^i\in\F_{q_0}^*$. Let $a_i'=a_i\mu^{-\frac{q^i-1}{q-1}}\lambda^{-i}\in K$ for $i\in\Lambda$ and put $a_i'=0$ otherwise, so that $g(x)=\alpha(g)\sum_{i=0}^{m-1}a_i'x^{q^i}$ by \eqref{eqn_expgx}.  It follows that $g$ stabilizes $\tilde{Y}:=\{\la u\ra_{\F_q}:u\in K\setminus\{0\}\}$. It is clear that $\cL_{st}(n,q)$ stabilizes $\tilde{Y}$.  We deduce that $\cL(n,q)$ stabilizes $\tilde{Y}$, and the hypothesis in this proposition holds with $K$ in place of $W$. Therefore, we take $W=K$ and correspondingly $v=\lambda=\mu=1$ from now on.

Let $\cU'=\cU_{n,q}\cap K^*$ and $\cU''=\cU_{n,q}\cap \F_{q}^*$, and set $n'=|\cU'|$. We have $n'=\gcd(n,q_0^m-1)$, and recall that $n_0=\gcd(n,q-1)$, cf. Lemma \ref{lem_ordbarpsi}.  We have $\cU_{n,q}\le \F_q^*K^*$, and it is elementary to deduce that $\cU_{n,q}=\cU'\cU''$ by comparing sizes in the quotient group $(\F_q^*K^*)/\F_{q_0}^*$. It follows that $n=\lcm(n_0,n')$.  As $n_0\mid(q-1)$, we obtain $\ord_{n'}(q)=\ord_n(q)=m$. It follows from $\gcd(m,t)=1$ that $\ord_{n'}(q_0)=m$,  so $\cU'=\cU_{n',q_0}$ by definition.   

Take $g\in \cL(n,q)\setminus\cL_{st}(n,q)$, so that $g(x)=\alpha(g)\ell(x)$ for some $\alpha(g)\in\F_q^*$ and $\ell(x)=\sum_{i=0}^{m-1}a_ix^{q^i}\in K[x]$ by \eqref{eqn_expgx}. By the fact $g(1)\in\cU_{n,q}=\cU_{n',q_0}\cU''$, we have $\ell(1)\alpha(g)=yz$ for some $y\in \cU_{n',q_0}$ and $z\in\cU''$. It follows that $z^{-1}\alpha(g)=y\ell(1)^{-1}\in\F_{q_0}^*$. We have $\rho_z\in\cL_{st}(n,q)$, and $\rho_z^{-1} g(u)=z^{-1}\alpha(g)\ell(u)\in K$ for $u\in K$. Therefore, $\rho_z^{-1} g$ stabilizes $K$ and $\cU_{n,q}\cap K=\cU_{n',q_0}$. It is not a monomial in $x$ when restricted to $K$, so it is not in $\cL_{st}(n',q_0)$. We conclude that $\cL_{st}(n',q_0)<\cL({n',q_0})$, i.e., the irreducible cyclic code $C'$ of length $n'$ over $\F_{q_0}$ is non-standard. The code $C$ can be obtained from $C'$ by applying first Construction \ref{exa_lift} and then Construction \ref{exa_ext}, up to permutation equivalence. This completes the proof.}
\end{proof}

\begin{proposition}\label{prop_C2}
Suppose that $\cL(n,q)$ stabilizes an $m_0$-space decomposition
\begin{equation}\label{eqn_C2D}
  \mathcal{D}:\ \F_{q^m}=\oplus_{i=1}^{k} V_i,\; k\ge 2,\;\dim_{\F_q}(V_i)=m_0 \textup{ for }1\le i\le k.
\end{equation}
Then $m_0=\ord_{n'}(q)$ {with $n'=n/k$}, and $C$ is permutation equivalent to the code in Example \ref{exa_imp}.
\end{proposition}
\begin{proof}
We have $m=m_0k$ by comparing dimensions. Let $X'=\{V_1,\ldots,V_k\}$. The direct sum of the $V_i$'s in a $\la\sigma\ra$-orbit on $X'$ is $\la\sigma\ra$-invariant, so $\la\sigma\ra$ is transitive on $X'$ by the fact that $\la\sigma\ra$ acts irreducibly on $V$. We assume without loss of generality that $V_i=\sigma^{i-1}(V_1)$ for $1\le i\le k$. It follows that (1) $\sigma^k$ stabilizes each $V_i$, (2) {$k|n$} by the orbit-stabilizer lemma. Therefore, each $V_i$ is a vector space over $\F_q(\xi^k)=\F_{q^{e}}$, and further $1,\xi,\ldots,\xi^{k-1}$ are linearly independent over $\F_{q^{e}}$.
Since $\F_q(\xi)=\F_{q^m}$, we deduce that $\F_{q^m}=\oplus_{i=1}^{k}\F_{q^{e}}\xi^{i-1}$. By comparing dimensions, we deduce that $m=ek$ and $\dim_{\F_{q^{e}}}(V_i)=1$. Since $m=m_0k$ by \eqref{eqn_C2D}, we have $e=m_0$. Summing up, we have shown that the triple $(n,n',k)$ satisfies the conditions in \eqref{eqn_nn0rCond}. By Lemma \ref{lem_C2para}, $C$ is permutation equivalent to the code in Example \ref{exa_imp}. This completes the proof.
\end{proof}

\begin{proposition}\label{prop_C4}
Suppose that $\cL(n,q)$ stabilizes a tensor product decomposition $V=V_1\otimes V_2$ and its induced action on $\{V_1,V_2\}$ is trivial, where $\dim_{\F_q}(V_i)=m_i$ for $i=1,2$ and $m_1\ge m_2\ge 2$. Then $\gcd(m_1,m_2)=1$, and $C$ can be obtained via Construction \ref{exa_product}.
\end{proposition}
\begin{proof}
We have $m=m_1m_2$ by comparing dimensions and $\cL(n,q)\le\GL(V_1)\circ\GL(V_2)$. There exist $g_1\in\GL(V_1)$ and $g_2\in\GL(V_2)$ such that $\sigma=g_1\circ g_2$. Since $\la \sigma\ra$ is irreducible on $\F_{q^m}$, we deduce that $\la g_i\ra$ is irreducible on $V_i$ for each $i$. The irreducible $\GL(V_1)$-subspaces of $V$ are the $(V_1\otimes u)$'s with $u\in V_2\setminus\{0\}$ by \cite[Lemma 4.4.3]{KL1990}, and $\GL(V_1)$ is the stabilizer in $\GL(V)$ of all those subspaces. We have similar conclusions for $\GL(V_2)$.

We claim that $\gcd(m_1,m_2)=1$. By Lemma \ref{lem_irre} and the fact $\la g_i\ra$ is irreducible on $V_i$, $o(g_i)$ divides $q^{m_i}-1$ for each $i$. It follows that $n=o(\sigma)$ divides
\begin{equation}\label{eqn_C4Nlcm}
N:=\lcm\left(q^{m_1}-1,q^{m_2}-1\right)=\frac{1}{q^{\gcd(m_1,m_2)}-1}(q^{m_1}-1)(q^{m_2}-1).
\end{equation}
In particular, $n$ divides $q^{\lcm(m_1,m_2)}-1$. Since $\ord_n(q)=m$ and $m=m_1m_2$, we deduce that $m_1m_2$ divides $\lcm(m_1,m_2)$, hence $\gcd(m_1,m_2)=1$.

Let $n_0=\gcd(n,q-1)$, and for $i=1,2$ let $n_i=\gcd(n,q^{m_i}-1)$, $d_i=n/n_i$, $d_i'=(q^{m_i}-1)/(q-1)$, $\xi_i=\xi^{d_i}$. We have $\gcd(d_1',d_2')=1$ from $\gcd(m_1,m_2)=1$. Also, $\gcd(n_1,n_2)=n_0$, $\gcd(n_i,q-1)=n_0$, and
\begin{equation}\label{eqn_C4nin0}
 \frac{n_i}{n_0}=\gcd\left(\frac{n}{n_0},d_i'\frac{q-1}{n_0}\right)=\gcd\left(\frac{n}{n_0},d_i'\right).
\end{equation}
We deduce from $n\mid N$ that $n$ divides $\frac{n_1n_2}{n_0}$, cf. \eqref{eqn_C4Nlcm}, and it is now elementary to deduce that $n=\frac{n_1n_2}{n_0}$, i.e., $n=\lcm(n_1,n_2)$. It follows that $\gcd(d_1,d_2)=\frac{nn_0}{n_1n_2}=1$.

We claim that $\ord_{n_i}(q)=m_i$ for $i=1,2$. We have $\ord_{n_i}(q)\le m_i$ by the fact $n_i\mid (q^{m_i}-1)$, and the claim follows from
\begin{equation}\label{eqn_C4ordnq}
  m=\ord_n(q)=\lcm\left(\ord_{n_1}(q),\ord_{n_2}(q)\right)\le\ord_{n_1}(q)\ord_{n_2}(q)\le m_1m_2.
\end{equation}
As a corollary, we have $\F_{q}(\xi_i)=\F_{q^{m_i}}$ and $\F_{q^{m_i}}\cap\cU_{n,q}=\cU_{n_i,q}$.

Let $\sigma_i=\sigma^{d_i}$ for $i=1,2$. There exist integers $a,b$ such that $ad_1+bd_2=1$ by the fact $\gcd(d_1,d_2)=1$. We have $\sigma=\sigma_1^a\sigma_2^b$. 
By considering inside the quotient group of $\GL(V_1)\circ\GL(V_2)$ by its center, we deduce that $g_1\sigma_1^{-a}=g_2^{-1}\sigma_2^{b}$ is in the center. It follows that  $\sigma_i\in \GL(V_i)$ for each $i$, and we assume without loss of generality that $g_1=\sigma_1^a$ and $g_2=\sigma_2^b$. Since $\la g_1\ra$ is irreducible on $V_1$, so is $\sigma_1$ on $V_1$. Since  $\F_{q^{m_1}}=\F_q(\xi_1)$, the irreducible $g_1$-subspaces are the $1$-dimensional $\F_{q^{m_1}}$-subspaces of $V$. The subspaces $(V_1\otimes u)$ with $u\in V_2\setminus\{0\}$ are $g_1$-invariant and have dimension $m_1$, so must be $1$-dimensional $\F_{q^{m_1}}$-subspaces. Similarly, each $v\otimes V_2$ with $v\in V_1\setminus\{0\}$ is a $1$-dimensional {$\F_{q^{m_2}}$}-subspace.

Take a nonzero element $v_i$ in $V_i$ for each $i$, and let $\eta$ be the element of $V=\F_{q^m}$ such that $\eta=v_1\otimes v_2$. Then $V_1\otimes v_2=\F_{q^{m_1}}\eta$, $v_1\otimes V_2=\F_{q^{m_2}}\eta$, and we have the following facts: $\{\xi_1^{j}v_1:0\le j\le m_1-1\}$ is a basis of $V_1$, and
\begin{equation}\label{eqn_C4V1otime}
\{V_1\otimes u: u\in V_2\setminus\{0\}\} =\{\F_{q^{m_1}}\alpha\eta:\alpha\in\F_{q^{m_2}}^*\}.
\end{equation}
Let $\GL_{m_1}(q)$ be the set of invertible $\F_q$-linear transformations of $\F_{q^{m_1}}$. For each $f\in\GL_{m_1}(q)$, we define $h_f^{(1)}\in\GL(V)$ such that $h_f^{(1)}(x)=\sum_{i=0}^{m_1-1}\sum_{j=0}^{m_2-1}a_{ij}f(\xi_1^i)\xi_2^j\eta$,
where $x=\sum_{i=0}^{m_1-1}\sum_{j=0}^{m_2-1}a_{ij}\xi_1^i\xi_2^j\eta$ with $a_{ij}\in\F_{q}$. It is $\F_{q^{m_2}}$-linear. Since $\GL(V_1)$ is the stabilizer of all elements in one side of \eqref{eqn_C4V1otime}, we deduce that $\GL(V_1)=\{h_f^{(1)}:f\in\GL_{m_1}(q)\}$. There is a similar conclusion for $\GL(V_2)$.

Since $\psi$ is in $\GL(V_1)\circ\GL(V_2)$, there are $f_1\in\GL_{m_1}(q)$, $f_2\in\GL_{m_2}(q)$ such that
\[
\sum_{i=0}^{m_1-1}\sum_{j=0}^{m_2-1}a_{ij}f_1(\xi_1^i)f_2(\xi_2^j)\eta
=\sum_{i=0}^{m_1-1}\sum_{j=0}^{m_2-1}a_{ij}\xi_1^{qi}\xi_2^{qj}\eta^q
\]
for all $a_{ij}\in\F_q$. It follows that $f_1(\xi_1^i)f_2(\xi_2^j)=\xi_1^{qi}\xi_2^{qj}\eta^{q-1}$ for $0\le i\le m_1-1$ and $0\le j\le m_2-1$. We deduce that $f_1(\xi_1^i)=\lambda_1 \xi_1^{qi}$, $f_2(\xi_2^j)=\lambda_2 \xi_2^{qj}$ for some $\lambda_i\in\F_{q^{m_i}}^*$ such that $\lambda_1\lambda_2=\eta^{q-1}$.  It follows that $o(\eta)$ divides $s:=(q^m-1,(q^{m_1}-1)(q^{m_2}-1))$. Let $F(x)=\frac{(x^m-1)(x-1)}{(x^{m_1}-1)(x^{m_2}-1)}$ which is a polynomial in $\mathbb{Z}[x]$. Since $\lim_{x\to 1}F(x)=1$, we deduce that $F(q)\equiv 1\pmod{q-1}$, so 
\[s=\frac{1}{q-1}(q^{m_1}-1)(q^{m_2}-1)=\lcm(q^{m_1}-1,q^{m_2}-1).\] 
Therefore, there are $\alpha_i\in\F_{q^{m_i}}^*$ for $i=1,2$ such that $\eta=\alpha_1\alpha_2$. By choosing $v_1,v_2$ properly, we thus assume without loss of generality that $\eta=1$.

We have $\cU_{n,q}=\cU_{n_1,q}\cdot\cU_{n_2,q}$ by the fact $n=\lcm(n_1,n_2)$. Let $\cL(n_i,q)$ be the stabilizer of $\cU_{n_i,q}$ in $\GL(V_i)$ for $i=1,2$, and let $\cL_{\textup{st}}(n_i,q)$ be the corresponding subgroup as conventionally defined. It is elementary to show that $\cL_{\textup{st}}(n,q)=\cL_{\textup{st}}(n_1,q)\circ\cL_{\textup{st}}(n_2,q)$. We claim that  $\cL(n,q)=\cL(n_1,q)\circ\cL(n_2,q)$. The reverse inclusion holds by our description of $\GL(V_1)$ and $\GL(V_2)$. Take $h\in\cL(n,q)$. There are $f_i\in\GL_{m_i}(q)$ for $i=1,2$ such that $h=h_{f_1}^{(1)}\circ h_{f_2}^{(2)}$. For integers $k,l$, we have $h(\xi_1^k\xi_2^l) =f_1(\xi_1^{k})f_2(\xi_2^{l})\in\cU_{n,q}$. By specifying $k=0$ or $l=0$ respectively, we deduce that there is $\lambda\in\F_q^*$ such that $h_{\lambda f_1}^{(1)}\in\cL(n_1,q)$, and $h_{\lambda^{-1}f_2}^{(2)}\in\cL(n_2,q)$. We have $h=h_{\lambda f_1}^{(1)}\circ h_{\lambda^{-1}f_2}^{(2)}$, and this establishes the claim.

Since $C$ is non-standard, at least one of $\cL_{\textup{st}}(n_1,q)\ne\cL(n_1,q)$, $\cL_{\textup{st}}(n_2,q)\ne\cL(n_2,q)$ must hold. Hence $C$ can be constructed via Construction \ref{exa_product}. This completes the proof.
\end{proof}

\begin{proposition}\label{prop_C7}
Suppose that $\cL(n,q)$ stabilizes a tensor decomposition $\F_{q^m}=V_1\otimes\cdots\otimes V_t$, where $\dim(V_i)={m_0}\ge 2$ for each $i$, $t\ge 2$. Then $m_0=t=2$, the triple $(n,n/2,2)$ satisfies
the conditions in \eqref{eqn_nn0rCond} and $C$ appears in Example \ref{exa_imp}.
\end{proposition}
\begin{proof}
We identify $V_i$ with the same vector space $U$. In this case, $\cL(n,q)$ is a subgroup of \[(\GL_{m_0}(q)\circ\cdots\circ\GL_{m_0}(q))\rtimes S_t,\] where $\pi\in S_t$ and $g_1\circ\cdots\circ g_t$ map $v_1\otimes\cdots\otimes v_t$ to $v_{\pi^{-1}(1)}\otimes \cdots\otimes v_{\pi^{-1}(t)}$ and  $g_1(v_1)\otimes\cdots\otimes g_t(v_t)$ respectively, where $v_1,\ldots,v_t$ are vectors in $U$.

An element of the group $\cL(n,q)$ permutes the set $X'=\{V_1,\ldots,V_t\}$. We claim that $\la\sigma\ra$ induces a transitive action on $X'$. Assume without loss of generality that $\{V_1,\dots,V_r\}$ is an orbit, where $1\le r\le t-1$. Let $V_1'=V_1\otimes\cdots\otimes V_r$, $V_2'=V_{r+1}\otimes\cdots\otimes V_t$. Then $\la\sigma\ra$ stabilizes $V_1'\otimes V_2'$, and by the same arguments as in the proof of Proposition \ref{prop_C4} we deduce that {$\gcd(\dim(V_1'),\dim(V_2'))=1$}: a contradiction to $m_0\ge 2$. This proves the claim.

We assume without loss of generality that $\sigma=(g_1\circ\cdots\circ g_t)\pi$, where $\pi=(1,2,\ldots,t)$ is the full cycle. We have $\sigma^t=h_1\circ\cdots\circ h_t$, where $h_1=g_1g_t\cdots g_3g_2$ and $h_i=g_ih_{i-1}g_i^{-1}$ for $2\le i\le t$. The $h_i$'s are pairwise conjugate in $\GL_{m_0}(q)$. We claim that each $\la h_i\ra$ is irreducible on $U$, so that $o(h_i)$ divides $q^{m_0}-1$. Since 
{$\sigma^{tn}=1$, the tensor product of the $h_i^n$ is the identity, so $h_1^n$ is scalar. Its order is prime to the characteristic, as is $n$, so} 
$h_1$ is semisimple on $U$.  Suppose to the contrary that $h_1$ is reducible on $U$, and take an $h_1$-invariant subspace $W$ of $U$ of smallest dimension. Let $m_1=\dim(W)$. Then $m_1\le m_0/2$. The subspace {$V'=W\otimes g_2(W) \otimes (g_3g_2) (W)\otimes\cdots\otimes (g_t\cdots g_2)(W)$} is $\sigma^t$-invariant. Then $W'=\sum_{i=0}^{t-1}\sigma^i(V')$ is $\sigma$-invariant, and it has dimension at most $m_1^tt$. We thus have $m_1^tt\ge m_0^t$ by the fact $\sigma$ is irreducible (and $W'=V$). It follows that $t\geq {(\frac{m_0}{m_1})^t}\geq 2^t$, which is impossible. This establishes the claim.

The order of $\sigma^t$ divides $q^{m_0}-1$ by the previous claim, so $n$ divides {$(q^{m_0}-1)t$}. By Lemma \ref{lem_ordqn} and the fact $\ord_n(q)=m_0^t$, we deduce that $m_0^t\leq m_0t$. This holds only when  $(m_0,t)=(2,2)$. In this case, $n$ divides $2(q^2-1)$. If $n$ is odd, then $n\mid (q^2-1)$, which contradicts $\ord_n(q)=4$. Hence $n$ is even, and so $q$ is odd. Let $n'=n/2$. Since $\ord_n(q)=4$, we deduce that $\ord_{n'}(q)=2$. The triple $(n,n',2)$ satisfies
the conditions in \eqref{eqn_nn0rCond}, so $C$ is covered by Example \ref{exa_imp}. This completes the proof.
\end{proof}

\begin{proposition}\label{prop_C6}
Suppose that there is an $r$-group $R$ of symplectic type of exponent $r\cdot\gcd(2,r)$  such that $R$ is absolutely irreducible on $\F_{q^m}$ and $\cL(n,q)$  normalizes $R$, where $r$ is a prime not dividing $q$. Let $e$ be the smallest integer such that $p^e\equiv 1\pmod{|Z(R)|}$. If $q=p^e$, then the claim in Theorem \ref{thm_mainCode} holds in this case.
\end{proposition}
\begin{proof}
We write $n_0=\gcd(n,q-1)$ and $n_1=n/n_0$. Let $Z_0$ and $Z$ be the centers of $R$ and $\GL_m(q)$ respectively, and take an element $c$ of $Z_0$ of order {$r$}. Let $\bar{R}=R/Z_0$ and write $|\bar{R}|=r^{2k}$. {Then $m=r^k$.} We have $C_{\GL_m(q)}(R)=Z$ by the fact that $R$ is absolutely irreducible, and $N_{\GL_m(q)}(R)=Z.C_{\Aut(R)}(Z_0)$.
{Moreover, $C_{\Aut(R)}(Z_0)$ has kernel $\bar R$ in its induced action on the $\F_r$-vector space $\bar R$, and its image is contained in $\Sp_{2k}(r)$ (or the corresponding orthogonal group if $|Z_0|=2$), which we elaborate here.} We identify $\bar{R}$ with $V'=\F_r^k\times\F_r^k$ such that (1) $[g,h]=c^{f(\bar{g},\bar{h})}$ for a non-degenerate alternating form $f$ on $V'$, (2) if $|Z_0|=2$, then $g^2=c^{Q(\bar{g})}$ for a quadratic form $Q$ on $V'$ with associated bilinear form $f$. For $h\in R$, the inner automorphism defined is {$i_h(g)=hgh^{-1}=gc^{-f(\bar{g},\bar{h})}$} for $g\in R$. For an isometry $\tau$ of $f$ ($Q$ if $|Z_0|=2$), there is $\tau_1\in C_{\Aut(R)}(Z_0)$ such that $\tau(\bar{g})=\overline{\tau_1(g)}$ for each $g\in R$. We have {$\tau_1\circ i_h\circ\tau_1^{-1}=i_{h'}$}, where $h'$ is any element of $R$ such that $\overline{h'}=\tau(\bar{h})$.

The group $\cL(n,q)$ acts on $R$ via conjugation, and this yields an embedding of $\overline{\cL(n,q)}$ in $C_{\Aut(R)}(Z_0)$. Let $b$ be the smallest integer such that $2k\le r^b$. Then $o(\bar{\psi})=m=r^k$, and it divides $r^{1+b}$ by Lemma \ref{lem_exponent}. It follows that $k\le b+1$, so $r^{b-1}<2k\le 2b+2$. If $r=2$, we deduce that $b\le 4$ and $k\le 5$. If $r$ is odd, then we deduce that $b=1$ for $r\ge 7$ and $b\le 2$ for $r\in\{3,5\}$. {If $k=2$ and $r\ge5$, then a Sylow $r$-subgroup of $C_{\Aut(R)}(Z_0)\cong\F_r^4\rtimes\Sp_4(r)$ has exponent $r$, which contradicts $o(\bar{\psi})=r^2$.} It follows that either $k=1$, or $(r,k)$ is one of the pairs $(3,2), (3,3),(5,3),(2,2),(2,3),(2,4)$ or $(2,5)$.

We first consider the case $k=1$ (so $\bar{R}=\F_r\times \F_r$).  In this case, $\ord_n(q)=m=r$, and $\cL_{\textup{st}}(n,q)$ is a subgroup of $(r^{1+2}.\Sp_2(r))\circ Z$, the latter has order $r^3(r^2-1)(q-1)$. Hence $nr|r^3(r^2-1)(q-1)$. It follows that $n$ divides $r^2(r^2-1)n_0$. Let $r'$ be a prime divisor of $\frac{n}{n_0}$. By the facts $\ord_n(q)=r$ and $n_0=\gcd(n,q-1)$, we deduce that $\ord_{r'n_0}(q)=r$. By Lemma \ref{lem_ordqn}, it holds that $r'\ge r$, and equality holds if and only if $r'\mid n_0$ and $r'\nmid \frac{q-1}{n_0}$. If $r$ is odd, all prime divisors of $r^2-1$ {are} at most $\frac{r+1}{2}<r$. We thus have $n_1\mid r$, and so $n_1=r$. The conditions in \eqref{eqn_nn0rCond} are satisfied, so $C$ is covered by Example \ref{exa_imp} if $r$ is odd.
{Assume that $r=2$. Here $|Z_0|\in \{2, 4\}$, so $e\leq 2$.}
In the quotient group $\PGL_m(q)=\GL_m(q)/Z$, $\overline{\cL_{\textup{st}}(n,q)}$ is a subgroup of $A_4$ or $S_4$ by \cite[Proposition 4.6.7]{KL1990}. The element $\bar{\sigma}$ has order $n_1$, and $S_4$ has no elements of order $6$. Therefore, $n_1$ is not a multiple of $6$. Since  $n_1\mid 12$, we deduce that $n_1$ is $2,3$ or $4$. If $n_1=2$, then $C$ is covered by Example \ref{exa_imp} for the same reason as the $r$ odd case. {If $e=2$, then $q=p^2$ and $q\equiv 1\pmod{n_1}$ for $n_1\in\{3,4\}$, which contradicts the assumption $\ord_{n_1}(q)=2$, cf. Lemma \ref{lem_ordbarpsi}. Hence $e=1$ and $q=p$.} If $n_1=3$, then $C$ would be standard by Lemma \ref{lem_C6n3n0}. If $n_1=4$, then $\overline{\cL_{\textup{st}}(n,q)}$ is a Sylow $2$-subgroup of $S_4$ and it contains $O_2(S_4)=\bar{R}$. Also, we have $o(\bar{\sigma})=4$ and $o(\bar{\psi})=2$, and they generate the Sylow $2$-subgroup of $S_4$. Since $C$ is non-standard, we have  $\overline{\cL(n,q)}\ne\overline{\cL_{\textup{st}}(n,q)}$; it follows that $\overline{\cL(n,q)}=S_4$. In its induced action on $X=\{\la\xi^i\ra_{\F_q}:0\le i\le 3\}$, $\la\bar{\sigma},\bar{\psi}\ra$ corresponds to a group of order $8$. Since $S_4$ has no normal subgroup of order $3$, $\overline{\cL(n,q)}$ induces the full symmetric group on $X$. By Lemma \ref{lem_C6n3n0}, we have $q=3$ and $C$ is covered by Example \ref{exa_whole}.

It remains to consider the seven remaining $(r,k)$ pairs. Let $n_1=r_1^{a_1}\cdots r_t^{a_t}$ with $r_i$'s being distinct primes and $a_i>0$ for each $i$. Let $m_i=\ord_{r_i^{a_i}}(q)$ for each $i$, and let $b_i=\gcd(m,r_i^{a_i-1}(r_i-1))$. Since $\phi(r_i^{a_i})=(r_i-1)r_i^{a_i-1}$, we deduce that $m_i$  divides $b_i$. It holds that $q^{m_i}\equiv 1\pmod{r_i^{a_i}}$. We have $\ord_{n_1}(q)=m$ by Lemma \ref{lem_ordbarpsi}, so $m$ is the least common multiple of the $m_i$'s. We examine those pairs one by one.
\begin{itemize}
  \item[(1)]If  $(r,k)=(2,2)$, we have $m=4$, and $|\overline{R}.\Sp(4,2)|=2^8\cdot 3^2\cdot 5$ which is a multiple of $n_1m$. We deduce that $r_i\in\{2,3,5\}$, and
  $b_i\mid4$ if $r_i=2$, while $b_i=2$ or $4$ according as $r_i=3$ or $5$. By using the $4$-dimensional absolutely irreducible $R.\Sp(4,2)$-module over $\F_5$, we realize $R.\Sp(4,2)$ as a subgroup of $\GL_4(5)$. By Magma \cite{Magma} we check that $\bar{R}.\Sp(4,2)$ has a unique conjugacy class of cyclic subgroups of order divisible by $5$, which has order $5$. We deduce that $n_1=5$ or $5\nmid n_1$. In the former case, $n=5n_0$, and the triple $(n,n_0,5)$ satisfies the conditions in Lemma \ref{lem_An1ClassS}. In the latter case, we deduce that $\ord_{2^{a_i}}(q)=4$ if $r_i=2$. By Lemma \ref{lem_basic} and the  fact {$\nu_2(q^2-1)\ge 3$}, we have $\nu_2(q^2-1)=\nu_2(q^4-1)-1\ge a_i-1$, i.e., $q^2\equiv 1\pmod{2^{a_i-1}}$. We thus have $n\mid 2(q^2-1)$, and $(n,n_0',2)$ satisfies the conditions in \eqref{eqn_nn0rCond}, where $n_0'=\gcd(n,q^2-1)$. In both cases, the code $C$ arises as depicted in Theorem \ref{thm_mainCode}.
  \item[(2)] If  $(r,k)=(2,4)$, we have $m=16$, and $|\overline{R}.\Sp(8,2)|=2^{24}\cdot 3^5\cdot 5^2\cdot7\cdot 17$. It follows that $r_i\in\{2,3,5,7,17\}$, and $b_i\mid16$ for $r_i=2$, while $b_i=2,4,2$ or $16$ according as $r_i=3,5,7$ or $17$. If $17$ does not divide $n_1$, then $\ord_{2^{a_i}}(q)=16$ for $r_i=2$, and we deduce that $C$ arises from Example \ref{exa_imp} as in the $(r,k)=(2,2)$ case. Assume that $17\mid n_1$. The group $\Sp_8(2)$ has only one conjugacy class of cyclic groups of order divisible by $17$, which has order $17$ and acts irreducibly on $\bar{R}$. Since $\bar{R}$ has exponent $2$ and $o(\bar{\sigma})=n_1$, we deduce that $n_1=17$ or $34$. If $n_1=17$, then $(n,n_0,17)$ satisfies the conditions in Lemma \ref{lem_An1ClassS} and so $C$ arises as described therein. If $n_1=34$, then $\bar{\sigma}^2$ acts on $\bar{R}$ irreducibly via conjugation. On the other hand, $\bar{\sigma}^{17}$ has order $2$, lies in $\bar{R}$ and commutes with $\bar{\sigma}^2$: a contradiction.
  \item[(3)]The cases where $(r,k)\in\{(3,2), (3,3),(5,3),(2,3),(2,5)\}$ are dealt with similarly, and we need to show that $n\mid r(q^{r^{k-1}}-1)$. It will then follow that the conditions in \eqref{eqn_nn0rCond} are satisfied for $(n,n_0',r)$ with $n_0'=\gcd(n,q^{r^{k-1}}-1)$, and so $C$ arises from Example \ref{exa_imp}: a contradiction. We only give details for $(r,k)=(5,3)$ here. In this case we have $m=125$, and $|\overline{R}.\Sp(6,5)|=2^{10}\cdot 3^4\cdot 5^{15}\cdot7\cdot13\cdot31$, which is a multiple of $n_1m$. The $r_i$'s are in the set $\{2,3,5,7,13,31\}$.  Since $m_i$  divides $\gcd(5^3,r_i^{a_i-1}(r_i-1))$, we deduce that: if $r_i\in\{2,3,7,13\}$, then $q\equiv 1\pmod{r_i^{a_i}}$; if $r_i=31$, then $q^5\equiv 1\pmod{r_i^{a_i}}$. We thus have $m_i=125$ when $r_i=5$. By Lemma \ref{lem_basic} and the fact $r_i^{a_i}\mid (q^{125}-1)$, we have $\nu_5(q^{25}-1)=\nu_{5}(q^{125}-1)-1\ge a_i-1$, so $q^{25}\equiv 1\pmod{r_i^{a_i-1}}$. Since $(q^{125}-1)/(q^{25}-1)\equiv5\pmod{n_0}$ and $n\mid (q^{125}-1)$, we conclude that $n\mid 5(q^{25}-1)$ as desired.
\end{itemize}
This completes the proof.
\end{proof}

We now have all the necessary lemmas and propositions to complete the proof of Theorem \ref{thm_mainCode}. Suppose to the contrary that the NSIC code $C$ of length $n$ over a finite field of characteristic $p$    
 is not covered by Examples \ref{exa_repetition}--\ref{exa_imp}, nor can $C$  be obtained from them by repeated uses of Constructions \ref{exa_ext}--\ref{exa_product}. {Fix the characteristic $p$, write $q=p^f$ and choose such a counterexample with parameters $(m,q,n)$ lexicographically minimal, where $m=\ord_n(q)$. } Suppose that $\cL(n,q)$ is a geometric subgroup of the similarity group $\Delta(V,\kappa)$, where $\kappa$ is one of the forms in (A)-(D). Since $\cL(n,q)$ acts irreducibly on $\F_{q^m}$, it is not of Aschbacher class $\cC_1$. It is not of Aschbacher class $\cC_8$ by the choice of the form $\kappa$. By Proposition \ref{prop_C3} and {the minimality of $m$}, the largest integer $b$ such that $\cL(n,q)$ is a subgroup of $\Gamma\textup{L}_{m/b}(q^b)$ is $1$, i.e., it is not of Aschbacher class $\cC_3$. By Proposition \ref{prop_C5} and {the minimality of $q$ among counterexamples of dimension $m$}, the smallest field over which the $\cL(n,q)$-module $V$ can be realized is $\F_q$, i.e., $\cL(n,q)$ is not of Aschbacher class $\cC_5$. By Proposition \ref{prop_C2},  $\cL(n,q)$ does not stabilize an $m_0$-space decomposition, i.e., it is not of Aschbacher class $\cC_2$. By Propositions \ref{prop_C4} and \ref{prop_C7},  {using the minimality of $m$ in the $\cC_4$ case}, $\cL(n,q)$ does not stabilize a tensor space decomposition of $V$, i.e., it is not of Aschbacher class $\cC_4$ or $\cC_7$. By Proposition \ref{prop_C6},  $\cL(n,q)$ is not of Aschbacher class $\cC_6$. To sum up, $\cL(n,q)$ is not a geometric subgroup of $\Delta(V,\kappa)$. \vspace{5mm}

Suppose that $\cL(n,q)$ is a subgroup of $\Delta(V,\kappa)$ of Aschbacher class $\cS$ from now on. In particular, $\cL(n,q)^{(\infty)}$ is absolutely irreducible on $V=\F_{q^m}$, and $\F_q$ is the smallest field over which this representation can be realized. Here, $\cL(n,q)^{(\infty)}$ is the {terminating member} of its derived series. We take the same notation as in Lemma \ref{lem_C9faithful}, and let $\sigma_1,\psi_1$ be the elements of $\textup{Sym}(X)$ that correspond to $\bar{\sigma},\bar{\psi}$ respectively, where $X$ is defined in (\ref{eqn_Xdef}). We have $o(\sigma_1)=n_1$ and $o(\psi_1)=m=\ord_{n_1}(q)$ by Lemma \ref{lem_ordbarpsi}, and $\psi_1$ normalizes $\la\sigma_1\ra$. By Lemma \ref{lem_ordqn}, we have $m\le n_1$.   By Lemma~\ref{lem_C9faithful}, we have one of the cases (a)-(d). We observe that in all those cases, {$\overline{\cL(n,q)}^{(\infty)}$} is transitive on $X$. 

We first consider case (a) of Lemma \ref{lem_C9faithful}. We have $\overline{\cL(n,q)}^{(\infty)}=M_{23}$ and $n_1=23$. The Schur multiplier and outer automorphism group of $M_{23}$ are both trivial. There is a unique conjugacy class of subgroups of order $23$ in $M_{23}$. The normalizer of a subgroup $C_{23}$ of order $23$ in $M_{23}$ is exactly $C_{23}:C_{11}$, so $m=11$. By \cite[Table 3]{Hiss-Malle} we have $q=2$, so $n=23$. The binary irreducible cyclic code of length $23$ is {the dual of} the Golay code, cf. Example \ref{exa_Golay}.

We next consider case (b) of Lemma \ref{lem_C9faithful}. Both $\PSL_2(11)$ and $M_{11}$ have one conjugacy class of cyclic subgroups of order $11$, and the normalizer of a subgroup of order $11$ has order $55$ in both groups. It follows that $m=5$. The ternary irreducible cyclic code of length $11$ is {the dual of} the Golay code in Example \ref{exa_Golay}, and the ternary irreducible cyclic code of length $22$ is obtained from the former code by Construction \ref{exa_ext}. We assume $q\ne 3$ in the following. The Schur multiplier  of $M_{11}$ is trivial, so $\overline{\cL(n,q)}^{(\infty)}=M_{11}$. By \cite[Table 3]{Hiss-Malle}, we must have $q=3$, which we have excluded from consideration. The Schur multiplier of $\PSL_2(11)$ has order $2$, so $\cL(n,q)^{(\infty)}=\SL_2(11)$ or $\PSL_2(11)$. We have $p\ne 11$ by the fact $\gcd(n,q)=1$ and $n/n_0=11$. By \cite{Hiss-Malle}, we deduce that $\cL(n,q)^{(\infty)}=\PSL_2(11)$ and there is a unique such $5$-dimensional absolutely irreducible representation of $\PSL_2(11)$ up to weak equivalence. Moreover, $\F_q$ is the smallest extension field of $\F_p$ that contains a root of $x^2+x+3=0$ in $\F_q$. If $p=3$, then $q=3$, which we have excluded from consideration.
If $p=2$, then {$q=4$ and} all the $\PSL_2(11)$-orbits on $V\setminus\{0\}$ have size larger than $n_1(q-1)=33$, contradicting $|\cU_{n,q}|=n\in \{11,33\}$.  Now we can assume that $p\ne 2,3$. Let $c$ be a root of $x^2+x+3=0$ in $\F_q$. There are two conjugacy classes of subgroups of index $11$ in $\PSL_2(11)$ and they are dealt with similarly, so we give details for one class only. There is a basis of $V$ with respect to which we have $\PSL_2(11)=\la A,B\ra$, {where the following matrices act on column vectors from the left: }
\[
A:=\begin{pmatrix}
0&1&0&0&0\\
1&0&0&0&0\\
0&0&0&0&1\\
1&-1&c&1&-c\\
0&0&1&0&0
\end{pmatrix}^{\mathsf T},\qquad
B:=\begin{pmatrix}
0&0&0&1&0\\
0&0&1&0&0\\
0&-1&-1&0&0\\
c+1&0&0&{-1}&c+2\\
1&0&0&-1&1
\end{pmatrix}^{\mathsf T}.
\]
{Set $Y=(BA)B(BA)^{-1}$.} Then  $H=\la A, Y\ra$ is the stabilizer of $\la 1\ra_{\F_q}\in X$ in this group.\footnote{We can take $H=\la A,(B^2A)B(B^2A)^{-1} \ra$ for the other conjugacy class.} We have $A^2=I_5$ and $B^3=I_5$, where $I_5$ is the  $5\times 5$ identity matrix. 
Then the $1$-eigenspace and $(-1)$-eigenspace of $A$ are \[\la (1,0,0,1,-c)^\top,\,(0,1,0,-1,c)^\top,\, (0,0,1,0,1)^\top\ra_{\F_q} \text{ and }\la (1,-1,0,0,0)^\top,(0,0,1,0,-1)^\top\ra_{\F_q},\] respectively. Similarly, the $1$-eigenspace and $(-1)$-eigenspace of $YAY^{-1}$ are
\begin{align*}
&\la(1,0,0,1,-c)^\top,(0,1,0,-c-1,-1)^\top,(0,0,1,-2,c)^\top\ra_{\F_q} \text{ and }\\
&{\la (0,-c-2,-2,1,-c)^\top,\,
    (1,-2c-1,-c-3,c+2,2-c)^\top\ra_{\F_q}
}
\end{align*}
respectively. Their common eigenvectors form the $1$-dimensional subspace {$\la (1,0,0,1,-c)^\top\ra_{\F_q}$}, so it equals $\la 1\ra_{\F_q}$. {Set $v=(1,0,0,1,-c)^\top$.} A direct computation shows that  $\la v\ra_{\F_q}$ is not fixed by $Y^{2}AY^{-2}$: a contradiction. Hence this case does not occur.

For case (c) of Lemma \ref{lem_C9faithful}, we have $\overline{\cL(n,q)}=A_{n_1}$ or $S_{n_1}$, and the former case only occurs when $n_1$ is {odd}. The Schur multiplier of $A_{n_1}$ has order $6$ if $n_1\in\{6,7\}$ and has order $2$ otherwise, cf. \cite[p.~173]{KL1990}. By Lemma \ref{lem_ordbarpsi}, we have $\ord_{n_1}(q)=m$, and thus $m\mid \phi (n_1)$,  where $\phi$ is the Euler totient function. We have $n\le n_1(q-1)$ by the same lemma. If $n_1\ge 5$ is a prime such that $\ord_{n_1}(q)=n_1-1$, then the code $C$ is obtained from Example \ref{exa_repetition} by applying Constructions \ref{exa_lift} and \ref{exa_ext} by Lemma \ref{lem_An1ClassS}. We thus assume that $\ord_{n_1}(q)\ne n_1-1$ if $n_1$ is a prime. If $V$ is the fully deleted module of $A_{n_1}$, then $q=p$ and $m=n_1-1$. We deduce from $\ord_{n_1}(p)=n_1-1$ that $n_1$ is a prime and so the code $C$ is covered by Lemma \ref{lem_An1ClassS}.

If $n_1\ge 10$ and $\cL(n,q)^{(\infty)}=A_{n_1}$, then there is no feasible $A_{n_1}$-module $V$ with $m\le\phi (n_1)$ by \cite[Theorem 5.3.5]{KL1990}. If $n_1\ge 12$ and $\cL(n,q)^{(\infty)}=2.A_{n_1}$, then the dimension $m$ of the faithful $2.{A_{n_1}}$-module $V$ is at least  $2^{\lfloor(n_1-3)/2\rfloor}$ by \cite{Kleshchev12}. It follows that $2^{\lfloor(n_1-3)/2\rfloor}\le \phi(n_1)$, but this does not hold for $n_1\ge 12$: a contradiction. If $n_1\in\{10,11\}$ and 
$\cL(n,q)^{(\infty)}=2.A_{n_1}$, then $(n_1,q,m)=(10,5,8)$ and there is a unique such representation up to weak equivalence by \cite[Table 3]{Hiss-Malle}. By using the {specific irreducible representation} in \cite{ATLASbook} or \cite{ATLASrep}, we check in Magma \cite{Magma} that all $2.A_{10}$-orbits have lengths larger than $n_1(q-1)=40$, so this case does not occur. If $n_1=6,8$ or $9$, then $m$ divides $\phi(n_1)=2,4$ or $6$ correspondingly.  By \cite[Proposition 5.3.7]{KL1990} there is such an absolutely irreducible projective representation of $A_{n_1}$ only if $(n_1,p)=(6,3)$ or $(8,2)$. Both cases are excluded by the facts that $\gcd(n,p)=1$ and $n_1\mid n$.

It remains to consider the case where $n_1\in\{5,7\}$. If $n_1=7$, we have $m\in \{2,3\}$ by the fact $m\mid \phi(7)$ and $m\ne n_1-1$. By \cite{Hiss-Malle}{Table 3}, the case $m=2$
does not occur, and the only feasible case has $m=3$, $\cL(n,q)^{(\infty)}=3.A_7$ and $q=5^2$, which is a unique such feasible module {up to weak equivalence}. We check in Magma \cite{Magma} that the $3.A_7$-orbits on $V\setminus\{0\}$ all have size larger than $7(q-1)=168$, so this case does not occur. If $n_1=5$, we have $m=2$ by the fact $m\mid\phi(5)$ and $m\ne n_1-1$. We have the following candidates: (1) $m=2$, $\cL(n,q)^{(\infty)}=A_5=\SL_2(4)$, $q=4$ and $V$ is the natural module of $\SL_2(4)$, (2) $m=2$, $\cL(n,q)^{(\infty)}=2.A_5$, $p\not\in\{2,5\}$, $\F_q$ is the smallest extension field of $\F_p$ that contains a root of $x^2+x-1$. For (1), the corresponding code $C$  is covered by Example \ref{exa_whole}. For (2), we have $q\equiv 4\pmod{5}$ by the fact $\ord_5(q)=2$. Let $\zeta$ be an element of order $5$ in $\F_{q^2}$, and let $b=\zeta+\zeta^4$. We have $b^2+b-1=0$, and $b$ is in $\F_q$. The minimal polynomial of $\zeta$ over $\F_q$ is $x^2-bx+1$. We have $X=\{\la\zeta^i\ra_{\F_q}:0\le i\le 4\}$, and take an element $g\in 2.A_5$ such that $\bar{g}$ induces the permutation $(\la\zeta\ra_{\F_q},\la\zeta^2\ra_{\F_q})(\la\zeta^3\ra_{\F_q},\la\zeta^4\ra_{\F_q})$. By replacing $g$ with $\sigma^ig$ for {some integer $i$} if necessary, we may assume that $g(1)=1$. We have $g(\zeta)=\lambda\zeta^2$ for some $\lambda\in \F_q^*$.  It follows that $g(\zeta^2)=g(b\zeta-1)=\lambda b^2\zeta-\lambda b-1$. We deduce from $g(\zeta^2)\in\la \zeta\ra_{\F_q}$ that $\lambda b=-1$. Similarly, $g(\zeta^3)=b\zeta-1-b\in \la b-\zeta\ra_{\F_q}$ and we deduce that $2b=0$ in $\F_q$, i.e., $p=2$: a contradiction. This completes the analysis of case (c).

We then consider case (d) of Lemma \ref{lem_C9faithful}.  We have $n_1=\frac{q'^{d}-1}{q'-1}$, where $q'=p_1^{f_1}$ for a prime $p_1$ and an integer $f_1$. {If $(d,q')=(2,8)$, then $m=\ord_9(q)$ divides $df_1=6$. If $(d,q')\ne(2,8)$, then 
there is a unique cyclic subgroup of order $n_1$ in $\PGL_{d}(q')$ up to conjugacy and its normalizer in $\PGaL_{d}(q')$ has order $df_1n_1$ by \cite{Jones} and \cite{Penttila2016Singer}. It follows that $m$ divides $df_1$.}  We do not need to consider the case $(d,q')=(2,4)$ by the equivalence of the natural permutation actions of $\PSL_2(4)$ and $A_5$ on five points.

We first consider the case where $p_1=p$ and $\SL_d(q')$ is a cover group of $\cL(n,q)^{(\infty)}$.  We regard $V$ as an absolutely irreducible $\SL_d(q')$-module. By \cite[Proposition 5.4.6]{KL1990}, we deduce that $f\mid f_1$ and $m=d_1^{f_1/f}$ for an irreducible $\SL_d(q')$-module of dimension $d_1$ over the algebraically closed field $\overline{\F_q}$. By \cite[Theorem 1.11.5]{BHR407}, we have $d_1\ge d$. Let $t=\frac{f_1}{f}\in\mathbb{N}$.  We have $m=\ord_{n_1}(q)=dt$ by Lemma \ref{lem_ordn1q}, so $dt=d_1^t\ge d^t$. If $t>1$, this holds only if $d=d_1=t=2$. Then $m=4$, and we have {$\cL(n,q)^{(\infty)}\cong \PSL_2(q^2)\cong \Omega^-_4(q)$}. The module $V$ has a non-degenerate $\Omega^-_4(q)$-invariant quadratic form, and the corresponding code $C$ is covered by Example \ref{exa_C8} by our analysis at the beginning of this section. If $t=1$, then $m=d$ and $q=q'$, and $V$ is the natural module of $\SL_m(q)$. We have  $\cL(n,q)^{(\infty)}=\SL_m(q)$, and $C$ is covered by Example \ref{exa_whole}.

We next consider the case where $p_1=p$ and  $\cL(n,q)^{(\infty)}$ does not have $\SL_d(q')$ as a cover group. By \cite[Table 5.1D]{KL1990}, we have $(d,q')\in\{(2,9),(3,2),(3,4),(4,2)\}$. If $(d,q')=(4,2)$, then $n_1=15$ and $m\mid 4$. By \cite[Theorem 1.11.5]{BHR407} we have $m\ge4$, so $m=4$. It follows that $V$ is the natural module of $\SL_4(2)$ and $q=2$, but then $\cL(n,q)^{(\infty)}=\SL_4(2)$: a contradiction.  If $(d,q')=(2,9)$, then $n_1=10$,  $m\mid 4$, and it follows that $\cL(n,q)^{(\infty)}=\PSL_2(9)=A_6$ and $m=4$. It has $\SL_2(9)$ as a cover group: a contradiction.  If $(d,q')=(3,2)$, then $\cL(n,q)^{(\infty)}=\SL_2(7)$, $n_1=7$ and $m=3$. If $(d,{q'})=(3,4)$,  $n_1=21$, $m\mid 6$. There is no feasible absolutely irreducible module $V$ in either of those two cases.

We then consider the case where $p_1\ne p$. If $d\ge 3$ and $(d,q')\not\in\{(3,2),(3,4)\}$, then $m\ge q'^{d-1}-1$ by  \cite{La-Se} (see also \cite[Table 5.3.A]{KL1990}). The {resulting inequality $p_1^{(d-1)f_1}-1\le df_1$} holds for no such triple $(p_1,d,f_1)$ with $p_1$ prime and $d\ge 3$: a contradiction. If $d=2$ and $q'\ne 9$, then we have \[\frac{q'-1}{\gcd(2,q'-1)}\le m\le 2f_1,\] which also leads to a similar contradiction unless $q'=5$.

We first show that $(d,q')\not\in\{(2,5),(2,9),(3,4)\}$. If $(d,q')=(2,5)$, then $m=2$, $n_1=6$ and $p\ne 2$ by the fact $\gcd(n,q)=1$. In this case, we have $\overline{\cL(n,q)}=\PGL_2(5)$ by the fact $\PSL_2(5)$ does not contain an element of order $6$. For the $2$-dimensional absolutely irreducible $2.\PSL_2(5)$-module $V$, the normalizer of $\PSL_2(5)$ in $\PGL(V)$ is $\PSL_2(5)$ by \cite[Proposition 4.5.1]{BHR407}. This contradicts the fact that $\overline{\cL(n,q)}=\PGL_2(5)$, so we have $(d,q')\ne (2,5)$. If $(d,q')=(2,9)$, then $n_1=10$, $m\mid 4$, and $p\ne 2$ by the fact $\gcd(n,q)=1$. By \cite{Hiss-Malle}, we have $\cL(n,q)^{(\infty)}=\SL_2(9)$, $m=4$, $q=p$ and $V$ has a non-degenerate alternating form $\kappa$. Among the three conjugacy classes of subgroups of index $2$ in $\PGaL_2(9)$, exactly one, say $G_0$, has an element $h$ of order $10$. The group $G_0$ has only one conjugacy class of cyclic subgroups of order $10$, and $|N_{G_0}(\la h\ra)|=20$, while $\la\bar{\sigma},\bar{\psi}\ra$ has order $40$. We deduce that $\overline{\cL(n,q)}=\PGaL_2(9)$. By \cite[Proposition 4.5.10]{BHR407}, the normalizer of $\PSL_2(9)$ in $\overline{\Delta(V,\kappa)}$ has order $2\cdot|\PSL_2(9)|$: a contradiction to $\overline{\cL(n,q)}=\PGaL_2(9)$. Hence $(d,q')\ne (2,9)$. If $(d,q')=(3,4)$, then $n_1=21$, $m\mid 6$, and $p\ne 3$ by the fact $\gcd(n,q)=1$. By \cite[Table 3]{Hiss-Malle}, we have $m=6$, $\cL(n,q)^{(\infty)}=6.\PSL_3(4)$, and {$\F_q$} is the splitting field of $x^2+x+1$ over $\F_p$. Moreover,  $V$ has a non-degenerate unitary form $\kappa$ if $p\equiv 5,11\pmod{12}$; we have $\kappa=0$ otherwise. By arguing as in the previous two cases, we deduce that $\overline{\cL(n,q)}=\PGaL_3(4)$. The normalizer of $\PSL_3(4)$ in $\overline{\Delta(V,\kappa)}$ has order $2\cdot|\PSL_3(4)|$ by \cite[Proposition 4.5.18]{BHR407}: a contradiction to $\overline{\cL(n,q)}=\PGaL_3(4)$. Hence $(d,q')\ne (3,4)$.

It remains to consider the case $(d,q')=(3,2)$. We have $n_1=7$ and $m=d=3$ by the fact $m\mid df_1$. We have $p\ne7$ {since $7=n_1\mid n$ and $\gcd(n,q)=1$}. 
By \cite[Table 2 (c)]{Hiss-Malle}, we have $\cL(n,q)^{(\infty)}=\PSL_3(2)\cong \PSL_2(7)$. The field $\F_q$ is the splitting field of $x^2+x+2$ over $\F_p$. We have $q\equiv2,4 \mod 7$ by the fact $\ord_7(q)=3$.  Let $c$ be a root of $x^2+x+2=0$ in $\F_q$. There are two conjugacy classes of subgroups of index $7$ in $\PSL_3(2)$ and they are dealt with similarly, so we give details for one class only. {Here $p$ should be odd, since we are in the case $p \ne p_1 = 2$.} There is a basis of $V$ with respect to which we have $\PSL_3(2)=\la A,B\ra$, {where the matrices act on column vectors from the left:
\[
A:=\begin{pmatrix}
1&-1-c&c\\
0&-1&0\\
0&0&-1
\end{pmatrix}^{\mathsf T},\qquad
B:=\begin{pmatrix}
0&1&0\\
0&0&1\\
1&0&0
\end{pmatrix}^{\mathsf T}.
\]
Set $Y=(BA)B(BA)^{-1}$. Then $H=\la A,\, Y\ra$}  is the stabilizer of $\la 1\ra_{\F_q}\in X$ in $\PSL_3(2)$.\footnote{{For the other conjugacy class, we can take $H=\la A,\,(B^{-1}A) B(B^{-1}A^{})^{-1}\ra$}.} We have $A^2=I_3$ and $B^3=I_3$, where $I_3$ is the $3\times 3$ identity matrix. 
 The $(-1)$-eigenspace of $A$ is $\la (0,1,0)^\top,(0,0,1)^\top\ra_{\F_q}$ and the $1$-eigenspace is {$\la (1, -(c+1)/2, c/2)^\top\ra_{\F_{q}}$}.  The common eigenvectors of $A$ and {$YAY^{-1}$} form the $1$-dimensional subspace  {$\la (0,1,-(c+1)/2)^\top\ra_{\F_q}$}, so it corresponds to $\la \mathbf{1}\ra_{\F_q}$. It is not stabilized by {$Y^2A{Y^{-2}}$}: a contradiction. Hence this case does not occur.\medskip

To summarize, we have now analyzed the four cases in Lemma \ref{lem_C9faithful}. It follows that Theorem \ref{thm_mainCode} holds when $\cL(n,q)$ is a subgroup of $\Delta(V,\kappa)$ of Aschbacher class $\cS$. This completes the proof of Theorem \ref{thm_mainCode}.

\section{Applications}\label{sec_application}
\subsection{Linear recurring sequence subgroups}\label{sec_lrcc} 
 NSIC codes are closely related to\linebreak non-standard linear recurring sequence subgroups, cf. \cite{Hollmann2023}.  Brison and Nogueira investigated them in a series of papers \cites{Brison2003,Brison2008,Brison2009,Brison2010,Brison2014}.
 
A sequence ${\bf s}=s_0,s_1,\ldots$ in the algebraic closure $\overline{\F_q}$ is termed  a linear recurring sequence of order $m$ if it satisfies a (homogeneous) linear recurrence relation of the form
\begin{equation}\label{eqn_sequence} s_k=a_{m-1}s_{k-1}+\cdots+ a_1  s_{k-m+1}+a_0 s_{k-m}
\end{equation}
for all integers $k\geq m$, where $m\geq 1$, $a_0\in\F_q^*$ and $a_1,\ldots,a_{m-1}\in\F_q$. The monic polynomial  $f(x)=x^m-a_{m-1}x^{m-1}-\cdots-a_1x-a_0$ in  $\F_q[x]$ with $f(0)=-a_0\ne 0$ is called the characteristic polynomial of the recurrence relation \eqref{eqn_sequence},   and such a sequence ${\bf s}$  is  referred to as  an $f$-sequence. We say that such an  $f$-sequence ${\bf s}$ in $\overline{\F}_q$ is {\em cyclic\/} if there exists $\alpha\in \overline{\F_q}$ such that $s_{k+1}=\alpha s_{k}$ for all $k\geq 0$.  In what follows, we assume that $f(x)$ is the minimal polynomial over $\F_q$ of a generator $\xi$ of $\cU_{n,q}$. 
For further information on linear recurrence relations and linear recurring sequences over finite fields, please refer to \cites{LRS1965,FiniteField,LRS,Zierler1959LRS}. As a consequence,
 an $f$-sequence ${\bf s}=(s_0,s_1,\ldots, s_{n-1})$ of period $n$ can represent the finite multiplicative subgroup $\cU_{n,q}\leq (\overline{\F_q})^*$; in this case, we refer to $\cU_{n,q}$ as an $f$-subgroup. Further, we say that $\cU_{n,q}$ is a non-standard $f$-subgroup if a non-cyclic $ f$-sequence can represent it, and we can call $\cU_{n,q}$ a standard  $f$-subgroup otherwise.  There is a one-to-one correspondence between non-standard pairs and non-standard linear recurring sequence subgroups, see \cite[Section 4]{Hollmann2023}.
 To sum up, the classification of  NSIC codes is equivalent to the classification of non-standard linear recurring sequence subgroups.  This seems to have been first investigated by Somer \cites{Somer1972,Somer1977}.

By using our main result, Theorem \ref{thm_mainCode}, we give a classification of non-standard non-degenerate irreducible cyclic codes for the case $m=2$. This concludes the study of linear recurring sequence subgroups with $m=2$ initiated and systematically studied by Brison and Nogueira \cites{Brison2003,Brison2008,Brison2009}.
\begin{corollary}
Let $C$ be a non-degenerate irreducible cyclic code of length $ n$ over the field $\F_q$. If $m=\ord_n(q) = 2$ and $C$ is non-standard, then $C$ {falls into} one of the following cases:
\begin{itemize}
    \item[(1)]  The code has length $n = q^2 - 1>3$  (as shown in Example~\ref{exa_whole}), and its permutation automorphism group is $\GL_2(q)$;
    \item[(2)] The code has length $ n = 2n_0 > 4$, where $ n_0 \mid (q-1) $ and $ \frac{q-1}{n_0}$ is odd. Its permutation automorphism group is  $C_{n_0}^2\rtimes C_2$.
    \item[(3)] The code has length $n=k(q_0^2-1)$, where $q=q_0^t$ for an odd integer $t>1$ and $k\mid \frac{q-1}{q_0-1}$. Its permutation automorphism group contains a subgroup isomorphic to $C_{k(q_0-1)}\circ\GL(2,q_0)$.
\end{itemize}
\end{corollary}
\begin{proof}
Suppose the code $C$ of dimension $m=2$ is one of the cases in Examples \ref{exa_repetition}--\ref{exa_imp}. The code $C$ in Example \ref{exa_whole} has length $q^2-1>3$ as given in Case (1). The code is not in Examples \ref{exa_repetition}, \ref{exa_C8} and \ref{exa_Golay} by the fact that $m=2$. Also, if the code is in Example \ref{exa_imp}, then $r=2$ and the triple $(n,n',2)$ satisfies the condition \eqref{eqn_nn0rCond}, so it occurs in Case (2). The group structure of $\cL(n,q)$ is described in the proof of Lemma \ref{lem_C2para}.  To be specific, let $\xi$ be a primitive $2n'$-th root of unity in $\F_{q^2}$. Then $\cL(n,q) $ stabilizes $\{\la 1\ra_{\F_{q}}, \la\xi\ra_{\F_q}\}$ {setwise}.  Consider the action of $g \in \mathcal{L}(n, q)$ on $\cU_{n,q}$. Without loss of generality assume that $g(1)=1,\, g(\xi) = \lambda_1 \xi $ for some $\lambda_1 \in \mathbb{F}_q^* \cap \langle \xi \rangle$.   Then $g(\xi^i)=\xi^i$ or $\lambda_1\xi^i$ according as $i$ is even or odd. This implies that  $\cL(n,q)=(C_{n_0})^2 \rtimes C_2$.

Let $C$ be an NSIC code over a finite field of characteristic $p$ that is obtained from Cases (1) and (2) by repeated uses of Constructions \ref{exa_ext}--\ref{exa_product}. By the previous paragraph, we may assume that $C$ is not in Examples \ref{exa_repetition}--\ref{exa_imp}. Observe that the output of Construction \ref{exa_expansion} or \ref{exa_product} has a larger dimension than the input, so $C$ is obtained by repeated uses of Constructions \ref{exa_ext}  and $\ref{exa_lift}$.   The output of Construction \ref{exa_ext} on Cases (1) and (2) is still in Cases (1) and (2). By Proposition \ref{prop_C5}, we assume without loss of generality that the code $C$ is obtained by applying first Construction \ref{exa_lift} and then Construction \ref{exa_ext}. Then there is a prime power $q_0$ such that $q=q_0^t$ with odd $t$. Since $C$ is not in Example \ref{exa_imp}, the input is not in Case (2). It follows that the original code $C'$ is from Case (1), which has length $q_0^2-1$.   Then the resulting code has length $k(q_0^2-1)$, where $k\mid \frac{q-1}{q_0-1}$. The group structure of $\cL(n,q)$ follows from the proof of Proposition \ref{prop_C5}, which is omitted here. This completes the proof.
\end{proof}

\subsection{The Schmidt-White conjecture}\label{sec:NSIC-2wt}  

In this {subsection}, we verify the Schmidt-White conjecture for NSIC codes. Suppose that $q=p^f$ with $p$ prime, and let {$n>1$} be a divisor of $q^m-1$ {with $m=\ord_n(q)$}. Take a primitive $n$-th root of unity $\xi$ in $\F_{q^m}$, and set $\cU_{n,q}=\{\xi^i:0\le i\le n-1\}$. Let $u=\frac{q^m-1}{n}$.  For each $\alpha\in\F_{q^m}$, define the vector
\begin{equation}
   \bfc_{\alpha}=(\tr_{\F_{q^m}/\F_q}(\alpha),\tr_{\F_{q^m}/\F_q}(\alpha\xi),\ldots,\tr_{\F_{q^m}/\F_q}(\alpha\xi^{n-1})).
\end{equation}
Let $C(q,m,u)=\{\bfc_\alpha:\alpha\in\F_{q^m}\}$, cf. \cite[Definition 2.2]{Schmidt2002}. All {non-degenerate} irreducible cyclic codes of length $n$ over $\F_q$ are {permutation equivalent} to $C(q,m,u)$. For $\bfc=(c_0,\ldots,c_{n-1})$, we write $w_H(\bfc)=\#\{0\le i\le n-1:c_i\ne0\}$ for the Hamming weight of $\bfc$. We have 
\begin{equation*}
 w_H(\bfc_\alpha)=n-\#\{x\in\cU_{n,q}:\tr_{\F_{q^m}/\F_q}(\alpha x)=0\}.
\end{equation*}
A code is called an $N$-weight code if the set of its nonzero Hamming weights has cardinality $N$. We say that the code $C(q,m,u)$ is \textit{a subfield code} if $\cU_{n,q}\cdot\F_q^*\cup\{0\}$ is a subfield of $\F_{q^m}$, and we say that $C(q,m,u)$ is \textit{a semiprimitive code} if there is an integer $j$ such that $p^j\equiv -1\pmod{u_\Delta}$, where $u_\Delta=\gcd\left(u,\frac{q^m-1}{q-1}\right)$, cf. \cite[Definition 6]{Vega15}. By \cite{DingYang2013DM}, the code $C(q,m,u)$ is a one-weight code if and only if $u_{\Delta}=1$, i.e., $\F_q^*\cdot\cU_{n,q}=\F_{q^m}^*$.

The Schmidt-White conjecture asserts that a one- or two-weight irreducible cyclic code $C(p,m,u)$ with $(p-1)\mid n$ is either a subfield code, a semiprimitive code, or belongs to an exceptional set of eleven codes that appears in \cite[Table 1]{Schmidt2002}, cf. \cite[Conjecture 4.4]{Schmidt2002}. Vega gave a critical review of one- and two-weight irreducible cyclic codes in \cite{Vega15}, and most importantly, he modified the definition of semiprimitive codes so that the conjecture holds for the two-weight irreducible cyclic codes constructed in \cite{DingYang2013DM} and \cite{Rao2010}.  The one-weight irreducible cyclic codes are well understood, so we focus on the two-weight cases. The reformulated version of the Schmidt-White Conjecture asserts that a two-weight irreducible cyclic code is either a semiprimitive code or appears in \cite[Table 1]{Schmidt2002} up to certain equivalence as explained in \cite[Remark 5]{Vega15}, cf. \cite[Conjecture 2]{Vega15}. We shall refer to \cite[Conjecture 2]{Vega15} as the Schmidt-White conjecture. In this subsection, we verify it for NSIC codes.

\begin{lemma}\label{lem_TwoweightFqt}
Suppose that $\ord_n(q)=m$ and $n_0=\gcd(n,q-1)$. If $\gcd(m,t)=1$ and $\min\{m,t\}\ge 2$, then the code $C(q^t,m,u')$ with $u'=\frac{q^{mt}-1}{n}$ has exactly two nonzero weights if and only if either $m=2$, or $t=2$ and $n=n_0\frac{q^m-1}{q-1}$  such that $\gcd\left(\frac{q-1}{n_0},\frac{q^m-1}{q-1}\right)=1$. Furthermore, if $C(q^t,m,u')$ is a two-weight code,  then it is a semiprimitive code.
\end{lemma}
\begin{proof}
Let $u=\frac{q^m-1}{n}$ and $\Delta=\frac{q^m-1}{q-1}$. We have $\ord_n(q^t)=\ord_n(q)=m$ by the fact $\gcd(m,t)=1$, and $\cU_{n,q}=\cU_{n,q^t}$. For $\alpha\in\F_{q^{m}}$ and $x\in\cU_{n,q}$, we have $\tr_{\F_{q^{mt}}/\F_{q^t}}(\alpha x)=\tr_{\F_{q^m}/\F_q}(\alpha x)$. It follows that $\bfc_{\alpha}=(\tr_{\F_{q^{mt}}/\F_{q^t}}(\alpha\xi^i):0\le i\le n-1)$ is in both $C(q,m,u)$ and $C(q^t,m,u')$. Therefore,  $C(q,m,u)$ is a subcode of $C(q^t,m,u')$.

We first consider the case $m=2$. Let $\cU'=\F_{q^t}^* \cU_{n,q}$ which is a multiplicative subgroup of $\F_{q^{mt}}^*$. Take $\delta\in\F_{q^{2t}}^*$ such that $\delta+\delta^{q^t}=0$. For $\alpha\in\F_{q^{2t}}^*$ and $x\in\cU_{n,q}$, $\tr_{\F_{q^{2t}}/\F_{q^t}}(\alpha x)=0$ if and only if $\alpha x\in\F_{q^t}^*\delta$. The weight of $\bfc_\alpha$ is $n-|\alpha\cU_{n,q}\cap \F_{q^t}^*\delta|$, which takes the value $n-|\cU_{n,q}\cap \F_{q^t}^*|$ or $n$ according as $\alpha\delta^{-1}\in\cU'$ or not. 
Hence $C(q^t,m,u')$ is a two-weight code. The number $t$ is odd by $\gcd(m,t)=1$, and we have $n\mid (q^2-1)$. It follows that $\gcd(n,q^t-1)=\gcd(n,q-1)=n_0$, $u_\Delta'=\gcd(u',q^t+1)=(q^t+1)n_0/n$, and $q^t\equiv -1\pmod{u_\Delta'}$. That is, $C(q^t,m,u')$ is a semiprimitive code. This proves the claims in the lemma for $m=2$.

We assume that $m\ge 3$ from now on. Take an $\F_q$-basis $\alpha_1,\ldots,\alpha_t$ of $\F_{q^t}$, which is also an $\F_{q^m}$-basis of $\F_{q^{mt}}$. For a codeword $\bfc_\beta\in C(q^t,m,u')$ with $\beta\in\F_{q^{mt}}$, define
$T_\beta=\{i\in\Z_n:\tr_{\F_{q^{mt}}/\F_{q^t}}(\beta \xi^i)=0\}$. The weight of $\bfc_\beta$ is $n-|T_\beta|$. We have $T_0=\{0,1,\ldots,n-1\}$ and $T_{\beta\xi^{-i}}=i+T_{\beta}$, where $i+S=\{i+x:x\in S\}$ for a subset $S$ of $\Z_n$.  Similarly, for $\beta=\sum_{i=1}^t\lambda_i\alpha_i$ with $\lambda_i$'s in $\F_{q^m}$, we have $\bfc_{\beta}=\sum_{i=1}^t\alpha_i\bfc_{\lambda_i}$. The weight of $\bfc_\beta$ is $n-|\bigcap_{i=1}^tT_{\lambda_i}|$ by the fact $\alpha_1,\ldots,\alpha_t$ is an $\F_q$-basis of $\F_{q^t}$.

We first assume that $C(q^t,m,u')$ is a two-weight code. Let $Y=\{T_{\alpha}:\alpha\in\F_{q^m}^*\}$, and let $Y'$ be the set consisting of the intersection of any $k$-subset of $Y$ with $1\le k\le t$. The elements of $Y'$ have two distinct sizes, say, $s_1,s_2$ with $s_1<s_2$. Since $T_{\alpha\xi^{-i}}=i+T_{\alpha}$ for $i\in\Z_n$, for any $S\in Y$, its translates $i+S$ are also in $Y$. 
{If $\emptyset\in Y$, choose a nonzero $\lambda\in\F_{q^m}$ with $\tr_{\F_{q^m}/\F_q}(\lambda)=\tr_{\F_{q^m}/\F_q}(\lambda\xi)=0$. Then $T_\lambda$ and $1+T_\lambda$ have nonempty intersection, so both sets and their intersection have size $s_2$. It follows that $T_\lambda=1+T_\lambda$ which must be $\Z_n$: a contradiction. } If $Y\setminus\{\emptyset\}$ is not empty, take an element $S$ in $Y$ that has the smallest size. We can not have $S=\Z_n$, since it corresponds to the zero codeword $\bfc_0$ and $T_0$ is not in $Y$. There is at least one $i\in\Z_n$ such that $S\ne i+S$, and so $s_1=|S\cap (i+S)|$ and $s_2=|S|$. Therefore, all the elements of $Y$ have size $s_2$, i.e., $C(q,m,u)$ is a one-weight code. If $Y=\{\emptyset\}$,  then $C(q,m,u)$ is a one-weight irreducible cyclic code. We have $\gcd\left(u,\Delta\right)=1$ by \cite[Theorems 1 and 2]{Vega15}.   It follows that $u$ is a divisor of $q-1$ relatively prime to $\Delta$, so we have $n_0=\gcd\left(\frac{q-1}{u}\Delta,q-1\right)=\frac{q-1}{u}$. That is, $u=\frac{q-1}{n_0}$ and $n=n_0\frac{q^{m}-1}{q-1}$. We thus have $\cU_{n,q}\cdot\F_q^*=\F_{q^m}^*$ and $n_0=|\cU_{n,q}\cap \F_q^*|=\frac{q-1}{u}$. We have
\begin{align*}
|T_\lambda\cap T_{\lambda'}|&={\frac{n_0}{q-1}}\#\{x\in\F_{q^m}^*:\tr_{\F_{q^{m}}/\F_{q}}(\lambda x)=0,\tr_{\F_{q^{m}}/\F_{q}}(\lambda' x)=0\}
\end{align*}
for $\lambda,\lambda'\in\F_{q^{m}}^*$. It follows that $\bfc_\lambda$ has weight $w_1=n-\frac{q^{m-1}-1}{q-1}n_0=n_0q^{m-1}$ by letting $\lambda'=\lambda$. If $\lambda,\lambda'$ are linearly independent over $\F_q$, then we deduce that $\bfc_{\lambda\alpha_1}+\bfc_{\lambda'\alpha_2}$ has weight $w_2=n-\frac{q^{m-2}-1}{q-1}n_0=n_0q^{m-2}(q+1)$. If $t\ge 3$, then we get a codeword of weight $w_3=n-\frac{q^{m-3}-1}{q-1}n_0=n_0q^{m-3}(q^2+q+1)$ similarly: a contradiction.  This establishes the necessary part of the first claim for $m\ge 3$.

Conversely, assume that $t=2$ and $n=n_0\Delta$  with $\gcd\left(\frac{q-1}{n_0},\Delta\right)=1$. Then $u=\frac{q^m-1}{n}=\frac{q-1}{n_0}$ which is relatively prime to $\Delta$. The set $\{n-|\bigcap_{i=1}^tT_{\lambda_i}|:\lambda_1,\ldots,\lambda_t\in\F_{q^m}\}$ contains exactly two nonzero elements $w_1$ and $w_2$ by the preceding analysis, so $C(q^t,m,u')$ is a two-weight code. The number $m$ is odd by $\gcd(m,t)=1$, and we deduce that $u_\Delta'=\gcd\left(u',\frac{q^{mt}-1}{q^t-1}\right)=\frac{q^m+1}{q+1}$. It follows that $q^m\equiv -1\pmod{u_\Delta'}$, so $C(q^t,m,u')$ is semiprimitive as desired.
\end{proof}

\begin{lemma}\label{lem_TwoweightC2}

Let $(n,n_0,r)$ be a triple that satisfies the conditions in \eqref{eqn_nn0rCond}, and set $m=\ord_n(q)$, $m_0=\ord_{n_0}(q)$, $u=\frac{q^m-1}{n}$ and $u_0=\frac{q^{m_0}-1}{n_0}$.  Then $C(q,m, u)$ is a two-weight code if and only if $C(q,m_0,u_0)$ is a one-weight code and $r=2$. Furthermore, if $C(q,m,u)$ is a two-weight code, then $C(q,m, u)$ must be a semiprimitive code.
\end{lemma}

\begin{proof}
Let $E=\F_{q^m}$ and $F=\F_{q^{m_0}}$ in this proof. Take an element $\xi$ of order $n$ in  $E$. We have $E=\F_q[\xi]$, $F=\F_q[\xi^r]$, and $1,\xi,\ldots,\xi^{r-1}$ form an $F$-basis of $E$. Take a dual basis $\gamma_0,\ldots,\gamma_{r-1}$ such that $\tr_{E/F}(\xi^i\gamma_j)=\delta_{ij}$, where $\delta_{ij}=1$ if $i=j$ and $=0$ otherwise. We have $\cU_{n,q}=\la\xi \ra$ and $\cU_{n_0,q}=\la \xi^r\ra$. For an element {$\alpha=\sum_{i=0}^{r-1}\lambda_i\gamma_i$} with $\lambda_i$'s in $\F_{q^{m_0}}$, we have
\[
\Tr_{E/{\F_q}}(\alpha\xi^{i+rj})=\Tr_{F/\F_q } \big( \xi^{rj} \cdot \Tr_{E/F}(\alpha \xi^i) \big)=\Tr_{F/\F_q }(\lambda_i\xi^{rj}).
\]

It follows that the Hamming weight of the codeword $\bfc_{\alpha}$ in $C(q,m,u)$ is

\begin{equation}\label{eqn_tt123}
    \sum_{i=0}^{r-1}\left(n_0- \#\{ x\in \cU_{n_0,q}: \Tr_{F/\F_q}(\lambda_i x)=0\} \right).
\end{equation}
It is the sum of the Hamming weights of $r$ codewords in $C(q,m_0,u_0)$.

We first assume that $C(q,m, u)$ is a two-weight code. By setting all but one of $\lambda_i$ to $0$ in \eqref{eqn_tt123}, we deduce that $C(q,m_0,u_0)$ has at most two weights. By setting all but two of the $\lambda_i$ to $0$, we deduce that $C(q,m_0,u_0)$ is a one-weight code. If its nonzero weight is $n_0-w'$, then the two nonzero weights of $C(q,m,u)$ are $n_0-w'$, $2(n_0-w')$. If $r\ge 3$, then we deduce that $C(q,m,u)$ contains a codeword of a third weight: a contradiction. Hence, we have $r=2$.

We next assume that $C(q,m_0,u_0)$ is a one-weight code and $r=2$. Let $\Delta=\frac{q^{m}-1}{q-1}$, $\Delta_0=\frac{q^{m_0}-1}{q-1}$ and $u_{\Delta}=\gcd(u, \Delta)$. It holds that $u=u_0(q^{m_0}+1)/2$, $\Delta=\Delta_0(q^{m_0}+1)$ by the fact $m=2m_0$, so $u_\Delta=\frac{1}{2}\gcd(u_0,2\Delta_0)(q^{m_0}+1)$. Since $C(q,m_0,u_0)$ is a one-weight code, we have $\gcd(u_0,\Delta_0)=1$ by \cite[Theorem 1]{Vega15}. It follows that $u_\Delta=\frac{1}{2}\gcd(u_0,2)(q^{m_0}+1)$, so $q^{m_0}\equiv -1\pmod{u_\Delta}$. Therefore, the code $C(q, m, u)$ is semiprimitive. This completes the proof.
\end{proof}

\begin{lemma}\label{lem_TwoWeightTensor}
Suppose that $\ord_n(q)=m>1$, $\ord_s(q)=t>1$ and $\gcd(m,t)=1$. Let \[n'=\lcm(n,s),\quad
u'=\frac{q^{mt}-1}{n'},\quad  u_1=\frac{q^m-1}{n},\quad  u_2=\frac{q^t-1}{s}.\] Then $C(q,mt,u')$ is a two-weight code if and only if  one of the following holds:
\begin{enumerate}
    \item[(1)] $\min\{m,t\}=2$ and both $C(q,m,u_1)$, $C(q,t,u_2)$ are one-weight codes;
    \item[(2)] one of $C(q,m,u_1)$, $C(q,t,u_2)$ is a two-weight code of dimension $2$ and the other is a one-weight code,
\end{enumerate}
and $C(q, mt, u')$ is a semiprimitive code in both cases.
\end{lemma}

\begin{proof}
Let $d=\gcd(n,s)$, $n_0=\gcd(n,q-1)$, $s_0=\gcd(s,q-1)$. By the fact $\gcd(m,t)=1$, we deduce that $d$ divides $q-1$. Take an element $\xi_1$ of order $n$ in $\F_{q^m}$ and an element $\xi_2$ of order $s$ in $\F_{q^t}$. We have $\cU_{n',q}=\cU_{n,q}\cdot\cU_{s,q}$, and $\cU_{n,q}\cap\cU_{s,q}\subseteq\F_q^*$. We showed in the proof of Construction \ref{exa_product} that $\ord_{n'}(q)=mt$.\medskip

We first establish the necessity part. Assume that $C(q, mt,u')$ is a two-weight code. For $\alpha,\alpha'\in\F_{q^m}^*$, we define $T_\alpha=\{i\in\Z_n: \tr_{\F_{q^{m}}/\F_q}(\alpha \xi_1^i)=0\}$ and $n_{\alpha,\alpha'}=\#(T_\alpha\cap T_{\alpha'})$. It holds that {$T_{\alpha\xi_1^k}=-k+T_{\alpha}$} for $0\le k\le n-1$. The codeword $\bfc_\alpha$ in $C(q,m,u_1)$ has the same weight as $\bfc_{\alpha x}$ for $x\in\cU_{n,q}$. Take $\alpha,\alpha'\in\F_{q^m}^*$ and $\beta,\beta'\in\F_{q^t}^*$ such that $\F_q\alpha\ne\F_q\alpha'$, $\F_q\beta\ne\F_q\beta'$. By the condition $\gcd(m,t)=1$, we deduce that
\[
\tr_{\F_{q^{mt}}/\F_q}(\alpha\beta \xi_1^i\xi_2^j)
=\tr_{\F_{q^t}/\F_q}\left(\beta\xi_2^j\tr_{\F_{q^{mt}}/\F_{q^t}}(\alpha\xi_1^i)\right)
=\tr_{\F_{q^t}/\F_q}(\beta\xi_2^j)\cdot \tr_{\F_{q^{m}}/\F_q}(\alpha\xi_1^i).
\]
Write $\bfc_{\alpha,\beta}:=\bfc_{\alpha\beta}$. The weight of the codeword $\bfc_{\alpha,\beta}$ is
\begin{align}
 \notag &n'-\frac{1}{d} \#\{(x,y)\in \cU_{n,q}\times \cU_{s,q}:\, \tr_{\F_{q^{m}}/\F_q}(\alpha x) \cdot\tr_{\F_{q^{t}}/\F_q}(\beta y)=0 \}\\
=&\frac{1}{d}{\left(n-|T_\alpha|\right)}\cdot \left(s-\#\{y\in\cU_{s,q}:\tr_{\F_{q^{t}}/\F_q}(\beta y)=0\}\right), \label{eqn_calbe}
\end{align}
which is $\frac{1}{d}$ of the product of two nonzero weights in $C(q,m,u_1)$ and $C(q,t,u_2)$ respectively. We deduce that at most one of those two codes can be a two-weight code, since otherwise $C(q,mt,u')$ would have at least three distinct weights. We assume without loss of generality that  $C(q,t,u_2)$ is a one-weight code. By \cite[Theorems 1 and 2]{Vega15}, we have $\gcd\left(u_2,\frac{q^t-1}{q-1}\right)=1$ and  $s=s_0\frac{q^t-1}{q-1}$ for some integer $s_0$, and the nonzero weight of  $C(q,t,u_2)$ is $s-w'=s_0q^{t-1}$ with $w'=\frac{q^{t-1}-1}{q-1}s_0$.

For $x\in\cU_{n,q}$, let $\eta_x=\tr_{\F_{q^{m}}/\F_q}(\alpha x)$ and $\eta_x'=\tr_{\F_{q^{m}}/\F_q}(\alpha' x)$. The weight of $\bfc_{\alpha,\beta}+\bfc_{\alpha',\beta'}$ is
\begin{align}
&n'-\frac{1}{d}\#\{(x,y)\in\cU_{n,q}\times\cU_{s,q}:\eta_x\tr_{\F_{q^t}/\F_q}(\beta y)+\eta_x'\tr_{\F_{q^t}/\F_q}(\beta'y)=0\}\notag\\
=&n'-\frac{1}{d}\sum_{x\in\cU_{n,q}}\#\{y\in\cU_{s,q}:\tr_{\F_{q^t}/\F_q}((\eta_x\beta+\eta_x'\beta')y)=0\}\notag\\
=&n'-\frac{1}{d}(sn_{\alpha,\alpha'}+w'(n-n_{\alpha,\alpha'}))
=\frac{1}{d}s_0q^{t-1}(n-n_{\alpha,\alpha'}),\label{eqn_tpcsum1}
\end{align}
We consider two cases according as $C(q,m,u_1)$ is a one-weight or two-weight code. First assume that it is a one-weight code. By \cite[Theorems 1 and 2]{Vega15}, we have $\gcd\left(u_1,\frac{q^m-1}{q-1}\right)=1$ and $u_1= \frac{q-1}{n_0}$, and so $\F_q^*\cdot\cU_{n,q}=\F_{q^m}^*$, $|\F_q^*\cap\cU_{n,q}|=n_0$. Also, the nonzero weight of $C(q,m,u_1)$ is $n_0q^{m-1}$. If $\F_q\alpha\ne\F_q\alpha'$, then we deduce that $|T_\alpha\cap T_{\alpha'}|=\frac{q^{m-2}-1}{q-1}n_0$. The codewords of the form $\bfc_{\alpha,\beta}$ have weight $w_1=\frac{n_0s_0}{d}q^{m+t-2}$, and the codewords of the form $\bfc_{\alpha,\beta}+\bfc_{\alpha',\beta'}$ have weight  $w_2=\frac{n_0s_0}{d}{q^{m+t-3}}(q+1)$ by \eqref{eqn_tpcsum1}.  If {$\min\{m,t\}\ge 3$}, then take linearly independent elements $\alpha_1,\alpha_2,\alpha_3$ in $\F_{q^m}$ and  linearly independent elements $\beta_1,\beta_2,\beta_3$ in {$\F_{q^t}$}. The codeword $\sum_{i=1}^3 \bfc_{\alpha_i\beta_i}$ has weight distinct from $w_1,w_2$: a contradiction. Therefore, we must have {$\min\{m,t\}=2$}. We thus have case (1) when $C(q,m,u_1)$ is a one-weight code.

We next assume that $C(q,m,u_1)$ is a two-weight code. Let $w_1'$, $w_2'$ be its nonzero weights and assume that $w_2'>w_1'$.  For $i=1,2$, let $\alpha_i$ be a nonzero element in $\F_{q^m}$ such that the codeword $\bfc_{\alpha_i}$ in $C(q,m,u_1)$ has weight $w_i'$. That is, $n-|T_{\alpha_i}|=w_i'$ for $i=1,2$. For each $x\in\F_{q^m}^*$, the codeword $\bfc_x$ in $C(q,m,u_1)$ has weight $n-|T_x|\in\{w_1',w_2'\}$. A codeword of the form $\bfc_{\alpha,\beta}$ with $\alpha\in\F_{q^m}^*$ and $\beta\in\F_{q^t}^*$ has weight $\frac{1}{d}s_0q^{t-1}w_1'$ or $\frac{1}{d}s_0q^{t-1}w_2'$  by \eqref{eqn_calbe}, and both weights occur. Since $C(q,mt,u')$ has exactly two nonzero weights, we deduce from \eqref{eqn_tpcsum1} that $n-n_{\alpha,\alpha'}\in\{w_1',w_2'\}$ provided that $\F_q\cdot\alpha\ne\F_q \cdot\alpha'$. By taking $(\alpha,\alpha')=(\alpha_2\xi_1^i,\alpha_2\xi_1^{i+1})$ and using the facts $w_2'>w_1'$ and $n_{\alpha,\alpha'}\le|T_{\alpha}|$, we deduce that $n-n_{\alpha,\alpha'}=w_2'$, so {$T_{\alpha_2 \xi_1^{i}}\subseteq T_{\alpha_2\xi_1^{i+1}}$} for each $i$. It follows that  $T_{\alpha_2}=-1+T_{\alpha_2}$, i.e., $T_{\alpha_2}=\emptyset$ or $\Z_n$. Since $w_2'>0$, we have $T_{\alpha_2}=\emptyset$ and $w_2'=n$. For $1\le k\le n-1$, we similarly deduce that either $T_{\alpha_1}=-k+T_{\alpha_1}$ or $T_{\alpha_1}\cap (-k+T_{\alpha_1})=\emptyset$. If $m\ge 3$, then there is $x\in\F_{q^m}^*$ such that $\Tr_{\F_{q^m}/\F_q}(x)=0$, $\Tr_{\F_{q^m}/\F_q}(x\xi_1)=0$ by linear algebra. It follows that $|T_x|\ne 0$ and thus $n-|T_x|=w_1'$. By specifying $\alpha_1=x$ and using the fact $\{0,1\}\subseteq T_x$, we deduce that $T_{\alpha_1}=-1+T_{\alpha_1}$. We similarly have $T_{\alpha_1}\in\{\emptyset,\Z_n\}$, and both cases lead to a contradiction to the fact that $C(q,m,u_1)$ is a two-weight code. Hence we have $m=2$. Take a nonzero element $\delta\in\F_{q^2}$ such that $\delta+\delta^q=0$. We have  $\{\xi_1^i:i\in T_\alpha\}=\alpha^{-1}\delta\cdot\F_q^*\cap\cU_{n,q}$, whose size is {$\gcd(n,q-1)$ or $0$} according as $\alpha^{-1}\delta\in\F_q^* \cdot\cU_{n,q}$ or not. This is case (2) in this lemma.\medskip

We now establish the sufficiency part and show that $C(q, mt,u')$ is a semiprimitive code. In both cases, each nonzero codeword is either of the form $\bfc_{\alpha,\beta}$ or of the form $\bfc_{\alpha,\beta}+\bfc_{\alpha',\beta'}$, where $\F_q\cdot \alpha\ne\F_q\cdot\alpha'$ and $\F_q\cdot\beta\ne\F_q \cdot\beta'$. The weight of such a codeword is either \eqref{eqn_calbe} or \eqref{eqn_tpcsum1}, and it takes exactly two distinct values by the same arguments as those for the necessity part. We conclude that $C(q, mt,u')$ is a two-weight code, and this establishes the sufficiency part. We assume without loss of generality that $m=2$ in both cases. For case (1), we deduce from $\gcd(u_1,q+1)=1$ and {$\gcd(u_2,\frac{q^t-1}{q-1})=1$} that
\[
u'_{\Delta}=\gcd\left(u', \frac{q^{2t}-1}{q-1}\right)=\frac{q^t+1}{q+1}\gcd\left( \frac{q-1}{\lcm(n_0,s_0)},\frac{q^{t}-1}{q-1} \right)=\frac{q^t+1}{q+1}.
\]
It follows that {$q^{t}\equiv -1 \pmod{u_{\Delta}'}$}. So $C(q,mt,u')$ is semiprimitive. For (2), we similarly deduce that \[u'_\Delta=\gcd(u', \frac{q^{2t}-1}{q-1})=\frac{(q^t+1)n_0}{n} \gcd(\frac{q-1}{\lcm(s_0,n_0) },\, \frac{q^t-1}{q-1} \cdot\frac{n}{n_0})\] is a divisor of $q^t+1$ {by $\frac{q-1}{\lcm(s_0,n_0)}\mid u_2$ and $\gcd(\frac{q-1}{\lcm(s_0,n_0)}, \frac{q^t-1}{q-1}) = 1$; and it follows that the remaining factor $\gcd(\frac{q-1}{\lcm(s_0,n_0)}, \frac{n}{n_0})$ divides $\frac{n}{n_0}$,}
so $q^t\equiv -1\pmod{u'_\Delta}$ also holds. Therefore, $C(q, mt,u')$ is a semiprimitive code in both cases. This completes the proof.
\end{proof}

\begin{thm}
  The Schmidt-White conjecture holds for NSIC codes.
\end{thm}
\begin{proof}
Let $\mathbf{C}_{11}$ be the dual of the ternary Golay code of length $11$, and let $\mathbf{C}_{22}$ be the ternary code of length $22$ obtained from $\mathbf{C}_{11}$ by Construction \ref{exa_ext}. The code in Example \ref{exa_repetition} is a two-weight code only when $(n,p)=(5,2)$, and it has more than two weights otherwise. This code of length $n=5$ over $\F_2$ is covered by Example \ref{exa_C8}, so we include it in Example \ref{exa_C8} rather than Example \ref{exa_repetition} for uniform treatment in this proof. Also, the codes in Example \ref{exa_whole} have only one nonzero weight, the code in Example \ref{exa_C8} has two nonzero weights and is semiprimitive, the dual of the binary Golay code has more than two weights, and the dual of the ternary Golay code is a two-weight code. {Its extension $\mathbf{C}_{22}$ is the exceptional code with $(p,m,u)=(3,5,11)$ in \cite[Table 1]{Schmidt2002}.} By Lemma \ref{lem_TwoweightC2}, the code in Example \ref{exa_imp} is semiprimitive if it has exactly two nonzero weights.

Let $C$ be a two-weight non-standard non-degenerate irreducible cyclic code over $\F_q$ that is obtained from a code $C_0$ in Examples \ref{exa_repetition}-\ref{exa_imp} by repeated use of Constructions \ref{exa_ext}-\ref{exa_product}, and assume that it is not semiprimitive, and is not permutation equivalent to either (the dual of) the ternary Golay code $\mathbf{C}_{11}$ or the related code $\mathbf{C}_{22}$. By the first paragraph of this proof, $C$ is not one of  Examples \ref{exa_repetition}-\ref{exa_imp}.  The input and output of Construction \ref{exa_ext} have the same number of nonzero weights by the proof of \cite[Lemma 2.5]{Schmidt2002}, and it is easy to check by definition that one is semiprimitive if and only if the other is. We assume without loss of generality that the last round is not Construction \ref{exa_ext}. If $C$ is a resulting code from Construction \ref{exa_lift} or Construction \ref{exa_product}, then it is semiprimitive by Lemmas \ref{lem_TwoweightFqt} and \ref{lem_TwoWeightTensor} respectively: a contradiction.  To summarize, we have shown that the last round is  Construction \ref{exa_expansion}.

We claim that $C_0$ is not a code in Example \ref{exa_imp}. It is easy to show that if $C_0$ is a code in Example \ref{exa_imp}, then so is $C$. For instance, if the input code of length $n$ over $\F_q$ in Construction \ref{exa_product} satisfies \eqref{eqn_nn0rCond}, i.e., $n=n'k$ and $\ord_n(q)=k\cdot\ord_{n'}(q)$ for some $k\ge 2$, then its output code has length $N=\lcm(n,s)$. Let $N'=\lcm(n',s)$, and set $k'=\frac{N}{N'}$ which is a divisor of $k$. We have $\ord_N(q)=\ord_{n}(q)t$, $\ord_{N'}(q)=\ord_{n'}(q)t$ by the fact $\gcd(m,t)=1$. By Lemma \ref{lem_ordqn} (3) we have $\ord_N(q)\le k'\ord_{N'}(q)$, so $k'=k$ and the output code is also a code in Example \ref{exa_imp}. The claim follows from Lemma \ref{lem_TwoweightC2} and the fact $C$ is not semiprimitive.

We write $C=C(q,m,u)$ which has length $n=\frac{q^m-1}{u}$ and dimension $m$. Let $p$ be the characteristic of $\F_q$. Let $\overline{\F}$ be the algebraic closure of $\F_q$, and for a positive integer $d$ we define $\cU_d=\{x\in\overline{\F}:x^d=1\}$. Let $N=\lcm(n,q-1)$, so that $\cU_N=\cU_n\F_q^*$. Assume that $C_0$ is a cyclic code of length $n_0$ and dimension $m_0$ over the field $\F$. For each intermediate code $C_i$ obtained after the $i$-th round, we write $n_i$ for its length, $\mathbb{E}_i$ for its field and $m_i$ for its dimension over $\mathbb{E}_i$.  If the $i$-th round is Construction \ref{exa_product}, then we write $\mathbb{K}_i$ for the extension field of $\mathbb{E}_i$ of degree $t_{i-1}=\frac{m_i}{m_{i-1}}$. Let $\mathbb{K}$ be the smallest field that contains all the $\mathbb{E}_i$'s and $\mathbb{K}_i$'s (if any). In Table \ref{atab}, we list the changes of code parameters after the $(i+1)$-st round. From the table we see that $m_i\mid m_{i+1}$, and $[\mathbb{K}:\F]$ is relatively prime to $m_0$. By the last column of Table \ref{atab}, $\cU_n$ is obtained from $\cU_{n_0}$ by adjoining subgroups of $\mathbb{K}^*$, so $\cU_n\subseteq\cU_{n_0}\mathbb{K}^*$. Since $n_0\mid n$ and $\F_q\subseteq \mathbb{K}$, we have $\cU_N=\cU_{n_0}\cU_d$ with $d=|\cU_N\cap\mathbb{K}^*|$. 

\begin{table}[h]
$$\begin{array}{cllll} \hline
\mbox{Construction }  &\mbox{Dimension}  & \mbox{Field} & \mbox{Conditions}& \mbox{Remark} \\ \hline
\mbox{\ref{exa_ext}} & m_{i+1}=m_i & \mathbb{E}_{i+1}=\mathbb{E}_i & &  \cU_{n_{i+1}}\subseteq \cU_{n_i}\mathbb{E}_i^*\\
\mbox{\ref{exa_lift}} & m_{i+1}=m_i  & [\mathbb{E}_{i+1}:\mathbb{E}_i]=t_i&\gcd(m_i,t_i)=1
&\cU_{n_{i+1}}=\cU_{n_i}\\
\mbox{\ref{exa_expansion}}&m_{i+1}=m_ir_i &
[\mathbb{E}_{i}:\mathbb{E}_{i+1}]=r_i  &r_i\ge 2 &\cU_{n_{i+1}}=\cU_{n_i}\\
\mbox{\ref{exa_product}}& m_{i+1}=m_it_i  & \mathbb{E}_{i+1}=\mathbb{E}_i&\gcd(m_i,t_i)=1 & \cU_{n_{i+1}}\subseteq \cU_{n_i}\mathbb{K}_{i+1}^*\\ \hline
\end{array}$$
\caption{The change of code parameters after the $(i+1)$-st round.}\label{atab}
\end{table}

Suppose that $C$ is obtained from the code $C'=C(q^r,m/r,u')$ by Construction \ref{exa_expansion}, where $m=\ord_n(q)$ and $r$ is a proper divisor of $m$. We claim that $q^r-1$ divides $n$, so that $C'$ is also a two-weight code by \cite[Corollary 2.9]{Schmidt2002}. As a second corollary, we have $\gcd\left(\frac{q^m-1}{n},\frac{q^m-1}{q^i-1}\right)=\frac{q^m-1}{n}$ for $i\in\{1,r\}$, so $C'$ is also not semiprimitive. By induction on the number of rounds, we see that there is no such code $C'$ and this completes the proof.

It remains to prove the claim that $q^r-1$ divides $n$. Let $C_I$, $C_{II}$ and $C_{III}$ be the irreducible cyclic codes of length $N$ over $\F_q$, $\F_p$ and $\F$ respectively. The code $C_I$ is obtained from $C$ by applying Construction \ref{exa_ext}, so it is a two-weight code.  By \cite[Lemma 2.5, Corollary 2.9]{Schmidt2002}, we deduce that all three codes are two-weight codes. Since $C_0$ is a code in Examples \ref{exa_repetition}-\ref{exa_Golay}, we have $\F=\F_p$ or $\F^*\subseteq \cU_{n_0}$. We thus deduce that $\F_{q^m}$ is the smallest field that contains both $\F$ and $\cU_N$ and $|\F^*|$ divides $N$. It is easy to deduce from $\cU_N=\cU_{n_0}\cU_d$ that $\gcd\left(\frac{q^m-1}{N},\frac{q^m-1}{|\F|-1}\right)=
\gcd\left(\frac{q^m-1}{n},\frac{q^m-1}{q-1}\right)$, so $C_{III}$ is not semiprimitive by the fact $C$ is not. If $\cU_d$ is not contained in $\F$, then $C_{III}$ is obtained from $C_0$ by Construction \ref{exa_product}. By Lemma \ref{lem_TwoWeightTensor}, we deduce that $C_{III}$ is semiprimitive: a contradiction. We thus have $\F_{q^m}=\F[\cU_N]=\F[\cU_{n_0}]$, and so $[\F_{q^m}:\F]=m_0$. The intermediate field $\F_{q^r}$ is contained in $\F_{q^m}\cap \mathbb{K}=\F$ by $\gcd([\mathbb{K}:\F],m_0)=1$. Since $r\ge 2$, we can not have $\F=\F_p$. Hence $\F^*\subseteq\cU_{n_0}$. We have $\cU_{n_0}\subseteq\cU_n$ by $n_0\mid n$, so $\F_{q^r}^*\subseteq \F^*\subseteq\cU_{n_0}\subseteq\cU_n$. We thus have $(q^r-1)\mid n$ as desired. 
\end{proof}

\subsection{The asymptotic behavior of the density of non-standard pairs}

As observed in \cite{Hollmann2023}, there is a bijection between non-standard pairs $(n,q)$ and {permutation classes of} NSIC codes of length $n$ over $\F_q$. In the survey \cite{OpenPcode}, Charpin noted that the results of \cite{Berger1996AutECcode} suggest the conjecture that almost all cyclic codes are standard. Let $C$ be a non-standard irreducible non-degenerate cyclic code over $\F_q$. Then $C$ is uniquely determined by the pair $(n,q)$, {up to permutation equivalence}, and in this case $\ord_n(q)=\dim(C)$. 
For a prime power $p^i$, define the proportion of non-standard pairs $(n,p^i)$ with $n\leq N$ by
\[
   R_{p^i}(N) \;=\;
   \frac{\#\{\,n : 1\leq n\leq N,\; (n,p^i)\ \text{is non-standard},\ \gcd(n,p)=1 \,\}}
        {\#\{\,n : 1\leq n\leq N,\; \gcd(n,p)=1 \,\}}.
\]
The Berger--Charpin conjecture for non-degenerate irreducible cyclic codes can then be reformulated as
\[
   \lim_{N\to\infty} R_{p^i}(N)\;=0,\, \forall\ i\geq 1.
\]
However, computing explicit values of $R_{p^i}(N)$ for large $N$ appears to be difficult. For the case $p=2$, we have used Theorem \ref{thm_mainCode} to obtain numerical data of $R_{2^i}(N)$ for various $N$ and~$i$. The results are listed in Table~\ref{tab:ratios}.

\begin{table}[!htbp]
    \centering
   { \begin{tabular}{c|c|c|c|c|c}\midrule
        $N$  & $10^3$ & $10^4$ & $5\cdot10^4$ & $10^5$ & $5\cdot 10^5$  \\ \hline
         $R_2(N)$ & $0.468$ & $ 0.4272$ & $0.40256$ & $  0.3952$ &  $0.377872$  \\ \hline 
         $R_4(N)$ & $0.300$ & $0.2866$ & $  0.27560$ & $  0.27258$ & $0.264384$ \\ \hline
         $R_8(N)$ & $ 0.344$ & $ 0.3070$& $0.29260$ & $ 0.28838$ &  $ 0.277024$\\ \hline
         $R_{16}(N)$ & $0.280$ &  $0.2644$ & $0.25460$ & $ 0.25168$ & $0.244588$ \\ \hline
         $R_{32}(N)$ & $0.410$ & $ 0.3692$ & $0.34712$ & $ 0.34076$ & $ 0.325976$ \\  \hline
         $R_{64}(N)$ & $0.218$ & $ 0.2026$ & $0.19676$ & $0.19546$ & $0.190592$ \\ \midrule
    \end{tabular}}
    \caption{Values of $R_{2^i}(N)$ computed by using \textsc{Magma} \cite{Magma}}
    \label{tab:ratios}
\end{table}

The numerical data in Table~\ref{tab:ratios} indicates that $R_{2^i}(N)$ decreases as $N$ grows, for each fixed $i$. This suggests that the density of non-standard pairs tends to zero, which aligns with the Berger-Charpin conjecture. Although the observed decrease is slow and explicit estimates for large $N$ remain difficult, the data provide experimental evidence consistent with the conjectural picture.

\section{Concluding remarks}
The permutation automorphism group of a {non-degenerate} irreducible cyclic code always contains a well-defined subgroup of the affine group generated by the cyclic shift and by the Frobenius permutations that permute the defining zeros of the code. We refer to these as standard permutation automorphisms; all other permutation automorphisms are called non-standard. A {non-degenerate} irreducible cyclic code is correspondingly called standard if all of its permutation automorphisms are standard, and non-standard otherwise.

In this paper, we classify the non-standard non-degenerate irreducible cyclic codes over finite fields. Our main theorem (Theorem \ref{thm_mainCode}) gives a complete classification: apart from a small number of explicit exceptional families and their descendants under certain natural constructions, every non-degenerate irreducible cyclic code is standard.

As applications, we obtained a general description of non-standard linear recurring sequence subgroups, extending earlier work of Brison and Nogueira \cites{Brison2003,Brison2008,Brison2009,Brison2010,Brison2014,Brison2021}, and we confirmed the Schmidt-White conjecture \cite{Schmidt2002} for all NSIC codes. These results show that the exceptional codes with extra symmetries are very limited in nature.

There remain a number of natural problems for further study. One is to extend the classification of standard and non-standard codes to more general cyclic codes, including degenerate and reducible cases, or codes with $\gcd(n,q)\ne 1$. We expect that, {for non-degenerate codes} and except for a few rare examples, the automorphism group should always lie inside the affine group. Another is to determine the permutation automorphism groups of cyclic codes in full generality. Even in the irreducible case with $\gcd(n,q)=1$, while we have identified the non-standard situations, the precise structure of $\PAut(C)$ in general remains to be described. We also note that the determination of automorphism groups is much more complicated if $\gcd(n,q)\neq 1$, and it may be out of reach to obtain a complete classification in that generality.

Finally, it would be interesting to investigate to what extent the methods developed here can be applied to other classes of codes, such as quasi-cyclic codes, and to understand more systematically the relation between group actions and code parameters. More broadly, the connection between irreducible cyclic codes, linear recurring sequences, and permutation group theory suggests further applications, for example, to the study of weight distributions or to other open conjectures related to two-weight codes.

\section*{Acknowledgments}
{This research was supported by the National Key Research and Development Program of China (Grant No.~2025YFA1017700).} Tao Feng was supported by the National Natural Science Foundation of China (Grant No.~12225110). Henk D.L.~Hollmann was supported by the Estonian Research Council grant PRG2531. Weicong Li was supported by the National Natural Science Foundation of China (Grant No.~12301422). Qing Xiang was supported by the National Natural Science Foundation of China (Grant No.~12131011). The authors would like to thank the anonymous referees and Professor Cai Heng Li for their helpful comments and valuable suggestions. \medskip

\noindent \textbf{Declaration on the use of generative AI tools.} None of the results or proofs in this paper were obtained through the use of AI tools. In addition, the manuscript has been written without any AI assistance, except for proofreading the revised version with the aid of GPT-6 Astra (OpenAI).


\end{document}